\RequirePackage{fix-cm}
\documentclass[smallcondensed,numbook]{svjour3}     
\smartqed  
\usepackage{graphicx}
\usepackage{url}
\usepackage{xcolor}
\definecolor{newcolor}{rgb}{.8,.349,.1}

\usepackage[hidelinks, colorlinks, allcolors=blue, bookmarksdepth=2]{hyperref}

\usepackage[numbered]{bookmark}

\usepackage{framed,multirow}

\usepackage{amsmath}
\usepackage{amssymb}
\usepackage{amsfonts}

\usepackage{mathrsfs}
\usepackage{upgreek}

\usepackage{epstopdf}
\usepackage{algorithmic}
\ifpdf
\DeclareGraphicsExtensions{.eps,.pdf,.png,.jpg}
\else
\DeclareGraphicsExtensions{.eps}
\fi
\usepackage[export]{adjustbox}

\usepackage{xspace}
\usepackage{bm}
\usepackage{physics}
\usepackage{cleveref}
\usepackage{enumerate}
\usepackage[caption=false]{subfig}
\usepackage{mathtools}
\usepackage{anyfontsize}

\usepackage{booktabs}
\usepackage{multirow, multicol, makecell}
\usepackage{longtable}
\newtheorem{assumption}{Assumption}

\crefname{assumption}{Assumption}{Assumptions}
\crefname{definition}{Definition}{Definitions}
\crefname{lemma}{Lemma}{Lemmas}
\crefname{remark}{Remark}{Remarks}
\crefname{theorem}{Theorem}{Theorems}
\crefname{proposition}{Proposition}{Propositions}
\crefname{section}{Section}{Sections}
\crefname{figure}{Fig.}{Figs.}
\crefname{equation}{}{}
\crefname{table}{Table}{Tables}

\newcommand{\fnc}[1]{\ensuremath{\mathcal{#1}}}

\newcommand{\uhk}[0]{\ensuremath{\bm{{u}}_{h,k}}}

\newcommand{\uk}[0]{\ensuremath{\bm{{u}}_{k}}}

\newcommand{\xk}[0]{\ensuremath{\bm{{x}}_{k}}}

\renewcommand{\H}[0]{\mathsf{H}}
\newcommand{\Hk}[0]{\mathsf{H}_{k}}
\newcommand{\Dk}[0]{\mathsf{D}_{k}}
\newcommand{\Dbk}[0]{\mathsf{D}_{b,k}}
\newcommand{\DDk}[0]{\mathsf{D}_{k}^{(2)}}
\newcommand{\Sk}[0]{\mathsf{S}_{k}}
\newcommand{\Qk}[0]{\mathsf{Q}_{k}}
\newcommand{\Ek}[0]{\mathsf{E}_{k}}
\newcommand{\D}[0]{\mathsf{D}}
\newcommand{\Q}[0]{\mathsf{Q}}
\newcommand{\E}[0]{\mathsf{E}}
\newcommand{\M}[0]{\mathsf{M}}
\newcommand{\T}[0]{\mathsf{T}}

\newcommand{\C}[0]{\mathsf{C}}
\newcommand{\A}[0]{\mathsf{A}}
\newcommand{\I}[0]{\mathsf{I}}

\newcommand{\X}[0]{\mathsf{X}}
\newcommand{\Y}[0]{\mathsf{Y}}

\newcommand{\R}[0]{\mathsf{R}}

\newcommand{\Rgk}[0]{\mathsf{R}_{\gamma k}}
\newcommand{\Rgv}[0]{\mathsf{R}_{\gamma v}}

\newcommand{\Dgk}[0]{\mathsf{D}_{\gamma k}}

\newcommand{\V}[0]{\mathsf{V}}

\newcommand{\eg}[0]{{e.g.\@}\xspace}
\newcommand{\ie}[0]{{i.e.\@}\xspace}

\newcommand{\ignore}[1]{} 

\newcommand{\polyref}[1]{\ensuremath{\mathbb{P}^{#1}}(\hat{\Omega})}

\newcommand{\polyk}[1]{\ensuremath{\mathbb{P}^{#1}}({\Omega}_k)}
\newcommand{\poly}[1]{\ensuremath{\mathbb{P}^{#1}}({\Omega})}
\newcommand{\cont}[1]{\ensuremath{\mathcal{C}^{#1}}({\Omega}_k)}

\newcommand{\IR}[1]{\mathbb{R}^{#1}}%
\newcommand{\IRtwo}[2]{\mathbb{R}^{{#1}\times{#2}}}%

\usepackage{etoolbox}
\newcommand*{\affaddr}[1]{#1} 
\newcommand*{\affmark}[1][*]{\textsuperscript{#1}}
\usepackage[a4paper,left=20mm,right=20mm,top=25mm,bottom=20mm]{geometry}

\begin{document}

\title{Stability and Functional Superconvergence of Narrow-Stencil Second-Derivative Generalized Summation-By-Parts Discretizations} 


\titlerunning{Stability and Functional Superconvergence of Narrow-Stencil SBP Discretizations}        

\author{Zelalem Arega Worku\protect\affmark[*] \and David W. Zingg\protect\affmark[*]}

\authorrunning{Z. Worku, D.W. Zingg}

\institute{Zelalem Arega Worku \at
	\email{zelalem.worku@mail.utoronto.ca}  
	\and
	David W. Zingg \at
	\email{dwz@oddjob.utias.utoronto.ca}
	\and
	\affaddr{\affmark[*]Institute for Aerospace Studies, University of Toronto, Toronto, Ontario, M3H 5T6, Canada}	
}

\date{}
\def\makeheadbox{\relax}
\maketitle
\vskip -2cm
\addcontentsline{toc}{section}{Abstract}
\begin{abstract}
We analyze the stability and functional superconvergence of discretizations of diffusion problems with the narrow-stencil second-derivative generalized summation-by-parts (SBP) operators coupled with simultaneous approximation terms (SATs). Provided that the primal and adjoint solutions are sufficiently smooth and the SBP-SAT discretization is primal and adjoint consistent, we show that linear functionals associated with the steady diffusion problem superconverge at a rate of $ 2p $ when a degree $ p+1 $ narrow-stencil or a degree $ p $ wide-stencil generalized SBP operator is used for the spatial discretization. Sufficient conditions for stability of adjoint consistent discretizations with the narrow-stencil generalized SBP operators are presented. The stability analysis assumes nullspace consistency of the second-derivative operator and the invertibility of the matrix approximating the first derivative at the element boundaries. The theoretical results are verified by numerical experiments with the one-dimensional Poisson problem.


\end{abstract}

\section{Introduction} \label{sec:introduction}
Compared to wide-stencil\footnote{Second-derivative operators formed by applying first-derivative operators twice.} summation-by-parts (SBP) operators, explicitly formed narrow-stencil\footnote{Also known as compact-stencil second-derivative operators.} second-derivative SBP operators provide smaller solution error, superior solution convergence rates, compact stencil width, and better damping of high frequency modes \cite{mattsson2004summation,mattsson2008stable,mattsson2012summation,del2015SecondDerivative,eriksson2018dual}. As with the wide-stencil operators, narrow-stencil second-derivative operators are coupled by simultaneous approximation terms (SATs) \cite{carpenter1994time}. However, the SAT coefficients derived for wide-stencil SBP operators must be modified for implementations with narrow-stencil SBP operators to achieve stability and adjoint consistency simultaneously. Unfortunately, the analysis required to find such SAT coefficients for narrow-stencil SBP operators is more involved, \eg, see \cite{eriksson2018dual}.  

Hicken and Zingg \cite{hicken2011superconvergent} showed that adjoint consistent SBP-SAT discretizations of linear elliptic partial differential equations (PDEs) lead to functional superconvergence (see also \cite{berg2012superconvergent,hicken2014dual,hicken2012output}). In their study, they analyzed discretizations with wide-stencil second-derivative classical SBP (CSBP) operators by posing second-order linear PDEs as a system of first-order equations and determined the conditions that the SATs must satisfy for adjoint consistency and functional superconvergence. A similar analysis is conducted in \cite{worku2020simultaneous} for multidimensional SBP operators, but without posing the second-order linear PDEs as a system of first-order equations. The latter approach enables analysis of functional accuracy of adjoint consistent discretizations of diffusion problems with narrow-stencil SBP operators. While the stability of discretizations arising from narrow-stencil second-derivative operators is well-studied (\eg, see \cite{carpenter1999stable,mattsson2008stable,mattsson2012summation,gong2011interface,mattsson2008discontinuous,mattsson2013solution}), it is only recently (see, \eg, \cite{eriksson2018dual,eriksson2018finite}) that conditions for which such discretizations satisfy both stability and adjoint consistency requirements are presented. Eriksson \cite{eriksson2018dual}, used the eigendecomposition technique to find the conditions on the SAT coefficients that enable construction of stable and adjoint consistent discretizations of diffusion problems. In a subsequent paper \cite{eriksson2018finite}, Eriksson and Nordstr{\"o}m used a variant of the approach in \cite{eriksson2018dual} to find a more general set of SAT coefficients. Although these SAT coefficients lead to adjoint consistent and stable discretizations in practice, the analysis in \cite{eriksson2018finite} assumes a condition that is not satisfied by many narrow-stencil second-derivative operators in the literature, including those in \cite{mattsson2004summation,mattsson2008discontinuous,mattsson2012summation,del2015SecondDerivative,mattsson2013solution}. Furthermore, it is not straightforward how the theory extends to narrow-stencil generalized SBP operators which have one or more of the following characteristics: exclusion of one or both boundary nodes, non-repeating interior point operators, and non-uniform nodal distribution \cite{del2015SecondDerivative}.

The first objective of this paper is to establish the conditions required for the stability of adjoint consistent SBP-SAT discretizations of diffusion problems with the generalized narrow-stencil second-derivative SBP operators. We use the ``borrowing trick"\cite{carpenter1999stable} in the energy stability analysis which directly applies to the diagonal- and block-norm\footnote{Also referred to as full-norm matrix.} narrow-stencil second-derivative SBP operators in \cite{mattsson2004summation,mattsson2008discontinuous,mattsson2012summation,mattsson2013solution} and to the generalized SBP operators of Del Rey Fern{\'a}ndez and Zingg \cite{del2015SecondDerivative} upon minor modifications of the derivative operators at element boundaries. The second objective is to show that primal and adjoint consistent discretizations lead to functional convergence rates of $ 2p $ when a degree $ p + 1$ narrow-stencil or a degree $ p $ wide-stencil diagonal-norm second-derivative generalized SBP operator is used to discretize steady diffusion problems for which the primal and adjoint solutions are sufficiently smooth. We also show that the functional converges at a rate of $ 2p $ irrespective of whether or not the scheme is adjoint consistent when a degree $ 2p - 1 $ dense-norm wide- or narrow-stencil second-derivative SBP operator is used to discretize the spatial derivatives. Finally, we specialize the generalized form of the SATs given in \cite{yan2018interior,worku2020simultaneous} for one-dimensional implementation and provide penalty coefficients corresponding to a few known types of SAT such that they lead to consistent, adjoint consistent, conservative, and stable discretizations when coupled with the narrow-stencil second-derivative generalized SBP operators. 

The paper is organized as follows. \cref{sec:notation} presents the notation and some important definitions. In \cref{sec:model problem}, we state the model problem and its SBP-SAT discretization. The main theoretical results that establish the functional superconvergence and energy stability of the SBP-SAT discretizations are presented in \cref{sec:theoretical results}. The theoretical results are verified using the steady version of the model problem, the Poisson equation, in \cref{sec:numerical results} and concluding remarks are presented in \cref{sec:conclusion}.

\section{Preliminaries} \label{sec:notation}
We closely follow the notation used in \cite{del2015SecondDerivative,fernandez2014generalized,yan2018interior,worku2020simultaneous}. A one-dimensional compact domain is considered, and it is tessellated into $ n_e $ non-overlapping elements, $ {\mathcal T}_h \coloneqq \{\{ \Omega_k\}_{k=1}^{n_e}: \Omega=\cup_{k=1}^{n_e} {\Omega}_k\}$. The boundaries of each element will be referred to as interfaces, and we denote their union by $ \Gamma_k \coloneqq \partial\Omega_k $. The set of all interior interfaces is denoted by $\Gamma ^I \coloneqq \{\Gamma_k \cap \Gamma_v : k,v=1,\dots,n_e, k\neq v \}$, while the element interfaces for which Dirichlet and Neumann boundary conditions are enforced are in the sets $ \Gamma^{D} $ and $ \Gamma^{N} $, respectively, and $ \Gamma:=\Gamma^I\cup\Gamma^D\cup\Gamma^N $. Operators associated with the left and right interfaces of $ \Omega_{k} $ bear the subscripts $ \ell $ and $ r $, respectively, and the left and right most elements are indicated by the subscripts $ L $ and $ R $, respectively, \eg, $ \D_{\ell L} $ is a derivative operator at the left interface of the left most element. The set of $n_p$ volume nodes in element $ \Omega_k $ is represented by $ \xk= \{x_i\}_{i=1}^{n_p} $. Uppercase script type, \eg, $\fnc{U}_k \in \cont{\infty}$, is used for continuous functions, and $\polyref{p} $ denotes the space of polynomials up to total degree $ p $, which has a cardinality of $ n_p^* = p+1 $. Bold letters, \eg, $ \uk \in \IR{n_p}$, delineate the restriction of $ \fnc{U}_k $ to grid points $ \xk $, while solution vectors to the discrete systems of equations have subscript $ h $, \eg, $ \uhk \in \IR{n_p}$. For the purpose of the functional convergence analysis in \cref{sec:functinal error}, we define $h\coloneqq \max_{a, b \in \bm{x}_k} |a - b|$ as the size of an element. Matrices are denoted by sans-serif uppercase letters, \eg, $\V \in \IRtwo{n_p}{n_p}$; $ \bm{1} $ denotes a vector consisting of all ones, $ \bm {0} $ denotes a vector or matrix consisting of all zeros. The sizes of $ \bm{1} $ and $ \bm {0} $ should be clear from context. 

Definitions of the first- and second-derivative SBP operators presented in \cite{del2015SecondDerivative} are stated below. For the construction of narrow-stencil second-derivative SBP operators, we refer the reader to \cite{fernandez2014generalized,del2015SecondDerivative,mattsson2012summation,mattsson2013solution}.

\begin{definition}[Generalized first-derivative SBP operator] \label{def:Dk}	The matrix $\Dk \in \IRtwo{n_p}{n_p}$ is a degree $ p $ SBP operator approximating the first derivative $ \pdv{x} $ on the set of nodes $ \xk $, which need neither be uniform nor include nodes on the boundaries and may have nodes outside the domain of element $ \Omega_k $, if \cite{del2015SecondDerivative}
	\begin{enumerate}
		\item $ \Dk \bm{p} = \pdv{\fnc{P}}{x}$ for all $\fnc{P} \in \polyk{p} $
		\item $ \Dk=\H_k^{-1} \Q_k $, where $ \H_k $ is a symmetric positive definite (SPD) matrix, and 
		\item $ \Qk = \Sk + \frac{1}{2} \Ek $, where $ \Sk = - \Sk^T $, $ \Ek = \Ek^T $, and $ \Ek $ satisfies 
		$ \bm{p}^T \E_k  \bm {q} = \sum_{\gamma\in \Gamma_k}[\fnc{P}]_{\gamma}[\fnc{Q}]_{\gamma}n_{\gamma k}$
		for all $ \fnc{P},\fnc{Q} \in \polyk{\tau} $, where $ \tau \ge p $, and $ n_{\gamma k} =1 $ if $ \gamma $ is the right interface of $ \Omega_{k} $, otherwise $ n_{\gamma k} =-1 $.
\end{enumerate} 
\end{definition}

The norm matrix, $ \Hk $, may be diagonal or dense. A dense-norm matrix refers to any norm matrix that is not diagonal, which includes the block-norm matrix. The block-norm matrix has diagonal entries at the interior points (containing $ h $) and dense blocks at the top-left and bottom-right corners corresponding to the boundary nodes. The $ L^2 $ inner product of two functions $ \fnc{P} $ and $ \fnc{Q} $ is approximated by \cite{hicken2013summation,fernandez2014generalized,hicken2016multidimensional,fernandez2018simultaneous}
\[ \bm{p}^T\Hk \bm{q} = \int_{{\Omega}_k}\fnc{P}\fnc{Q}\dd{\Omega} +\order{h^{2p}},\]
and $ \Hk $ defines the norm
\[ \bm{u}^T \Hk \bm{u} = \norm{\bm {u}}_{\H} = \int_{{\Omega}_k} \fnc{U}^2\dd{\Omega} + \order{h^{2p}}. \]
The $ \E_k $ matrix is constructed as \cite{fernandez2014generalized,fernandez2018simultaneous}
\begin{equation}\label{eq:Ek}
	\Ek = \sum_{\gamma\subset\Gamma_k} n_{\gamma k}\Rgk^T\Rgk = \R_{r k}^T\R_{r k} - \R_{\ell k}^T\R_{\ell k},
\end{equation}
where $ \Rgk $ is an extrapolation row vector of at least order $ h^{\tau+1}$ accuracy, \ie, $ \R_{\gamma k}\uk = [\fnc{U}_k]_\gamma + \fnc{O}(h^{\ge \tau+1})$. Furthermore, we define an operator that extrapolates the product of the diffusion coefficient and the derivative of the solution from volume nodes to an interface as 
\begin{equation}\label{eq:normal derivative}
\Dgk =n_{\gamma k} \Rgk \Lambda_k\Dbk.
\end{equation} 

\begin{definition} [Order-matched narrow-stencil second-derivative generalized SBP operator]\label{def:D2} The narrow-stencil second-derivative operator $\DDk$ of degree $p+1$, approximating $\frac{\partial}{\partial x}(\lambda_k\frac{\partial {\fnc{U}_k}}{\partial x})$, is order-matched with the first-derivative operator $\Dk=\H_k^{-1}\Q_k$ of degree $p$ on the nodal set $ \xk $ if \cite{del2015SecondDerivative}
\begin{equation} \label{eq:D2 accuracy}
\DDk(\lambda_k)\bm{p}_k = \pdv{x}(\lambda_k\pdv{\fnc{P}_k}{x}),\quad \forall \; (\lambda_k \fnc{P}_k) \in \polyk{p+1},
\end{equation}
and $ \DDk $ is of the form 
\begin{equation}\label{eq:D2 decomposition 1}
\DDk= \Hk^{-1}[-\M_k + \Ek \Lambda_k \Dbk], 
\end{equation}
where $ \M_k = \sum_{i=1}^{n_p}\Lambda_k(i,i)\bar{\M}_{i}  $, $\bar{\M}_i$ are symmetric positive semidefinite matrices, \[\Lambda_k = {\rm {diag}}(\lambda_k(x_1),\lambda_k(x_2), ...,\lambda_k(x_{n_p})),\] and $\Dbk$ is an approximation to the first derivative of degree and order $\ge p+1$.
\end{definition} 

The order-matched SBP operators in \cref{def:D2} are assumed to have a diagonal-norm matrix. Note that for the $ m^{th} $ derivative, the degree and order are related by $order = degree -m+1$; consequently, both the diagonal-norm narrow-stencil $\DDk$ and $\Dk$ operators are order $p$ accurate, while a diagonal-norm wide-stencil second-derivative operator, which has the decomposition 
\begin{equation}
	\Dk \Lambda_k\Dk = \Hk^{-1}[-\Dk^T\Hk\Lambda_k\Dk + \Ek \Lambda_k \Dk],
\end{equation} 
is order $ p-1 $ accurate \cite{del2015SecondDerivative}. Similar to the diagonal-norm SBP operators, block-norm SBP operators have an order $2p$ centered-difference interior operator. At the boundaries, however, the block-norm wide- and narrow-stencil second-derivative operators are closed with order $ 2p-2 $ one-sided stencils, unlike the order $ p-1 $ and $ p $ one-sided stencils used with the diagonal-norm wide- and narrow-stencil SBP operators, respectively. Furthermore, the $ \Dbk $ matrix of a block-norm operator contains order $ 2p - 1 $ approximations of the first derivative at rows corresponding to the boundary nodes (see, \eg, \cite{mattsson2013solution,mattsson2004summation} for definition and discussion regarding the block-norm SBP operators).

\begin{remark}
	In this work, we do not assume that $ \M_k $ is necessarily symmetric positive semidefinite; rather we assume that $ \M_k + \M_k^T$ is symmetric positive semidefinite, which allows the analysis to be extended to a more general class of explicitly formed second-derivative operators, including the block-norm SBP operators in \cite{mattsson2013solution,mattsson2004summation}, which do not have symmetric $ \M_k $ matrix.
\end{remark}

Another decomposition of second-derivative SBP operators, which is instrumental for the adjoint consistency and functional superconvergence analyses in \cref{sec:theoretical results}, is presented below.
\begin{proposition}\label{prop:D2 decomposition}
	A second-derivative operator of the form \cref{eq:D2 decomposition 1}, for which $ \M_k $ is not necessarily symmetric, can be decomposed as\\
	\begin{equation}\label{eq:D2 decomposition 3}
	\begin{aligned}
	\D_{k}^{(2)}&=\H_{k}^{-1}\left(\D_{k}^{(2)}\right)^{T}\H_k-\H_{k}^{-1}\D_{r k}^{T}\R_{r k}-\H_{k}^{-1}\D_{\ell k}^{T}\R_{\ell k} 
	\\
	&\quad +\H_{k}^{-1}\R_{r k}^{T}\D_{r k}+\H_{k}^{-1}\R_{\ell k}^{T}\D_{\ell k} -\H_{k}^{-1}\left(\M_{k}-\M_{k}^{T}\right).
	\end{aligned}
	\end{equation}	
\end{proposition} 
\begin{proof}
	Substituting \cref{eq:Ek} and \cref{eq:normal derivative} into \cref{eq:D2 decomposition 1}, we have 
	\begin{equation*}
	\begin{aligned}
	\D_{k}^{(2)}=\H_{k}^{-1}\left[-\M_k+\E_{k}\Lambda_{k}\D_{b,k}\right]&=-\H_{k}^{-1}\M_k+\H_{k}^{-1}\left(\R_{r k}^{T}\R_{r k}-\R_{\ell k}^{T}\R_{\ell k}\right)\Lambda_{k}\D_{b,k}
	\\
	&=-\H_{k}^{-1}\M_k+\H_{k}^{-1}\R_{r k}^{T}\D_{r k}+\H_{k}^{-1}\R_{\ell k}^{T}\D_{\ell k}
	\end{aligned}
	\end{equation*}
	Adding and subtracting $\H_{k}^{-1}\left(\D_{k}^{(2)}\right)^{T}\H_{k}$, we get
	\begin{equation*}
	\begin{aligned}
	\D_{k}^{(2)}&=-\H_{k}^{-1}\M_{k}+\H_{k}^{-1}\left(\D_{k}^{(2)}\right)^{T}\H_{k}-\H_{k}^{-1}\left(\D_{k}^{(2)}\right)^{T}\H_{k}
	+\H_{k}^{-1}\R_{r k}^{T}\D_{r k}+\H_{k}^{-1}\R_{\ell k}^{T}\D_{\ell k}
	\\
	& =\H_{k}^{-1}\left(\D_{k}^{(2)}\right)^{T}\H_{k}-\H_{k}^{-1}\left[-\M_{k}^{T}+\D_{b,k}^{T}\Lambda_{k}^{T}\E_{k}^{T}\right]
	+\H_{k}^{-1}\R_{r k}^{T}\D_{r k} +\H_{k}^{-1}\R_{\ell k}^{T}\D_{\ell k}-\H_{k}^{-1}\M_{k}
	\\
	&=\H_{k}^{-1}\left(\D_{k}^{(2)}\right)^{T}\H_k-\H_{k}^{-1}\D_{r k}^{T}\R_{r k}-\H_{k}^{-1}\D_{\ell k}^{T}\R_{\ell k}
	\\&\quad
	+\H_{k}^{-1}\R_{r k}^{T}\D_{r k}+\H_{k}^{-1}\R_{\ell k}^{T}\D_{\ell k} -\H_{k}^{-1}\left(\M_{k}-\M_{k}^{T}\right),
	\end{aligned}
	\end{equation*}
	which is the desired result.	
	\qed
\end{proof}

\section{Model Problem and SBP-SAT Discretization} \label{sec:model problem}
We consider the one-dimensional diffusion problem
\begin{align} \label{eq:diffusion problem}
\pdv{\fnc{U}}{t}-\pdv{x}\left(\lambda\pdv{\fnc{U}}{x}\right)& =\fnc{F}\quad\forall x\in\Omega, && \fnc{U}=\fnc{U}_{0}\;\text{at}\;t=0, 
&&
\fnc{U}|_{\Gamma^D} =\fnc{U}_{D},
&& n_{\gamma}\left(\lambda\pdv{\fnc{U}}{x}\right)\bigg|_{\Gamma^N}=\fnc{U}_{N},
\end{align}
where $ \fnc{F} \in L^2{(\Omega)}$, $ \lambda = \lambda(x)$ is a positive diffusivity coefficient, and $ \Gamma^D $ is not empty. For functional error analysis and numerical experiment purposes, we consider the steady version of \cref{eq:diffusion problem}, the Poisson problem. We also consider a compatible linear functional of the form 
\begin{equation} \label{eq:functional 1}
\fnc{I} (\fnc{U})= \int_{\Omega}\fnc{G}\fnc{U}\dd{\Omega} 
- \psi_{D}\left[\lambda\frac{\partial{\cal U}}{\partial x}n_{\gamma}\right]_{\Gamma^{D}}
+\psi_{N}{\cal U}|_{\Gamma^{N}},
\end{equation}
where $ \fnc{G}\in L^2{(\Omega)}$, $\psi_{N}=n_{\gamma} (\lambda\pdv{\psi}{x}) \in L^2{(\Gamma^N)}$, and $\psi_{D}\in L^2{(\Gamma^D)}$. A linear functional is compatible with the steady version of \cref{eq:diffusion problem} if \cite{hartmann2007adjoint}
\begin{equation}\label{eq:compatible functional}
\begin{aligned}
	&\int_{\Omega}\psi \pdv{x}\left(\lambda\pdv{\fnc{U}}{x}\right) \dd{\Omega} 
	+{\cal U}_{D}\left[\lambda\frac{\partial\psi}{\partial x}n_{\gamma}\right]_{\Gamma^{D}}
	-{\cal U}_{N}{\cal \psi}|_{\Gamma^{N}} \\ 
	&\qquad \quad= 
	\int_{\Omega}\fnc{U} \pdv{x}\left(\lambda\pdv{\psi}{x}\right) \dd{\Omega} 
	+ \psi_{D}\left[\lambda\frac{\partial{\cal U}}{\partial x}n_{\gamma}\right]_{\Gamma^{D}}
	-\psi_{N}{\cal U}|_{\Gamma^{N}}, 
\end{aligned}
\end{equation}
\ie,
\begin{equation}\label{eq:functional 2}
{\cal I}\left({\cal U}\right)=\fnc{I}\left(\psi\right)=\int_{\Omega}\psi{\cal F}{\rm d}\Omega -{\cal U}_{D}\left[\lambda\frac{\partial\psi}{\partial x}n_{\gamma}\right]_{\Gamma^{D}}+{\cal U}_{N}{\cal \psi}|_{\Gamma^{N}}.
\end{equation}
Under the compatibility condition on the functional, the adjoint, $ \psi $, satisfies the PDE (see, \eg, \cite{hicken2011superconvergent,yan2018interior,hartmann2007adjoint})
\begin{equation} \label{eq:adjoint problem}
\begin{aligned}
-\pdv{x}\left(\lambda\pdv{\psi}{x}\right)& =\fnc{G}\quad\forall x\in\Omega, && 
\psi|_{\Gamma^D} =\psi_{D}, && n_{\gamma}\left(\lambda\pdv{\psi}{x}\right)\bigg|_{\Gamma^N}=\psi_{N}.
\end{aligned}
\end{equation}

The SBP-SAT semi-discretization of the diffusion problem, \cref{eq:diffusion problem}, is given by
\begin{equation} \label{eq:SBP-SAT discretization}
\dv{\uhk}{t}=\D^{(2)}_{k}\uhk+\bm{f}_{k}-\H_{k}^{-1}\bm{s}_{k}^{I}(\uhk) -\H_{k}^{-1}\bm{s}_{k}^{B}(\uhk,{u}_{D}, {u}_{N})\eqqcolon R_{h,u},
\end{equation}
where $ \bm{f}_k $ is the restriction of $ \fnc{F} $ to the volume nodes in $ \Omega_{k} $ and the interface SATs, $ \bm{s}_{k}^{I} $,  and boundary SATs, $ \bm{s}_{k}^{B} $, given in \cite{yan2018interior,worku2020simultaneous} are specialized for one-dimensional implementation as
\begin{equation} \label{eq:interface SATs}
\bm{s}_{k}^{I}(\uhk)=\sum_{\gamma\subset\Gamma_{k}^{I}}\left[\begin{array}{cc}
\R_{\gamma k}^{T} & \D_{\gamma k}^{T}\end{array}\right]\left[\begin{array}{cc}
\T_{\gamma  k}^{(1)} & \T_{\gamma k}^{(3)}
\\
\T_{\gamma k}^{(2)} & \T_{\gamma k}^{(4)}
\end{array}\right]\left[\begin{array}{c}
\R_{\gamma k}\bm{u}_{h,k}-\R_{\gamma v}\bm{u}_{h,v}
\\
\D_{\gamma k}\bm{u}_{h,k}+\D_{\gamma v}\bm{u}_{h,v}
\end{array}\right]
\end{equation}
and 
\begin{align} \label{eq:boundary SATs}
\bm{s}_{k}^{B}(\uhk,\bm{u}_{ D}, \bm{u}_{ N}) &=\left\{\left[\begin{array}{cc}
\R_{\gamma k}^{T} & \D_{\gamma k}^{T}\end{array}\right]\left[\begin{array}{c}
\T_{\gamma k}^{(D)}\\
-1
\end{array}\right](\Rgk\uhk  -{u}_{D})\right\}_{\gamma\subset\Gamma^{D}}
 +\left\{\R_{\gamma k}^{T}\left(\D_{\gamma k}\bm{u}_{h,k}-{u}_{N}\right)\right\}_{\gamma\subset\Gamma^{N}}.
\end{align}
The SAT coefficients $\T_{\gamma k}^{(1)}, \T_{\gamma k}^{(2)}, \T_{\gamma k}^{(3)}, \T_{\gamma k}^{(4)}, \T_{\gamma k}^{(D)} \in \IR{}$ are determined such that the scheme satisfies desired properties such as conservation, adjoint consistency, and energy stability. For implementations with wide-stencil operators, we replace $ \DDk $  by $ \Dk\Lambda_k\Dk $ in \cref{eq:SBP-SAT discretization} and $ \Dbk $ by $ \Dk $ in \cref{eq:normal derivative}. 

Substituting the restriction of the exact solution to grid points, $ \uk $, in \cref{eq:SBP-SAT discretization,eq:interface SATs,eq:boundary SATs}, we see that the right-hand side (RHS) of \cref{eq:SBP-SAT discretization} yields a discretization error of $ \fnc{O}(h^{p}) $ when an order-matched narrow-stencil second-derivative SBP operator is used; hence, the discretization of the primal problem is consistent. In contrast, for diagonal-norm wide-stencil SBP operators, the discretization error is $ \fnc{O}(h^{p-1}) $ while for block-norm second-derivative SBP operators, it reduces to $ \fnc{O}(h^{2p-2}) $.  

\section{Theoretical Results}\label{sec:theoretical results}
In this section, we present the two main results of this paper. After establishing the conditions required for adjoint consistency and conservation, we show that primal and adjoint consistent SBP-SAT discretizations of the Poisson problem with the diagonal-norm narrow-stencil second-derivative operators lead to functional superconvergence. To achieve this goal, we closely follow the technique used to show functional superconvergence in \cite{worku2020simultaneous}. Then, we use the energy method to find sufficient conditions that the SATs must satisfy for the stability of discretizations with narrow-stencil generalized SBP operators before stating a few concrete examples of such SATs.

\subsection{Adjoint Consistency} \label{sec:adjoint consistency}
Adjoint consistency requires that the discrete adjoint problem, 
\begin{equation}
\sum_{\Omega_{k}\in{\cal T}_{h}}\left(L_{h,k}^*(\bm{\psi}_h)-\bm{g}_k\right)=\bm{0}, 
\end{equation}
where $ L_{h,k}^*$ is the discrete adjoint operator, corresponding to the steady version of the primal problem \cref{eq:SBP-SAT discretization} satisfy
\begin{equation} \label{eq:adjoint consistency}
\lim_{h \rightarrow 0}\sum_{\Omega_{k}\in{\cal T}_{h}}\norm{L_{h,k}^{*}\left(\psi_{k}\right)-\bm{g}_{k}}_{\H_{k}}=0.
\end{equation}

To find the discrete adjoint operator, we begin by discretizing the two forms of the functional, \cref{eq:functional 1,eq:functional 2}, as 
\begin{align}
I_{h}\left(\bm{u}_{h}\right)&=\sum_{\Omega_{k}\in{\cal T}_{h}}\bm{g}_{k}^{T}\H_{k}\bm{u}_{h,k}-\psi_{D}\D_{\ell L}\bm{u}_{h,L}+\psi_{N}\R_{r R}\bm{u}_{h,R}
+\psi_{D}\T_{\ell L}^{(D)}\left(\R_{\ell L}\bm{u}_{h,L}-u_D\right), \label{eq:functional discrete 1}
\\
I_{h}\left(\bm{\psi}_{h}\right)&=\sum_{\Omega_{k}\in{\cal T}_{h}}\bm{f}_{k}^{T}\H_{k}\bm{\psi}_{h,k}-u_{D}\D_{\ell L}\bm{\psi}_{h,L}+u_{N}\R_{r R}\bm{\psi}_{h,R}
+u_{D}\T_{\ell L}^{(D)}\left(\R_{\ell L}\bm{\psi}_{h,L}-\psi_D\right),	\label{eq:functional discrete 2}		
\end{align}
where we have assumed that the Dirichlet and Neumann boundary conditions are enforced on the left and right boundaries, respectively. The last terms in \cref{eq:functional discrete 1,eq:functional discrete 2} arise from consistent modifications of the functional, see \cite{hartmann2007adjoint,hicken2011superconvergent,hicken2012output,yan2018interior,worku2020simultaneous}. Note that in cases where a Dirichlet boundary condition is enforced on both boundaries, we apply the Dirichlet SATs given in \cref{eq:boundary SATs} on both boundaries and modify the discrete functionals, \cref{eq:functional discrete 1,eq:functional discrete 2}, by replacing the Neumann boundary terms by Dirichlet right boundary terms similar to those given for the left boundary. The theory developed holds for such cases without significant modification.

To derive the conditions required for adjoint consistency, we set $ I_h(\bm{u}_h) - I_h(\bm{\psi}_h) = 0 $, which is a discrete analogue of the relation $ \fnc{I}(\fnc{U}) - \fnc{I}(\fnc{\psi}) = 0$. Adding $ \sum_{\Omega_{k}\subset{\cal T}_{h}}\bm{\psi}_{h,k}^{T}\H_{k}R_{h,k}+I_{h}\left(\bm{\psi}_{h}\right)-I_{h}\left(\bm{\psi}_{h}\right)=0 $ to the RHS of \cref{eq:functional discrete 1} and rearranging we find 
\begin{equation}\label{eq:adjoint 1}
	\begin{aligned}
		I_{h}\left(\bm{u}_{h}\right)&=I_{h}\left(\bm{\psi}_{h}\right)+\sum_{\Omega_{k}\subset{\cal T}_{h}}\bm{g}_{k}^{T}\H_{k}\bm{u}_{h,k}-\psi_{D}\D_{\ell L}\bm{u}_{h,L}+\psi_{N}\R_{r R}\bm{u}_{h,R}
		-u_{D}\T_{\ell L}^{(D)}\left(\R_{\ell L}\bm{\psi}_{h,L}-\psi_{D}\right)
		\\&\quad
		+\psi_{D}\T_{\ell L}^{(D)}\left(\R_{\ell L}\bm{u}_{h,L}-u_{D}\right)+u_{D}\D_{\ell L}\bm{\psi}_{h,L}-u_{N}\R_{r R}\bm{\psi}_{h,R}
		\\&\quad
		+\sum_{\Omega_{k}\subset{\cal T}_{h}}\left[\bm{\psi}_{h,k}^{T}\H_{k}\D_{k}^{(2)}\bm{u}_{h,k}-\bm{\psi}_{h,k}^{T}\bm{s}_{k}^{I}(\bm{u}_{h,k})-\bm{\psi}_{h,k}^{T}\bm{s}_{k}^{B}(\bm{u}_{h,k},\bm{u}_{D},\bm{u}_{N})\right].
	\end{aligned}
\end{equation}
Transposing \cref{eq:adjoint 1}, enforcing $ I_h(\bm{u}_h) - I_h(\bm{\psi}_h) = 0 $, applying identity \cref{eq:D2 decomposition 3}, and simplifying, we obtain
\begin{equation}
	\begin{aligned}
		&\sum_{\Omega_{k}\subset{\cal T}_{h}}\bigg\{\bm{u}_{h,k}^{T}\H_{k}\left(\D_{k}^{(2)}\bm{\psi}_{h,k}+\bm{g}_{k}\right)+\bm{u}_{h,k}^{T}\left(\M_{k}-\M_{k}^{T}\right)\bm{\psi}_{h,k}\bigg\}
		-\bm{u}_{h,L}^{T}\R_{\ell L}^{T}\T_{\ell L}^{(D)}\left(\R_{\ell L}\bm{\psi}_{h,L}-\psi_{D}\right)
		\\&
		\quad-\sum_{\gamma\subset\Gamma^{I}}\left[\begin{array}{c}
		\R_{\gamma k}\bm{u}_{h,k}\\
		\R_{\gamma v}\bm{u}_{h,v}\\
		\D_{\gamma k}\bm{u}_{h,k}\\
		\D_{\gamma v}\bm{u}_{h,v}
		\end{array}\right]^{T}\left[\begin{array}{cccc}
		\T_{\gamma k}^{(1)} & -\T_{\gamma v}^{(1)} & \T_{\gamma k}^{(2)}+1 & -\T_{\gamma v}^{(2)}\\
		-\T_{\gamma k}^{(1)} & \T_{\gamma v}^{(1)} & -\T_{\gamma k}^{(2)} & \T_{\gamma v}^{(2)}+1\\
		\T_{\gamma k}^{(3)}-1 & \T_{\gamma v}^{(3)} & \T_{\gamma k}^{(4)} & \T_{\gamma v}^{(4)}\\
		\T_{\gamma k}^{(3)} & \T_{\gamma v}^{(3)}-1 & \T_{\gamma k}^{(4)} & \T_{\gamma v}^{(4)}
		\end{array}\right]\left[\begin{array}{c}
		\R_{\gamma k}\bm{\psi}_{h,k}\\
		\R_{\gamma v}\bm{\psi}_{h,v}\\
		\D_{\gamma k}\bm{\psi}_{h,k}\\
		\D_{\gamma v}\bm{\psi}_{h,v}
		\end{array}\right]
		\\&
		\quad+\bm{u}_{h,L}^{T}\D_{\ell L}^{T}\left(\R_{\ell L}\bm{\psi}_{h,L}-\psi_{D}\right)
		-\bm{u}_{h,R}^{T}\R_{r R}^{T}\left(\D_{r R}\bm{\psi}_{h,k}-\psi_{N}\right)=0,
	\end{aligned}
\end{equation} 
from which we extract the discrete adjoint operator on element $ \Omega_{k} $ as 
\begin{equation}\label{eq:discrete adjoint operator}
	\begin{aligned}
		L_{h,k}^{*} (\psi_h)&= -\DDk \bm{\psi}_{h,k} - \H_k^{-1}(\M_k-\M_k^T)\bm{\psi}_{h,k} 
		+ \H^{-1}_k(\bm{s}_k^I)^*(\bm{\psi}_{h,k}) +  \H^{-1}_k(\bm{s}_k^B)^*(\bm{\psi}_{h,k},{\psi}_{D}, {\psi}_{N}),
	\end{aligned}
\end{equation}
where the interface and boundary SATs for the adjoint problem are given, respectively, by
\begin{equation} \label{eq:interface SATs adjoint}
\left(\bm{s}_{k}^{I}\right)^{*}  =\sum_{\gamma\subset\Gamma_{k}^{I}}\begin{bmatrix}
\R_{\gamma k}^{T} & \D_{\gamma k}^{T}\end{bmatrix}
\begin{bmatrix}
\T_{\gamma k}^{(1)} & -\T_{\gamma v}^{(1)} & \T_{\gamma k}^{(2)}+1 & -\T_{\gamma v}^{(2)}\\
\T_{\gamma k}^{(3)}-1 & \T_{\gamma v}^{(3)} & \T_{\gamma k}^{(4)} & \T_{\gamma v}^{(4)}
\end{bmatrix}
\begin{bmatrix}
\R_{\gamma k}\bm{\psi}_{h,k}\\
\R_{\gamma v}\bm{\psi}_{h,v}\\
\D_{\gamma k}\bm{\psi}_{h,k}\\
\D_{\gamma v}\bm{\psi}_{h,v}
\end{bmatrix},
\end{equation}
\begin{equation}\label{eq:boundary SATs adjoint}
\begin{aligned}
\left(\bm{s}_{k}^{B}\right)^{*}&=\left\{\left[\begin{array}{cc}
\R_{\gamma k}^{T} & \D_{\gamma k}^{T}\end{array}\right]\left[\begin{array}{c}
\T_{\gamma k}^{(D)}\\
-1
\end{array}\right]\left[\begin{array}{cc}
\R_{\gamma k}\bm{\psi}_{h,k}-{\psi}_{D}\end{array}\right]\right\}_{\gamma\subset\Gamma^{D}} +\left\{\R_{\gamma k}^{T}\left(\D_{\gamma k}\bm{\psi}_{h,k}-{\psi}_{N}\right)\right\}_{\gamma\subset\Gamma^{N}}.
\end{aligned}
\end{equation}
Furthermore, we define the residual of the SBP-SAT discretization of the adjoint problem as
\begin{align} \label{eq:SBP-SAT adjoint}
		R_{h,\psi}&\coloneqq\D_k^{(2)}\bm{\psi}_{h,k}+ \bm{g}_k + \H_k^{-1}(\M_k-\M_k^T)\psi_{h,k} 
		-\H^{-1}_k(\bm{s}_k^I)^*(\bm{\psi}_{h,k}) - \H^{-1}_k(\bm{s}_k^B)^*(\bm{\psi}_{h,k},{\psi}_{D}, {\psi}_{N}) = \bm{0},
\end{align}

Substituting the exact adjoint solution into \cref{eq:SBP-SAT adjoint}, we observe that $ R_{h,\psi} $ is $ \fnc{O}(h^{\ge p - 1}) $, \ie, the discretization of the adjoint problem is consistent, if
\begin{align}\label{eq:adjoint consistency conditions}
\T_{\gamma k}^{(1)} = \T_{\gamma v}^{(1)}, && \T_{\gamma k}^{(2)}+1=-\T_{\gamma v}^{(2)}, && \T_{\gamma k}^{(3)} - 1 = -\T_{\gamma v}^{(3)}, && \T_{\gamma  k}^{(4)} =  \T_{\gamma  v}^{(4)}, && \M_k = \M_k^T.   
\end{align}  
For discretizations with wide-stencil second-derivative operators, the last condition, $ \M_k=\M_k^T $, is satisfied by default.
 
\subsection{Conservation}\label{sec:conservation}
For conservation, the homogeneous diffusion problem \cref{eq:diffusion problem}, \ie, $ \fnc{F}=0 $, should satisfy Gauss's theorem discretely, \ie, $ \sum_{\Omega_{k} \subset \fnc{T}_h} \bm{1}^T\H_k {\rm{d}}\bm{u}_k/{\rm{d}}t$ must depend only on the boundary terms. Premultiplying $ R_{h,u} $ defined in \cref{eq:SBP-SAT discretization} by $ \bm{1}^T\H_k $, setting $ \bm{f}_k = 0 $, summing over all elements, and applying the decomposition of $ \DDk $ given in \cref{eq:D2 decomposition 1} yields
\begin{equation}\label{eq:Conservation}
	 \begin{aligned}
		 \sum_{\Omega_{k}\subset{\cal T}_{h}}\bm{1}^{T}\H_{k}R_{h,u}&=-\sum_{\gamma\subset\Gamma^{I}}\left[\begin{array}{c}
		 1\\
		 1\\
		 0\\
		 0
		 \end{array}\right]^{T}\left[\begin{array}{cccc}
		 \T_{\gamma k}^{(1)} & -\T_{\gamma k}^{(1)} & \T_{\gamma k}^{(3)}-1 & \T_{\gamma k}^{(3)}\\
		 -\T_{\gamma v}^{(1)} & \T_{\gamma v}^{(1)} & \T_{\gamma v}^{(3)} & \T_{\gamma v}^{(3)}-1\\
		 \T_{\gamma k}^{(2)} & -\T_{\gamma k}^{(2)} & \T_{\gamma k}^{(4)} & \T_{\gamma k}^{(4)}\\
		 -\T_{\gamma v}^{(2)} & \T_{\gamma v}^{(2)} & \T_{\gamma v}^{(4)} & \T_{\gamma v}^{(4)}
		 \end{array}\right]\left[\begin{array}{c}
		 \R_{\gamma k}\bm{u}_{h,k}\\
		 \R_{\gamma v}\bm{u}_{h,v}\\
		 \D_{\gamma k}\bm{u}_{h,k}\\
		 \D_{\gamma v}\bm{u}_{h,v}
		 \end{array}\right]
		 \\
		 &\qquad-\sum_{\Omega_{k}\subset\mathcal{T}_{h}}\bm{1}^{T}\M_{k}\bm{u}_{h,k}
		 -\left\{\left[\begin{array}{c}
		 1\\
		 0
		 \end{array}\right]^{T}\left[\begin{array}{cc}
		 \T_{\gamma}^{D} & -1\\
		 -1 & 0
		 \end{array}\right]\left[\begin{array}{c}
		 \R_{\gamma k}\bm{u}_{k}-{u}_{D}\\
		 \D_{\gamma k}\bm{u}_{k}
		 \end{array}\right]\right\}_{\gamma\subset\Gamma^{D}} + u_{N},
	 \end{aligned} 
\end{equation}
which reduces to a sum of boundary terms only,
\begin{equation}
	\sum_{\Omega_{k}\subset{\cal T}_{h}}\bm{1}^{T}\H_{k}R_{h,u}=\left\{\D_{\gamma k}\bm{u}_{k}-\T_{\gamma}^{D}\left(\R_{\gamma k}\bm{u}_{k}-u_{D}\right)\right\}_{\gamma\subset\Gamma^{D}} + u_{N},
\end{equation}
as required for conservation of the discretization if 
\begin{align}\label{eq:conservation conditions}
\T_{\gamma k}^{(1)} = \T_{\gamma v}^{(1)}, && \T_{\gamma k}^{(3)} - 1 = -\T_{\gamma v}^{(3)}, && \bm{1}^T\M_k = \bm{0}.  
\end{align} 
Comparing \cref{eq:conservation conditions} and \cref{eq:adjoint consistency conditions} and noting that $ \M_k \bm{1}=\bm{0} $, we see that adjoint consistency implies conservation, as noted in \cite{arnold2002unified,hartmann2013higher,worku2020simultaneous}. 

\subsection{Functional Superconvergence}\label{sec:functinal error}
Without loss of generality, we assume that the domain is tessellated using two elements, $ \Omega_L $ and $ \Omega_R $. In the subsequent analysis, we will use the vectors $ \bm{u}$, $\bm{u}_h$, $\bm{\psi}$, $\bm{\psi}_h$, $\bm{f}$, $\bm{g}$, $\mathbb{E}(\bm{u}_h)$, $\mathbb{F}(\bm{\psi}_h) \in \IR{2n_p} $ given by 
\begin{equation}
\begin{aligned}
	\bm{u}_{h}&=\begin{bmatrix}
	\bm{u}_{h,L} \\ \bm{u}_{h,R}\end{bmatrix},
	&
	\bm{\psi}_{h}&=\begin{bmatrix}
	\bm{\psi}_{h,L} \\ \bm{\psi}_{h,R}\end{bmatrix},
	&
	\bm{u}&=\begin{bmatrix}
	\bm{u}_{L} \\ \bm{u}_{R}\end{bmatrix},
	&
	\bm{\psi}&=\begin{bmatrix}
	\bm{\psi}_{L} \\ \bm{\psi}_{R}\end{bmatrix},
	&
	\bm{f}&=\begin{bmatrix}
	\bm{f}_{L} \\ \bm{f}_{R}\end{bmatrix},
	&
	\bm{g}&=\begin{bmatrix}
	\bm{g}_{L} \\ \bm{g}_{R}\end{bmatrix},
\end{aligned}
\end{equation}
\begin{align}
	\mathbb{E}\left(\bm{u}_{h}\right)&=\begin{bmatrix}
	\left(\R_{\ell L}^{T}\T_{\ell L}^{(D)}-\D_{\ell L}^{T}\right)\left(\R_{\ell L}\bm{u}_{h,L}-u_{D}\right)\\
	\R_{r R}^{T}\left(\D_{r R}\bm{u}_{h,R}- {u}_{N}\right)
	\end{bmatrix}, \quad
	\mathbb{F}\left(\bm{\psi}_{h}\right)=\begin{bmatrix}
	\left(\R_{\ell L}^{T}\T_{\ell L}^{(D)}-\D_{\ell L}^{T}\right)\left(\R_{\ell L}\bm{\psi}_{h,L}-\psi_{D}\right)\\
	\R_{r R}^{T}\left(\D_{r R}\bm{\psi}_{h,R}-{\psi}_{N}\right)
	\end{bmatrix},
\end{align}	
and the matrices $ \mathbb{A},\mathbb{B},\mathbb{H},\mathbb{D}^{(2)} \in \IRtwo{2n_p}{2n_p} $ with block entries
\begin{equation}
	\begin{aligned}
		\mathbb{A}_{11}&=\begin{bmatrix}
		\R_{r L}\\
		\D_{r L}
		\end{bmatrix}^{T}\begin{bmatrix}
		\T_{r L}^{(1)} & \T_{r L}^{(3)}\\
		\T_{r L}^{(2)} & \T_{r L}^{(4)}
		\end{bmatrix}\begin{bmatrix}
		\R_{r L}\\
		\D_{r L}
		\end{bmatrix},
		& \mathbb{B}_{11}&=\begin{bmatrix}
		\R_{r L}\\
		\D_{r L}
		\end{bmatrix}^{T}\begin{bmatrix}
		\T_{r L}^{(1)} & \T_{r L}^{(2)}+1\\
		\T_{r L}^{(3)}-1 & \T_{r L}^{(4)}
		\end{bmatrix}\begin{bmatrix}
		\R_{r L}\\
		\D_{r L}
		\end{bmatrix},
		\\
		\mathbb{A}_{12}&=\begin{bmatrix}
		\R_{r L}\\
		\D_{r L}
		\end{bmatrix}^{T}\begin{bmatrix}
		\T_{r L}^{(1)} & \T_{r L}^{(3)}\\
		\T_{r L}^{(2)} & \T_{r L}^{(4)}
		\end{bmatrix}\begin{bmatrix}
		-\R_{\ell R}\\
		\D_{\ell R}
		\end{bmatrix},&\mathbb{B}_{12}&=\begin{bmatrix}
		\R_{r L}\\
		\D_{r L}
		\end{bmatrix}^{T}\begin{bmatrix}
		\T_{r L}^{(1)} & \T_{r L}^{(2)}+1\\
		\T_{r L}^{(3)}-1 & \T_{r L}^{(4)}
		\end{bmatrix}\begin{bmatrix}
		-\R_{\ell R}\\
		\D_{\ell R}
		\end{bmatrix},		
		\\
		\mathbb{A}_{21}&=\begin{bmatrix}
		\R_{\ell R}\\
		\D_{\ell R}
		\end{bmatrix}^{T}\begin{bmatrix}
		\T_{\ell R}^{(1)} & \T_{\ell R}^{(3)}\\
		\T_{\ell R}^{(2)} & \T_{\ell R}^{(4)}
		\end{bmatrix}\begin{bmatrix}
		-\R_{r L}\\
		\D_{r L}
		\end{bmatrix},&\mathbb{B}_{21}&=\begin{bmatrix}
		\R_{\ell R}\\
		\D_{\ell R}
		\end{bmatrix}^{T}\begin{bmatrix}
		\T_{\ell R}^{(1)} & \T_{\ell R}^{(2)}+1\\
		\T_{\ell R}^{(3)}-1 & \T_{\ell R}^{(4)}
		\end{bmatrix}\begin{bmatrix}
		-\R_{r L}\\
		\D_{r L}
		\end{bmatrix},
		\\
		\mathbb{A}_{22}&=\begin{bmatrix}
		\R_{\ell R}\\
		\D_{\ell R}
		\end{bmatrix}^{T}\begin{bmatrix}
		\T_{\ell R}^{(1)} & \T_{\ell R}^{(3)}\\
		\T_{\ell R}^{(2)} & \T_{\ell R}^{(4)}
		\end{bmatrix}\begin{bmatrix}
		\R_{\ell R}\\
		\D_{\ell R}
		\end{bmatrix},&\mathbb{B}_{22}&=\begin{bmatrix}
		\R_{\ell R}\\
		\D_{\ell R}
		\end{bmatrix}^{T}\begin{bmatrix}
		\T_{\ell R}^{(1)} & \T_{\ell R}^{(2)}+1\\
		\T_{\ell R}^{(3)}-1 & \T_{\ell R}^{(4)}
		\end{bmatrix}\begin{bmatrix}
		\R_{\ell R}\\
		\D_{\ell R}
		\end{bmatrix},
		\\
		\mathbb{H}&=\begin{bmatrix}
		\H_{L}\\
		& \H_{R}
		\end{bmatrix},
		&\mathbb{D}^{(2)}&=\begin{bmatrix}
		\D_{L}^{(2)}\\
		& \D_{R}^{(2)}
		\end{bmatrix},
		\\
		\mathbb{M}&=\begin{bmatrix}
		\M_L - \M_L^T\\
		& \M_R - \M_R^T
		\end{bmatrix}.
	\end{aligned}
\end{equation}
Note that for adjoint consistent schemes, it can be shown, using \cref{eq:adjoint consistency conditions}, that 
\begin{equation}
	\mathbb{A}_{12}^T = \mathbb{B}_{21},\quad \text{and}\quad \mathbb{A}_{21}^T = \mathbb{B}_{12}.
\end{equation} 
The discrete residuals corresponding to the steady version of \cref{eq:diffusion problem} and the adjoint problem \cref{eq:adjoint problem} can now be written, respectively, as
\begin{align}
	R_{h,u}\left(\bm{u}_{h}\right)&=-\mathbb{D}^{(2)}\bm{u}_{h}-\bm{f}+\mathbb{H}^{-1}\mathbb{A}\bm{u}_{h}+\mathbb{H}^{-1}\mathbb{E}\left(\bm{u}_{h}\right)=\bm{0},
	\\
	R_{h,\psi}\left(\bm{\psi}_{h}\right)&=-\mathbb{D}^{(2)}\bm{\psi}_{h}-\bm{g} +\mathbb{H}^{-1}\mathbb{B}\bm{\psi}_{h}+\mathbb{H}^{-1}\mathbb{F}\left(\bm{\psi}_{h}\right)- \mathbb{H}^{-1}\mathbb{M}\bm{\psi}_h=\bm{0}.
\end{align} 
Before stating the main result, we present an assumption regarding the primal and adjoint solution accuracy.
\begin{assumption}\label{assu:solution accuracy}
	We assume that unique numerical solutions for the steady version of the discrete primal equation \cref{eq:SBP-SAT discretization} and the discrete adjoint problem \cref{eq:SBP-SAT adjoint} exist, and these solutions are at least order $ h^{p+1} $ accurate in the maximum norm, \ie, $\norm{\bm{u}- \bm{u}_h}_{\infty} = \fnc{O}(h^{\ge p + 1})$ and $\norm{\bm{\psi}- \bm{\psi}_h}_{\infty}  = \fnc{O}(h^{\ge p+1})$.
\end{assumption}
\cref{assu:solution accuracy} is not necessary if pointwise stability of the SBP-SAT discretization for the Poisson problem can be demonstrated, see \cite{gustafsson1981convergence,svard2006order,svard2019convergence,hicken2012output,hicken2011superconvergent,penner2020superconvergent}. Numerical experiments with adjoint consistent discretizations show a primal and adjoint solution convergence rate of $ p+1 $ when a degree $ p $ diagonal-norm wide-stencil second-derivative SBP operator is used. In contrast, a primal and adjoint solution convergence rate of $ p+2 $ is observed when a degree $ p+1 $ order-matched narrow-stencil second-derivative SBP operator is used with adjoint consistent SATs. The block-norm wide- and narrow-stencil second-derivative operators, on the other hand, exhibit a primal solution convergence rate of $ 2p $.

We present the order of accuracy of the discrete functional approximating $ \fnc{I}(\fnc{U})=\fnc{I}(\psi) $ in the following theorem. 
\begin{theorem}\label{thm:functional superconvergence}
	Let the primal solution of the steady version of \cref{eq:diffusion problem} and the adjoint solution of \cref{eq:adjoint problem} be $ \fnc{U}, \psi \in \fnc{C}^{2p+2}(\Omega)$, respectively, the variable coefficient in \cref{eq:diffusion problem} and \cref{eq:adjoint problem} be $ \lambda \in \fnc{C}^{2p+1}(\Omega) $, and the source terms in \cref{eq:diffusion problem} and \cref{eq:adjoint problem} be $\fnc{F},\fnc{G} \in \fnc{C}^{2p}(\Omega)$, respectively. If $ \bm{u}_h,\bm{\psi}_h \in \IR{n_e n_p} $ are solutions to consistent discretizations of the steady version of \cref{eq:diffusion problem} and \cref{eq:adjoint problem}, respectively, and \cref{assu:solution accuracy} holds, then the discrete functionals \cref{eq:functional discrete 1,eq:functional discrete 2}	
	are order $ h^{2p} $ accurate approximations to the compatible linear functional $ \fnc{I}(\fnc{U})=\fnc{I}(\psi) $ given by \cref{eq:functional 1} and \cref{eq:functional 2}, \ie, 
	\begin{align} 
		\fnc{I}(\fnc{U})-I_h(\bm{u}_h)&=\order{h^{2p}}, \label{eq:functional error thm 1}\\
		\fnc{I}(\psi)-I_h(\bm{\psi}_h)&=\order{h^{2p}}.
		\label{eq:functional error thm 2}
	\end{align}
\end{theorem}
\begin{proof}
	It is sufficient to show that the result holds for a domain tessellated by two elements, $ \Omega_L $ and $ \Omega_R $, as the interface SATs considered couple immediate neighboring elements only. We let the Dirichlet and Neumann boundary conditions be implemented at the left and right boundaries of the domain. The boundary terms in both forms of the functional, \cref{eq:functional 1} and \cref{eq:functional 2}, involve the products $ \psi \lambda {\partial}\fnc{U}/{\partial}x $ and $ \fnc{U} \lambda {\partial}\fnc{\psi}/{\partial}x$. Using the continuity of $\psi$, $ \fnc{U}$, and $\lambda$, we can approximate $ (\psi \lambda {\partial}\fnc{U}/{\partial}x) \in \fnc{C}^{2p+1}(\Omega)$ and $(\fnc{U} \lambda {\partial}\fnc{\psi}/{\partial}x)  \in \fnc{C}^{2p+1}(\Omega)$ at the boundary nodes by degree $ \le 2p $ polynomials. The integrands in the volume integrals of \cref{eq:functional 1} and \cref{eq:functional 2} are $ 2p $ times differentiable, \ie, $ (\fnc{G}\fnc{U}), (\psi\fnc{F}) \in\fnc{C}^{2p}(\Omega)$. Since integrals are approximated by quadratures of order $ h^{2p} $, replacing $ (\fnc{G}\fnc{U}), (\psi\fnc{F}) \in\fnc{C}^{2p}(\Omega)$ by $ (\widetilde{\fnc{G}\fnc{U}}), (\widetilde{\psi\fnc{F}})\in\poly{2p-1} $ in the functionals introduces an error of order $ h^{2p} $. Therefore, we consider $ \widetilde{\fnc{U}}, (\widetilde{\lambda{\partial}\fnc{U}/{\partial}x}) \in \poly{p} $ to be at least order $ h^{p+1} $ approximations of $ \fnc{U} $ and $ {\lambda{\partial}\fnc{U}/{\partial}x} $, respectively, and thus $ \widetilde{\fnc{F}} \in\poly{p-1} $ due to the steady version of the primal PDE, \cref{eq:diffusion problem}. Similarly, considering $ \widetilde{\psi}, (\widetilde{\lambda{\partial}\psi/{\partial}x}) \in \poly{p} $ to be at least order $ h^{p+1} $ approximations of $ \psi$ and $ ({\lambda{\partial}\psi/{\partial}x}) $, respectively, gives $ \widetilde{\fnc{G}} \in \poly{p-1}$ due to the adjoint PDE, \cref{eq:adjoint problem}. For primal and adjoint consistent discretizations, the numerical primal and adjoint solutions are order $ h^{\ge p+1} $ accurate despite the polynomial approximations; hence, it is sufficient to show that either \cref{eq:functional error thm 1} or \cref{eq:functional error thm 2} hold for the polynomial integrands instead of the general continuous functions. Note that compatible functionals satisfy $ \fnc{I}(\fnc{U}) = \fnc{I}(\psi) $, and we enforced the condition $ I_h(\bm{u}_h) = I_h(\bm{\psi}_h) $ to find the discrete adjoint problem; hence, $  \fnc{I}(\fnc{U})-I_h(\bm{u}_h) = \fnc{I}(\psi) - I_h(\bm{\psi}_h)$. For the rest of the proof, we drop the tilde sign used to distinguish polynomials from the general continuous functions.
	
	If $ {\fnc{U}}\in\poly{p}  $ and $({\lambda{\partial}\fnc{U}/{\partial}x}) \in \poly{p}$, then we discretize \cref{eq:functional 1} to find
	\begin{equation} \label{eq:func exact 1}
		{\cal I}({\cal U})={\bm{u}}_{L}^{T}\H_{L}\bm{g}_{L}+\bm{g}_{R}^{T}\H_{R}{\bm{u}}_{R}-\psi_{D}w_{\ell L}+\psi_{N}u_{r R}+{\cal O}\left(h^{2p}\right),
	\end{equation}
	where $ w_{\ell L} = [\lambda\frac{\partial{{\cal U}}}{\partial x}n_{\ell}]_{\Gamma^{D}} $ and $ u_{r R} = {\fnc{U}}|_{\Gamma^N} $. Subtracting \cref{eq:functional discrete 1} from \cref{eq:func exact 1} and rearranging, we have
	\begin{equation}\label{eq:func proof 1}
		\begin{aligned}
			\mathcal{I}\left(\mathcal{U}\right)&=I_h{\bm{u}_h}-\bm{g}_{L}^{T}\H_{L}\left(\bm{u}_{h,L}-{\bm{u}}_{L}\right) +\psi_{D}\left(\D_{\ell L}\bm{u}_{h,L}-w_{\ell L}\right)
			-\psi_{D}\T_{\ell L}^{(D)}\left(\R_{\ell L}\bm{u}_{h,L}-u_{\ell L}\right) 
			\\& \quad -\bm{g}_{R}^{T}\H_{R}\left(\bm{u}_{h,R}-{\bm{u}}_{R}\right) -\psi_{N}\left(\R_{r R}\bm{u}_{h,R}-u_{r R}\right)+{\cal O}\left(h^{2p}\right).
		\end{aligned}
	\end{equation}
	Since $ {\fnc{U}}\in\poly{p}$, the $ \Rgk $ and $ \Dgk $ matrices are exact when applied to the restriction of $ \fnc{U} $ to the grid points, \eg, $ \R_{r R}{\bm{u}}_R = u_{r R} $ and $ \D_{\ell L}{\bm{u}}_L = w_{\ell L} $. Applying this property in \cref{eq:func proof 1} and simplifying we obtain
	\begin{equation} \label{eq:func proof 2}
		\mathcal{I}\left(\mathcal{U}\right)=I_{h}\left(\bm{u}_{h}\right)-\bm{g}^{T}\mathbb{H}\left(\bm{u}_{h}-{\bm{u}}\right)-\left[\begin{array}{c}
		\psi_{D}\T_{\ell L}^{(D)}\R_{\ell L}-\psi_{D}\D_{\ell L}\\
		\psi_{N}\R_{r R}
		\end{array}\right]\left(\bm{u}_{h}-{\bm{u}}\right)+{\cal O}\left(h^{2p}\right).
	\end{equation}
	Adding $ \bm{\psi}^{T}\mathbb{H}R_{h,u}\left(\bm{u}_{h}\right)=0 $ to the RHS of \cref{eq:func proof 2} and rearranging terms, we have
	\begin{equation} \label{eq:func proof 3}
		\begin{aligned}
			\mathcal{I}\left(\mathcal{U}\right)&=I_{h}\left(\bm{u}_{h}\right)-\bm{\psi}^{T}\mathbb{H}\mathbb{D}^{(2)}{\bm{u}}-\bm{\psi}^{T}\mathbb{H}\bm{f}+\bigg\{-\bm{g}^{T}-\bm{\psi}^{T}\mathbb{H}\mathbb{D}^{(2)}\mathbb{H}^{-1}
			-\left[\begin{array}{c}
			\psi_{D}\T_{\ell L}^{(D)}\R_{\ell L}-\psi_{D}\D_{\ell L}\\
			\psi_{N}\R_{r R}
			\end{array}\right]\mathbb{H}^{-1}
			\\&\quad
			+\bm{\psi}^{T}\mathbb{A}\mathbb{H}^{-1}
			+\bm{\psi}^{T}\left[\begin{array}{c}
			\R_{\ell L}^{T}\T_{\ell L}^{(D)}\R_{\ell L}-\D_{\ell L}^{T}\R_{\ell L}\\
			\R_{r R}^{T}\D_{r R}
			\end{array}\right]\mathbb{H}^{-1}\bigg\}\mathbb{H}\left(\bm{u}_{h}-{\bm{u}}\right)+{\cal O}\left(h^{2p}\right).
		\end{aligned}
	\end{equation} 
	Using the identity in \cref{eq:D2 decomposition 3} we can write 
	\begin{equation} \label{eq:D2 decomposition 4}
	\begin{aligned}
		-\mathbb{H}\mathbb{D}^{(2)}\mathbb{H}^{-1}&=-\left(\mathbb{D}^{(2)}\right)^{T}+\begin{bmatrix}
		\D_{\ell L}^{T}\R_{\ell L}-\R_{\ell L}^{T}\D_{\ell L}& \bm{0}\\
		\bm{0} & \D_{r R}^{T}\R_{r R}-\R_{r R}^{T}\D_{r R}	\end{bmatrix}\mathbb{H}^{-1}
		\\
		&\quad +\begin{bmatrix}
		\D_{r L}^{T}\R_{r L}-\R_{r L}^{T}\D_{r L} & \bm{0}\\
		\bm{0} & \D_{\ell R}^{T}\R_{\ell R}-\R_{\ell R}^{T}\D_{\ell R}
		\end{bmatrix}\mathbb{H}^{-1} + \mathbb{M}\mathbb{H}^{-1},
	\end{aligned}
	\end{equation}
	which, after substituting into \cref{eq:func proof 3} and simplifying, gives
	\begin{equation} \label{eq:func proof 4}
	\begin{aligned}
		\mathcal{I}\left(\mathcal{U}\right)&=I_{h}\left(\bm{u}_{h}\right)-\bm{\psi}^{T}\mathbb{H}\left[\mathbb{D}^{(2)}{\bm{u}}+\bm{f}\right]+\bigg\{-\bm{g}^{T}-\bm{\psi}^{T}\left(\mathbb{D}^{(2)}\right)^{T}
		\\
		&\quad +\bm{\psi}^{T}\mathbb{B}^{T}\mathbb{H}^{-1}+\left[\mathbb{F}\left(\bm{\psi}\right)\right]^{T}\mathbb{H}^{-1}-\bm{\psi}^{T}\mathbb{M}^T\mathbb{H}^{-1}\bigg\}\mathbb{H}\left(\bm{u}_{h}-{\bm{u}}\right)+{\cal O}\left(h^{2p}\right).
	\end{aligned}
	\end{equation}
	Since $ {\fnc{U}}\in \poly{p}$, the second term on the RHS vanishes due to the primal PDE. The third term is $ \fnc{O}(h^{\ge 2p+1}) $ due to the consistency of the adjoint discretization, the fact that $ \mathbb{H}  $ is $ \fnc{O}(h) $, and \cref{assu:solution accuracy}. Therefore, $ \mathcal{I}\left(\mathcal{U}\right)=I_{h}\left(\bm{u}_{h}\right)+{\cal O}\left(h^{2p}\right)$.
	
	Alternatively, if we consider $ {\psi}\in\poly{p}  $ and $({\lambda{\partial}\psi/{\partial}x}) \in \poly{p}$, we start by discretizing the second form of the functional, \cref{eq:functional 2},
	\begin{equation} \label{eq:func exact 2}
		{\cal I}(\psi)={\bm{\psi}}_{L}^{T}\H_{L}\bm{f}_{L}+{\bm{\psi}}_{R}^{T}\H_{R}\bm{f}_{R}-u_{D}z_{\ell L}+u_{N}\psi_{r R}+{\cal O}\left(h^{2p}\right),
	\end{equation}
	where $ z_{\ell L} = \left[\lambda\frac{\partial{\psi}}{\partial x}n_{\ell}\right]_{\Gamma^{D}} $ and $ \psi_{r R} = {\cal {\psi}}|_{\Gamma^{N}} $. Subtracting \cref{eq:functional discrete 2} from \cref{eq:func exact 2} and rearranging, we obtain 
	\begin{equation}
		\begin{aligned}
			\mathcal{I}\left(\psi\right)&= I_h(\bm{\psi}_h)-\bm{f}_{L}^{T}\H_{L}\left(\bm{\psi}_{h,L}-{\bm{\psi}}_{L}\right)+u_{D}\left(\D_{\ell L}\bm{\psi}_{h,L}-z_{\ell L}\right)
			-u_{D}\T_{\ell L}^{(D)}\left(\R_{\ell L}\bm{\psi}_{h,L}-\psi_{\ell L}\right)
			\\&\quad
			-\bm{f}_{R}^{T}\H_{R}\left(\bm{\psi}_{h,R}-{\bm{\psi}}_{R}\right)
			-u_{N}\left(\R_{r R}\bm{\psi}_{h,R}-\psi_{r R}\right)+{\cal O}\left(h^{2p}\right).
		\end{aligned}
	\end{equation}
	Using the accuracies of $ \Rgk $ and $ \Dgk $ to approximate the boundary terms, adding $ \bm{u}^{T}\mathbb{H}R_{h,\psi}\left(\bm{\psi}_{h}\right)=0 $, and simplifying leads to
	\begin{equation}
		\begin{aligned}
			\mathcal{I}\left(\psi\right)&=I_{h}\left(\bm{\psi}_{h}\right)-\left[\bm{u}^{T}\mathbb{H}\mathbb{D}^{(2)}{\bm{\psi}}+\bm{u}^{T}\mathbb{H}\bm{g}\right]+\bigg\{-\bm{f}^{T}-\bm{u}^{T}\mathbb{H}\mathbb{D}^{(2)}\mathbb{H}^{-1}
			-\left[\begin{array}{c}
			u_{D}\T_{\ell L}^{(D)}\R_{\ell L}-u_{D}\D_{\ell L}\\
			u_{N}\R_{r R}
			\end{array}\right]\mathbb{H}^{-1}
			\\&\quad
			+\bm{u}^{T}\mathbb{B}\mathbb{H}^{-1}
			+\bm{u}^{T}\left[\begin{array}{c}
			\R_{\ell L}^{T}\T_{\ell L}^{(D)}\R_{\ell L}-\D_{\ell L}^{T}\R_{\ell L}\\
			\R_{r R}^{T}\D_{r R}
			\end{array}\right]\mathbb{H}^{-1}\bigg\}\mathbb{H}\left(\bm{\psi}_{h}-{\bm{\psi}}\right)- \bm{u}^T\mathbb{M}\bm{\psi}_h + {\cal O}\left(h^{2p}\right).
		\end{aligned}
	\end{equation}
	Using the identity \cref{eq:D2 decomposition 4} and simplifying, we find
	\begin{equation}\label{eq:func proof 5}
	\begin{aligned}
		\mathcal{I}\left(\psi\right)&=I_{h}\left(\bm{\psi}_{h}\right)-\bm{u}^{T}\mathbb{H}\left[\mathbb{D}^{(2)}{\bm{\psi}}+\bm{g}\right]+\bigg\{-\bm{f}^{T}-\bm{u}^{T}\left(\mathbb{D}^{(2)}\right)^{T}
		+\bm{u}^{T}\mathbb{A}^{T}\mathbb{H}^{-1}
		\\&\quad
		+\left[\mathbb{E}\left(\bm{u}\right)\right]^{T}\mathbb{H}^{-1} -\bm{u}^T\mathbb{M}^T\mathbb{H}^{-1}\bigg\}\mathbb{H}\left(\bm{\psi}_{h}-{\bm{\psi}}\right) - \bm{u}^T\mathbb{M}\bm{\psi}_h+{\cal O}\left(h^{2p}\right).
	\end{aligned}
	\end{equation}
	Noting that $ \mathbb{M} + \mathbb{M}^T = \bm{0} $, we have 
		\begin{equation}\label{eq:func proof 6}
	\begin{aligned}
	\mathcal{I}\left(\psi\right)&=I_{h}\left(\bm{\psi}_{h}\right)-\bm{u}^{T}\mathbb{H}\left[\mathbb{D}^{(2)}{\bm{\psi}}+\bm{g}\right]+\bigg\{-\bm{f}^{T}-\bm{u}^{T}\left(\mathbb{D}^{(2)}\right)^{T}
	+\bm{u}^{T}\mathbb{A}^{T}\mathbb{H}^{-1}
	\\&\quad
	+\left[\mathbb{E}\left(\bm{u}\right)\right]^{T}\mathbb{H}^{-1} \bigg\}\mathbb{H}\left(\bm{\psi}_{h}-{\bm{\psi}}\right) - \bm{u}^T\mathbb{M}\bm{\psi}+{\cal O}\left(h^{2p}\right).
	\end{aligned}
	\end{equation}
	The second and fourth terms on the RHS of \cref{eq:func proof 6} vanish due to the adjoint PDE and the adjoint consistency requirement that $ \mathbb{M}=\bm{0} $, respectively. The third term is $ \fnc{O}(h^{\ge 2p+1}) $ due to the consistency of the primal discretization, the scaling of the norm matrix, and \cref{assu:solution accuracy}. Therefore, the estimates in \cref{eq:functional error thm 1,eq:functional error thm 2} hold.
	\qed
\end{proof}

\begin{remark}
	For implementations with the block-norm wide- or narrow-stencil  second-derivative SBP operators of the type presented in \cite{mattsson2013solution}, the estimate in \cref{eq:functional error thm 1} is attained even for adjoint inconsistent schemes. Note that for these types of operator, we have $ \norm{\bm{u}_{h,k} -\bm{u}_k}_\infty = \fnc{O}(h^{2p})$ in \cref{eq:func proof 4}. The block-norm narrow-stencil operators have $ \M_k \ne \M_k^T $; hence, they lead to adjoint inconsistent schemes even when coupled with adjoint consistent SATs.
\end{remark}
\subsection{Stability Analysis}\label{sec:energy analysis}
We use the energy method to analyze the stability of the SBP-SAT discretization of \cref{eq:diffusion problem}. The residual of the discretization for the homogeneous version of the problem, \ie, $ \fnc{F}=0 $, $ \fnc{U}_D = 0$, and $ \fnc{U}_N = 0$, summed over all elements can be written as 
\begin{equation}
	\begin{aligned}
		R_{h}(\bm{u}_{h},\bm{v})&=-\sum_{\Omega_{k}\in\mathcal{T}_{h}}\bm{v}_{k}^{T}\M_{k}\bm{u}_{h,k}-\sum_{\gamma\subset\Gamma^{D}}\left[\begin{array}{c}
		\R_{\gamma k}\bm{v}_{k}\\
		\D_{\gamma k}\bm{v}_{k}
		\end{array}\right]^{T}\left[\begin{array}{cc}
		\T_{\gamma k}^{(D)} & -1\\
		-1 & 0
		\end{array}\right]\left[\begin{array}{c}
		\R_{\gamma k}\bm{u}_{h,k}\\
		\D_{\gamma k}\bm{u}_{h,k}
		\end{array}\right]\\&\quad-\sum_{\gamma\subset\Gamma^{I}}\left[\begin{array}{c}
		\R_{\gamma k}\bm{v}_{k}\\
		\R_{\gamma v}\bm{v}_{v}\\
		\D_{\gamma k}\bm{v}_{k}\\
		\D_{\gamma v}\bm{v}_{v}
		\end{array}\right]^{T}\left[\begin{array}{cccc}
		\T_{\gamma k}^{(1)} & -\T_{\gamma k}^{(1)} & \T_{\gamma k}^{(3)}-1 & \T_{\gamma k}^{(3)}\\
		-\T_{\gamma v}^{(1)} & \T_{\gamma v}^{(1)} & \T_{\gamma v}^{(3)} & \T_{\gamma v}^{(3)}-1\\
		\T_{\gamma k}^{(2)} & -\T_{\gamma k}^{(2)} & \T_{\gamma k}^{(4)} & \T_{\gamma k}^{(4)}\\
		-\T_{\gamma v}^{(2)} & \T_{\gamma v}^{(2)} & \T_{\gamma v}^{(4)} & \T_{\gamma v}^{(4)}
		\end{array}\right]\left[\begin{array}{c}
		\R_{\gamma k}\bm{u}_{h,k}\\
		\R_{\gamma v}\bm{u}_{h,v}\\
		\D_{\gamma k}\bm{u}_{h,k}\\
		\D_{\gamma v}\bm{u}_{h,v}
		\end{array}\right]
	\end{aligned}
\end{equation}
for $ \bm{v}\in \IR{n_e n_p} $. In \cite{yan2018interior,worku2020simultaneous}, a factorization of $ \M_k $ for wide-stencil operators allowed the use of the borrowing trick and enabled $ R_h(\bm{u}_h,\bm{v}) $ to be written in terms of interface contributions only. However, the same factorization cannot be applied for narrow-stencil operators because $ \D_{\gamma k} $ is constructed using a modified derivative operator at the element boundaries, $ \D_{b,k} $, instead of $ \D_k $. To circumvent this, we make the following assumption.
\begin{assumption}\label{assu:Dbk}
	The $ \D_{b,k} $ matrix is invertible or can be modified such that it is invertible. 
\end{assumption}  

The invertibility requirement on $ \D_{b,k} $ is not too restrictive. In fact, all the narrow-stencil second-derivative operators in \cite{mattsson2004summation,mattsson2008discontinuous,mattsson2012summation,mattsson2013solution} either have invertible $ \D_{b,k} $ matrix or their $ \D_{b,k} $ matrix can be modified such that it is invertible. For operators that include nodes at element boundaries, the only requirement for $ \D_{b,k} $ to be invertible is that its interior diagonal entries are nonzero, \eg, $ \D_{b,k} $ can be constructed from the identity matrix by modifying the first and last rows such that these rows approximate the first derivative to degree $ \ge p+1 $. The invertibility of  $ \D_{b,k} $ matrix constructed in this manner can easily be verified using Gershgorin's theorem. For generalized narrow-stencil second-derivative operators with nodes at element boundaries, \eg, the hybrid Gauss-trapezoidal-Lobatto (HGTL) operators in \cite{del2015SecondDerivative}, a similar modification can be applied to obtain an invertible $ \D_{b,k} $ matrix. In contrast, all except the degree two hybrid Gauss-trapezoidal (HGT) operators in \cite{del2015SecondDerivative}, which do not include boundary nodes, do not yield an invertible $ \Dbk $ matrix even after applying the modification discussed. However, it is likely possible to construct HGT operators such that $ \D_{b,k} $ is invertible by enforcing a condition on the free variables during the construction of the operators. The structures of the invertible $ \D_{b,k} $ matrices of the degree two CSBP and HGT operators are,  
\begin{align*}
\left[\begin{array}{ccccccccccc}
\times & \times & \times & \times\\
& 1\\
&  & 1\\
\\
\\
&  &  &  &  & \ddots\\
\\
\\
&  &  &  &  &  &  &  & 1\\
&  &  &  &  &  &  &  &  & 1\\
&  &  &  &  &  &  & \times & \times & \times & \times
\end{array}\right], && 
\left[\begin{array}{ccccccccccc}
\times & \times & \times & \times & \times\\
\times & \times & \times & \times & \times\\
\times & \times & \times & \times & \times\\
\times & \times & \times & \times & \times\\
&  &  &  & 1\\
&  &  &  &  & \ddots\\
&  &  &  &  &  & 1\\
&  &  &  &  &  & \times & \times & \times & \times & \times\\
&  &  &  &  &  & \times & \times & \times & \times & \times\\
&  &  &  &  &  & \times & \times & \times & \times & \times\\
&  &  &  &  &  & \times & \times & \times & \times & \times
\end{array}\right],
\end{align*}
respectively, where each row containing $ \times $ in its entries approximate the first derivative. 

Another important assumption that is required in the subsequent energy stability analysis for adjoint consistent discretizations with narrow-stencil second-derivative operators is presented below.
\begin{assumption} \label{assu:nullspace}
	The first and second-derivative SBP operators, $ \Dk $ and $ \D^{(2)}_k $, are nullspace consistent, \ie, the nonzero vectors in the nullspace of $ \Dk $ and $ \D^{(2)}_k $ are $ \fnc{N}(\Dk) = {\rm span}\{\bm{1}\}\eqqcolon \bm{v}_c$ and $ \fnc{N}(\D^{(2)}_k) = {\rm span}\{\bm{1}, \bm{x}_k\} $, respectively.
\end{assumption}
It should be noted that consistency of an SBP operator does not necessarily imply nullspace consistency and vice versa. The operators defined in \cref{def:Dk,def:D2} are consistent because they satisfy the accuracy conditions, \ie, they differentiate polynomials up to a required degree exactly \cite{svard2019convergence}. In contrast, nullspace consistency requires that the nullspaces of $ \Dk $ and $ \DDk $ exclusively contain vectors in $ {\rm span}\{\bm{1}\} $ and $ {\rm span}\{\bm{1}, \bm{x}_k\} $, respectively. SBP derivative operators are consistent by construction, and most of them are nullspace consistent as well \cite{svard2019convergence}. 

Using \cref{assu:Dbk} and enforcing the conditions necessary for conservation, \cref{eq:conservation conditions}, we can now write the sum of the residual and its transpose as
\begin{equation} \label{eq:residual energy 1}
	\begin{aligned}
		2R_{h}(\bm{u}_{h},\bm{v}_{h})&= R_{h}(\bm{u}_{h},\bm{v}_{h}) + R_{h}^T(\bm{u}_{h},\bm{v}_{h}) =
		\\ 
		&\quad-\sum_{\gamma\subset\Gamma^{I}}
		\begin{bmatrix}
		\R_{\gamma k}\bm{u}_{h,k}\\
		\R_{\gamma v}\bm{u}_{h,v}\\
		\D_{b,k}\bm{u}_{h,k}\\
		\D_{b,v}\bm{u}_{h,v}
		\end{bmatrix}^{T}\begin{bmatrix}
		2\T_{\gamma k}^{(1)} & -2\T_{\gamma k}^{(1)} & \sigma_{k}\C_{\gamma k} & -\sigma_{v}\C_{\gamma v}\\
		-2\T_{\gamma k}^{(1)} & 2\T_{\gamma v}^{(1)} & -\sigma_{k}\C_{\gamma k} & \sigma_{v}\C_{\gamma v}\\
		\sigma_{k}\C_{\gamma k}^{T} & -\sigma_{k}\C_{\gamma k}^{T} & \alpha_{\gamma k}\V_{k} & \bm{0}\\
		-\sigma_{v}\C_{\gamma v}^{T} & \sigma_{v}\C_{\gamma v}^{T} & \bm{0} & \alpha_{\gamma v}\V_{v}
		\end{bmatrix}
		\begin{bmatrix}
		\R_{\gamma k}\bm{u}_{h,k}\\
		\R_{\gamma v}\bm{u}_{h,v}\\
		\D_{b,k}\bm{u}_{h,k}\\
		\D_{b,v}\bm{u}_{h,v}
		\end{bmatrix}
		\\
		&\quad-\sum_{\gamma\subset\Gamma^{I}}\begin{bmatrix}
		\D_{\gamma k}\bm{u}_{h,k}\\
		\D_{\gamma v}\bm{u}_{h,v}
		\end{bmatrix}^{T}\begin{bmatrix}
		2\T_{\gamma k}^{(4)} & 2\T_{\gamma k}^{(4)}\\
		2\T_{\gamma k}^{(4)} & 2\T_{\gamma}^{(4)}
		\end{bmatrix}\begin{bmatrix}
		\D_{\gamma k}\bm{u}_{h,k}\\
		\D_{\gamma v}\bm{u}_{h,v}
		\end{bmatrix}
		\\&\quad -\sum_{\gamma\subset\Gamma^{D}}\begin{bmatrix}
		\R_{\gamma k}\bm{u}_{h,k}\\
		\D_{b,k}\bm{u}_{h,k}
		\end{bmatrix}^{T}\begin{bmatrix}
		2\T_{\gamma k}^{(D)} & -2\C_{\gamma k}\\
		-2\C_{\gamma k}^{T} & \alpha_{\gamma k}\V_{k}
		\end{bmatrix}
		\begin{bmatrix}
		\R_{\gamma k}\bm{u}_{h,k}\\
		\D_{b,k}\bm{u}_{h,k}
		\end{bmatrix},
	\end{aligned}
\end{equation}
where $ \C_{\gamma k} = n_{\gamma k}\Rgk\Lambda_{k} $, $ \C_{\gamma v} = n_{\gamma v}\Rgv\Lambda_{v} $, $ \sigma_{k}=\T_{\gamma k}^{(2)}+\T_{\gamma k}^{(3)}-1 $, $ \sigma_{v}=\T_{\gamma v}^{(2)}+\T_{\gamma v}^{(3)}-1 $, $ \alpha_{\gamma k}$ is a positive interface weight factor satisfying the relation $ \sum_{\gamma \in \Gamma_k} \alpha_{\gamma k} = 1 $, and 
\begin{align}
	 \V_{k}=\D_{b,k}^{-T}(\M_{k}+\M^T_{k})\D_{b,k}^{-1}, \qquad \V_{v}=\D_{b,v}^{-T}(\M_{v}+\M^T_{v})\D_{b,v}^{-1}. 
\end{align} 
We note that $ \V_{k} $ is positive semidefinite since $ \bm{v}^T(\M_{k} + \M_{k}^T)\bm{v}\ge 0 $ for all $ \bm{v}\in \IR{n_p} $ implies
\begin{equation}
(\Dbk^{-1}\bm{v})^T(\M_{k} + \M_{k}^T)(\Dbk^{-1}\bm{v})\ge 0.
\end{equation}
Moreover, we have
\begin{equation} \label{eq:Dbk vc}
	 \Dbk \bm{v}_c = \bm{v}_0, \quad \text{or} \quad \Dbk^{-1} \bm{v}_0 = \bm{v}_c,
\end{equation}
where $ \bm{v}_0 $ represents vectors containing zero at the entries corresponding to the rows for which $ \Dbk $ contains consistent approximations of the first derivative and the values of $ \bm{v}_c $ at all other entries.

For diagonal-norm narrow-stencil SBP operators that are constructed as in \cite{mattsson2004summation,mattsson2008discontinuous,mattsson2012summation,del2015SecondDerivative,mattsson2013solution}, we can determine the vectors in the nullspace of $ \V_{k} $ using \cref{assu:nullspace,assu:Dbk}.
\begin{lemma}\label{lem:nullspace}
	Consider a consistent diagonal-norm narrow-stencil second-derivative SBP operator of the form \cref{eq:D2 decomposition 1} for which $ \M_{k} = \M_{k}^T $, the $ \E_k $ matrix is constructed such that it has nonzero rows only at row indices where the $ \Dbk $ matrix contains consistent approximations of the first derivative, the $ \DDk $ matrix has larger dense blocks at the top left and bottom right corners than the $ \E_k $ matrix, and \cref{assu:Dbk,assu:nullspace} hold. Then, we have $ \fnc{N}(\M_k) = \bm{v}_c$ and $ \fnc{N}(\V_{k}) = \fnc{N}(\V_{v})= \bm{v}_0 $.
\end{lemma}
\begin{proof}
The nullspace consistency of the $ \DDk $ in \cref{assu:nullspace} implies the following:
\begin{align} 
		\Hk \DDk \bm{v}_c &= \left(-\M_k + \Ek \Lambda_{k}\Dbk \right)\bm{v}_c = -\M_k\bm{v}_c + \Ek \Lambda_{k}\Dbk \bm{v}_c = \bm{0}, \label{eq:nullspace of M 1}
		\\
		\Hk \DDk \bm{x}_k &= \left(-\M_k + \Ek \Lambda_{k}\Dbk \right)\bm{x}_k = -\M_k\bm{x}_k + \Ek \Lambda_{k}\Dbk \bm{x}_k = \bm{0}.
		\label{eq:nullspace of M 2}
\end{align}
The second term in the last equality in \cref{eq:nullspace of M 1} is zero due to the structure of the $ \Ek $ matrix, \ie, $ \Ek \Lambda_{k}\Dbk \bm{v}_c = \bm{0}$; thus, $ \M_k \bm{v}_c = \bm{0} $. Furthermore, $ \M_k\bm{x}_k\ne \bm{0} $ in \cref{eq:nullspace of M 2} since otherwise we would obtain 
\begin{equation}\label{eq:Lem1 1}
	\Hk \DDk \bm{x}_k = \Ek \Lambda_{k}\Dbk \bm{x}_k = \R_{r k}^T\D_{r k}\bm{x}_k + \R_{\ell k}^T \D_{\ell k}\bm{x}_k = \R_{r k}^T\lambda_{r} - \R_{\ell k}^T \lambda_{\ell}= \bm{0},
\end{equation} 
which is not possible as $ \R_{r k} $ and $ \R_{\ell k} $ do not have nonzero values at the same entries, and $ \lambda > 0 $. We have used the accuracy of $ \D_{r k} $ and $ \D_{\ell k} $ in the penultimate equality in \cref{eq:Lem1 1}, \ie, $ \D_{r k}\bm{x}_k = \lambda_{r} $ and $ \D_{\ell k}\bm{x}_k = -\lambda_{\ell} $, where $ \lambda_{\ell} $ and $ \lambda_{r} $ are at least order $ h^{p+1} $ approximations of $ \lambda $ at the left and right boundaries of $ \Omega_{k} $, respectively.
Hence, there is no vector spanned by $ \{\bm{1},\bm{x}_k\} $ other than $ \bm{v}_c $ that is in the nullspace of $ \M_k $. If there exists a nontrivial vector $ \bm{v} $ such that $\bm{v} \notin {\rm span}\{\bm{1}, \bm{x}_k\} $ and $  \M_k \bm{v} =  \bm{0} $, then
\begin{equation} \label{eq:Lem1 2}
	\H_k\DDk \bm{v}=\Ek \Lambda_{k}\Dbk \bm{v} \ne \bm{0},
\end{equation} 
because $ \DDk $ is nullspace consistent and $ \H_k $ is SPD. The vector $\Ek \Lambda_{k}\Dbk \bm{v} $ has zero entries at rows corresponding to the zero rows of the $ \Ek $ matrix. By construction, $ \DDk $ has larger dense blocks at the top left and bottom right corners (consisting of more rows and columns) than the $ \Ek $ matrix; therefore, it follows from the nullspace consistency of the $ \DDk $ matrix that $ [\H_k\DDk \bm{v}]_i \ne 0$ and $ [\Ek \Lambda_{k}\Dbk \bm{v}]_i = 0 $, at least for one entry, the $ i $-th entry, near the boundaries. This implies that the equality in \cref{eq:Lem1 2} cannot hold for any vector $\bm{v}\notin {\rm span}\{\bm{1},\bm{x}_k\} $; hence, we have $ \fnc{N}(\M_k)=\bm{v}_c $. Since $  \M_{k}=  \M_{k}^T $, it follows that $ \fnc{N}(\M_{k}^{T})=\bm{v}_c $. Using the result in \cref{eq:Dbk vc} with the fact that $ \fnc{N}(\M_k) = \fnc{N}(\M_{k}^T)= \bm{v}_c $, we obtain
\begin{equation} \label{eq:Vk nullspace}
	\V_{k}\bm{v}_0 = \Dbk^{-T}(\M_k + \M_{k}^T) \Dbk^{-1}\bm{v}_0 = \Dbk^{-1}(\M_k + \M_{k}^T) \bm{v}_c = \bm{0}.
\end{equation}
Thus, $ \bm{v}_0 $ is the only nontrivial vector in the nullspace of $ \V_{k} $. Analogous results hold for $ \V_{v} $.
\qed
\end{proof}

In \cite{eriksson2018finite}, the stability conditions that the SATs must satisfy were derived for diagonal-norm narrow-stencil SBP operators assuming that $ \M_{k}$ is SPD; however, most operators in the literature, \eg, \cite{mattsson2004summation,mattsson2008discontinuous,mattsson2012summation,del2015SecondDerivative,mattsson2013solution}, do not satisfy this requirement. For dense-norm narrow-stencil second-derivative SBP operators, we make the following assumption regarding the nullspaces of $ \M_k $ and $ \M_k^T $:
\begin{assumption}\label{ass:nullspace of M_k}
	For dense-norm narrow-stencil SBP operators, $ \bm{v}_c $ is the only nontrivial vector in the nullspaces of $ \M_k $ and $ \M_k^T $, \ie, $ \fnc{N}(\M_k) = \fnc{N}(\M_k^T) = \bm{v}_c $.  
\end{assumption}
Under \cref{ass:nullspace of M_k}, \cref{eq:Vk nullspace} gives $ \fnc{N}(\V_{k}) = \fnc{N}(\V_{v})= \bm{v}_0 $ for dense-norm narrow-stencil SBP operators. Before proceeding with the energy analysis of the SBP-SAT discretization, we state an essential theorem, which is proved in \cite{albert1969conditions,gallier2010schur}.
\begin{theorem} \label{thm:Positive semi-definiteness}
	A symmetric matrix of the form 
	$\Y =\bigl[\begin{smallmatrix}
	\Y_{11} & \Y_{12} \\ \Y_{12}^T & \Y_{22}
	\end{smallmatrix}\bigr]$ is positive semidefinite if and only if
		\begin{equation}  \label{eq:positive semidefiniteness}
		\Y_{22}\succeq 0, \quad  (\I-\Y_{22}\Y_{22}^{+})\Y_{12}^T  = \bm{0}, \quad \text{and} \quad \Y_{11} - \Y_{12}\Y_{22}^{+}\Y_{12}^T \succeq 0,
		\end{equation}
	where $ \Y^{+}$ denotes the Moore-Penrose pseudoinverse of $ \Y $ and $ \Y \succeq 0 $ indicates that $ \Y $ is positive semidefinite.
\end{theorem}

An SBP-SAT discretization is energy stable if 
\begin{equation}
\dv{}{t}\left(\norm{\bm{u}_h}^2_\H\right) =\bm{u}_h^T \H\dv{\bm{u}_h}{t}+\dv{\bm{u}_h^T}{t}\H\bm{u}_h = 2 R_{h}\left(\bm{u}_{h},\bm{u}_{h}\right) \le 0.
\end{equation} 
The sum of the residual and its transpose for conservative schemes, $ 2 R_{h}(\bm{u}_{h},\bm{u}_{h})$, given in \cref{eq:residual energy 1} satisfies the energy stability condition if 
\begin{equation}
\begin{aligned} \label{eq:matrices for stability}
	&\A\coloneqq\begin{bmatrix}
	\A_{11} & \A_{12} \\
	\A_{21} & \A_{22}
	\end{bmatrix} =\begin{bmatrix}
	2\T_{\gamma k}^{(1)} & -2\T_{\gamma k}^{(1)} & \sigma_{k}\C_{\gamma k} & -\sigma_{v}\C_{\gamma v}\\
	-2\T_{\gamma k}^{(1)} & 2\T_{\gamma v}^{(1)} & -\sigma_{k}\C_{\gamma k} & \sigma_{v}\C_{\gamma v}\\
	\sigma_{k}\C_{\gamma k}^{T} & -\sigma_{k}\C_{\gamma k}^{T} & \alpha_{\gamma k}\V_{k} & \bm{0}\\
	-\sigma_{v}\C_{\gamma v}^{T} & \sigma_{v}\C_{\gamma v}^{T} & \bm{0} & \alpha_{\gamma v}\V_{v}
	\end{bmatrix}, \quad
	\begin{bmatrix}
	2\T_{\gamma k}^{(4)} & 2\T_{\gamma k}^{(4)}\\
	2\T_{\gamma k}^{(4)} & 2\T_{\gamma}^{(4)}
	\end{bmatrix},
	\quad
	\begin{bmatrix}
	2\T_{\gamma k}^{(D)} & -2\C_{\gamma k}\\
	-2\C_{\gamma k}^{T} & \alpha_{\gamma k}\V_{k}
	\end{bmatrix},
\end{aligned}
\end{equation}
are positive semidefinite. We partition the matrix $ \A \in \IRtwo{(2+2n_p)}{(2+2n_p)} $ using four blocks, namely $ \A_{11}\in \IRtwo{2}{2} $, $ \A_{12} \in \IRtwo{2}{2n_p} $, $ \A_{21} \in \IRtwo{2n_p}{2} $, and $ \A_{22} \in \IRtwo{2n_p}{2n_p} $. 

For adjoint consistent SATs, we enforce all the conditions in \cref{eq:adjoint consistency conditions}. Furthermore, to obtain a symmetric $ \A $ matrix we require that $ \T_{\gamma k}^{(3)}-\T_{\gamma k}^{(2)}=1 $, as in \cite{yan2018interior}. This condition is satisfied by the SATs corresponding to some of the popular discontinuous Galerkin fluxes for elliptic PDEs \cite{worku2020simultaneous}, \eg, the modified method of Bassi and Rebay (BR2) \cite{bassi1997highbr2}, local discontinuous Galerkin (LDG) \cite{shu2001different}, and compact discontinuous Galerkin (CDG) \cite{peraire2008compact} methods. With the adjoint consistency and symmetry conditions in place, the components of the $ \A $ matrix become
\begin{equation} \label{eq:components of A}
	\begin{aligned}
		\A_{11}&=\begin{bmatrix}
			2\T_{\gamma k}^{(1)} & -2\T_{\gamma k}^{(1)} \\
			-2\T_{\gamma k}^{(1)} & 2\T_{\gamma v}^{(1)}
		\end{bmatrix},
		&&
		\A_{12} =\begin{bmatrix}
			 2\T_{\gamma k}^{(2)}\C_{\gamma k} & -2\T_{\gamma v}^{(2)}\C_{\gamma v} \\
			 -2\T_{\gamma k}^{(2)}\C_{\gamma k} & 2\T_{\gamma v}^{(2)}\C_{\gamma v}
		\end{bmatrix},
		\\
		\A_{21} &= \begin{bmatrix}
			2\T_{\gamma k}^{(2)}\C_{\gamma k}^{T} & -2\T_{\gamma k}^{(2)}\C_{\gamma k}^{T}\\
			-2\T_{\gamma v}^{(2)}\C_{\gamma v}^{T} & 2\T_{\gamma v}^{(2)}\C_{\gamma v}^{T}
		\end{bmatrix},
		&&
		\A_{22} = \begin{bmatrix}
			\alpha_{\gamma k}\V_{k} & \bm{0}\\
			\bm{0} & \alpha_{\gamma v}\V_{v}
		\end{bmatrix}.
	\end{aligned}
\end{equation}

\begin{theorem} \label{thm:stability}
	Let \cref{assu:Dbk,assu:nullspace,ass:nullspace of M_k} hold, then an adjoint consistent SBP-SAT discretization of the diffusion problem, \cref{eq:diffusion problem}, that produces a symmetric $ \A $ matrix, \ie, $ \T_{\gamma k}^{(3)}-\T_{\gamma k}^{(2)}=1 $, and uses a narrow-stencil second-derivative operator of the form \cref{eq:D2 decomposition 1} for the spatial discretization is energy stable if 
	\begin{align}
		\T_{\gamma k}^{(1)}&\ge \frac{2}{\alpha_{\gamma k}}\T_{\gamma k}^{(2)}\R_{\gamma k}\Lambda_{k}\V_{k}^{+}\Lambda_{k}\R_{\gamma k}^{T} \T_{\gamma k}^{(2)}+\frac{2}{\alpha_{\gamma v}}\T_{\gamma v}^{(2)}\R_{\gamma v}\Lambda_{v}\V_{v}^{+}\Lambda_{v}\R_{\gamma v}^{T}\T_{\gamma v}^{(2)},
		\label{eq:1 stability thm}
		\\
		\T_{\gamma k}^{(D)}&\ge \frac{2}{\alpha_{\gamma k}}\R_{\gamma k}\Lambda_{k}\V_{k}^{+}\Lambda_{k}\R_{\gamma k}^{T},
		\label{eq:2 stability thm}
		\\
		\T_{\gamma k}^{(4)}& \ge 0
		\label{eq:3 stability thm}.
	\end{align}
\end{theorem}
\begin{proof}
 We wish to show that the matrices in \cref{eq:matrices for stability} are positive semidefinite. The matrix $ \A $, whose components are given in \cref{eq:components of A}, is symmetric; thus, we can use \cref{thm:Positive semi-definiteness} to determine the conditions required for it to be positive semidefinite. We have $ \A_{22} \succeq 0$ because $ \V_{k} $ and $ \V_{v} $ are positive semidefinite, $ \alpha_{\gamma k}>0$, and $\alpha_{\gamma v} >0 $. Therefore, the first condition in \cref{thm:Positive semi-definiteness} is satisfied. The second condition in \cref{thm:Positive semi-definiteness} requires that 
 \begin{equation} \label{eq:stability proof 1}
 	\left(\I-\A_{22}\A_{22}^{+}\right)\A_{21}= 2 
 	\begin{bmatrix}
 	\I-\V_{k}\V_{k}^{+} & \bm{0}\\
 	\bm{0} & \I-\V_{v}\V_{v}^{+}
 	\end{bmatrix}
 	\begin{bmatrix}
 	\C_{\gamma k}^{T}\T_{\gamma k}^{(2)} & -\C_{\gamma k}^{T}\T_{\gamma k}^{(2)}\\
 	-\C_{\gamma v}^{T}\T_{\gamma v}^{(2)} & \C_{\gamma v}^{T}\T_{\gamma v}^{(2)}
 	\end{bmatrix}=0.
 \end{equation}
 To show that \cref{eq:stability proof 1} holds, we consider the singular value decomposition of $ \V_{k} $,
 \begin{equation}\label{eq:svd V}
 	\V_{k}=\X\Sigma\Y^{T},
 \end{equation}
 where the columns of $ \X $ and $ \Y $ contain orthonormal basis vectors of the column and row spaces, respectively, and $ \Sigma $ is a diagonal matrix containing the singular values of $ \V_{k} $ along its diagonal. \cref{lem:nullspace} and \cref{ass:nullspace of M_k} ensure that the matrix $ \V_{k} \in \IRtwo{n_p}{n_p} $ has only one nontrivial vector in its nullspace; hence, the first $ n_p - 1 $ columns of $ \X $ contain orthonormal basis vectors that span the column space of $ \V_{k} $ and the last column contains the vector in the nullspace of $ \V_{k} $, which is $ \bm{v}_0 $. We also note that 
 \begin{equation} \label{eq:stability proof 2}
 	\V_{k}\V_{k}^{+}=\X\Sigma\Y^{T}\Y\Sigma^{+}\X^{T}=\X\Sigma\Sigma^{+}\X^{T}=\X\I_m\X^{T},
 \end{equation} 
 where we have used the orthonormality of $ \Y $ in the second equality, \ie, $ \Y^{T}\Y=\I $, and $ \I_m $ denotes the identity matrix with the last diagonal entry set to zero (since $ \Sigma_{ii} = 0$ for $ i>m $, where $ m = n_p -1 $ is the rank of $ \V_{k} $). Therefore, for operators that include the boundary nodes, we have $ \I-\V_{k}\V_{k}^{+}=\I-\X\I_m\X^{T} $, which gives
 \begin{equation}\label{eq:stability proof 3}
	 	\I-\X\I_m\X^{T}=
	 	\I
	 	-\begin{bmatrix}
	 	\times & \dots &  & \times & 0\\
	 	\times & \dots  &  & \times & \times\\
	 	\vdots &  &  & \vdots & \vdots\\
	 	\times & \dots  &  & \times & \times\\
	 	\times & \dots &  & \times & 0
	 	\end{bmatrix}
	 	\begin{bmatrix}
	 	1\\
	 	& 1\\
	 	&  & \ddots\\
	 	&  &  & 1\\
	 	&  &  &  & 0
	 	\end{bmatrix}
	 	\begin{bmatrix}
	 	\times & \times & \dots & \times & \times\\
	 	\vdots & \vdots &  & \vdots& \vdots\\
	 	\\
	 	\times & \times & \dots & \times & \times\\
	 	0 & \times & \dots & \times & 0
	 	\end{bmatrix}
	 	=\begin{bmatrix}
	 	0 & 0 & \dots & 0 & 0\\
	 	0 & \times & \dots & \times & 0\\
	 	\vdots & \vdots &  & \vdots & \vdots\\
	 	0 & \times & \dots & \times & 0\\
	 	0 & 0 & \dots & 0 & 0
	 	\end{bmatrix},
\end{equation}
where $ \times $ denotes an entry that we do not need to specify for this analysis. Similarly, for operators that do not include boundary nodes,  we obtain
\begin{align}  \label{eq:stability proof 4}
	\I-\V_{k}\V_{k}^{+}=\I-\X\I_m\X^{T}=
	\begin{bmatrix}
	0\\
	& \ddots\\
	&  & 0\\
	&  &  & \times & \dots & \times\\
	&  &  & \vdots &  & \vdots\\
	&  &  & \times & \dots & \times\\
	&  &  &  &  &  & 0\\
	&  &  &  &  &  &  & \ddots\\
	&  &  &  &  &  &  &  & 0
	\end{bmatrix},
\end{align}
\ie, the first and last $ s $ rows are zero, where $ s $ is half of the number of zero entries in $ \bm{v}_0 $. For operators that have nodes at the boundaries, the LHS of \cref{eq:stability proof 1} can be evaluated using \cref{eq:stability proof 3} as
\begin{equation}  \label{eq:stability proof 5}
	\left(\I-\A_{22}\A_{22}^{+}\right)\A_{21}
	=\begin{bmatrix}
	\begin{array}{ccccc}
	0 & 0 & \dots & 0 & 0\\
	0 & \times & \dots & \times & 0\\
	\vdots & \vdots &  & \vdots & \vdots\\
	0 & \times & \dots & \times & 0\\
	0 & 0 & \dots & 0 & 0
	\end{array}\\
	& \begin{array}{ccccc}
	0 & 0 & \dots & 0 & 0\\
	0 & \times & \dots & \times & 0\\
	\vdots & \vdots &  & \vdots & \vdots\\
	0 & \times & \dots & \times & 0\\
	0 & 0 & \dots & 0 & 0
	\end{array}
	\end{bmatrix}
	\begin{bmatrix}
	\begin{array}{c}
	0\\
	\vdots\\
	\\
	0\\
	\T_{\gamma k}^{(2)}\lambda_{\gamma k}
	\end{array} & \begin{array}{c}
	0\\
	\vdots\\
	\\
	0\\
	-\T_{\gamma k}^{(2)}\lambda_{\gamma k}
	\end{array}\\
	\begin{array}{c}
	-\T_{\gamma v}^{(2)}\lambda_{\gamma v}\\
	0\\
	\vdots\\
	0\\
	0
	\end{array} & \begin{array}{c}
	\T_{\gamma v}^{(2)}\lambda_{\gamma v}\\
	0\\
	\vdots\\
	0\\
	0
	\end{array}
	\end{bmatrix}=0.
\end{equation}
Similarly, substituting \cref{eq:stability proof 4} into \cref{eq:stability proof 1}, it is straightforward to show that the condition $ \left(\I-\A_{22}\A_{22}^{+}\right)\A_{21} = 0$ also holds for operators that do not include boundary nodes. For $ \A $ to be positive semidefinite, it remains to find sufficient conditions to satisfy the last requirement in \cref{thm:stability}, \ie, $ \A_{11}-\A_{12}\A_{22}^{+}\A_{21} \succeq 0$, which, after some algebra, gives the condition
\begin{align}
		\begin{bmatrix}
		1 & -1\\
		-1 & 1
		\end{bmatrix}\otimes\left[\T_{1}-\left(\frac{2}{\alpha_{\gamma k}}\T_{\gamma k}^{(2)}\C_{\gamma k}\V_{k}^{+}\C_{\gamma k}^{T}\T_{\gamma k}^{(2)}+\frac{2}{\alpha_{\gamma v}}\T_{\gamma v}^{(2)}\C_{\gamma v}\V_{v}^{+}\C_{\gamma v}^{T}\T_{\gamma v}^{(2)}\right)\right] \succeq 0,	
		\label{eq:stability proof 6}	 
\end{align}
where $ \otimes $ denotes the Kronecker product. Since $ \bigl[\begin{smallmatrix}
1 & -1 \\ -1 & 1
\end{smallmatrix}\bigr] \succeq 0 $, the inequality in \cref{eq:stability proof 6} is satisfied if 
\begin{align}
	\T_{\gamma}^{(1)} &\ge \frac{2}{\alpha_{\gamma k}}\T_{\gamma k}^{(2)}\C_{\gamma k}\V_{k}^{+}\C_{\gamma k}^{T}\T_{\gamma k}^{(2)}+ \frac{2}{\alpha_{\gamma v}}\T_{\gamma v}^{(2)}\C_{\gamma v}\V_{v}^{+}\C_{\gamma v}^{T}\T_{\gamma v}^{(2)},
\end{align}
which is the same as the condition given by \cref{eq:1 stability thm}. Note that  $ \C_{\gamma k}\V_{k}^{+}\C_{\gamma k}^{T} = \R_{\gamma k}\Lambda_{k}\V_{k}^{+}\Lambda_{k}\R_{\gamma k}^{T} $ since $ n_{\gamma k}^2 = 1 $.

The second matrix in \cref{eq:matrices for stability} can be written as 
\begin{equation}
	\begin{bmatrix}
	2\T_{\gamma k}^{(4)} & 2\T_{\gamma k}^{(4)}\\
	2\T_{\gamma k}^{(4)} & 2\T_{\gamma k}^{(4)}
	\end{bmatrix} = 2\T_{\gamma k}^{(4)}\begin{bmatrix}
	1 & 1\\
	1 & 1
	\end{bmatrix},
\end{equation}
which is positive semidefinite provided $ \T_{\gamma k}^{(4)}=\T_{\gamma v}^{(4)} \ge 0$, since $ \bigl[\begin{smallmatrix}
1 & 1 \\ 1 & 1
\end{smallmatrix}\bigr] \succeq 0$ . Finally, we note that the last matrix in \cref{eq:matrices for stability} is symmetric, and $ \alpha_{\gamma k}\V_{k} \succeq 0 $. Furthermore, using the same approach used to obtain \cref{eq:stability proof 5} it can be shown that 
\begin{equation}
	(\I - \V_{k}\V_{k}^{+})(-2\C_{\gamma k}^T)=0
\end{equation}
irrespective of whether or not the operator includes boundary nodes. The last condition required for positive semidefiniteness of the last matrix in \cref{eq:matrices for stability} is
\begin{equation}
	2\T_{\gamma k}^{(D)} - 4\C_{\gamma k}\left(\alpha_{\gamma k}\V_{k}\right)^{+}\C_{\gamma k}^{T} \ge 0,
\end{equation}
which is satisfied if \cref{eq:2 stability thm} holds. Therefore, the conditions in \cref{thm:stability} are indeed sufficient for energy stability.
\qed
\end{proof}

\begin{remark}
	In practice, the pseudoinverse $ \V_{k}^{+} $ for a given SBP operator is computed once and for all on the reference element. It is then scaled by the inverse of the metric Jacobian when the operator is mapped to the physical elements. In Table 1 of \cite{eriksson2018dual}, equivalent values of $ 2\R_{\gamma k}\V_{k}^{+}\R_{\gamma k}^T $, denoted by $ q $ and scaled by the mesh spacing, are tabulated for the constant-coefficient diagonal-norm narrow-stencil SBP operators presented in \cite{mattsson2004summation}. The scaling used can be written as $ h= 1/[(n_e n_p - 1) - (n_e - 1)] $, and $ 2\V_k^{+} = (\Dbk^{-T}\M_k \Dbk^{-1})^{+} $ for operators with $ \M_{k} = \M_{k}^T $ is computed as $ \Dbk (\widetilde{\M}_k)^{-1} \Dbk^{T}$, where $ \widetilde{\M}_k$ is obtained by perturbing $ \M_k $ such that the corner values of $ \Dbk (\widetilde{\M}_k)^{-1} \Dbk^{T} $ are independent of the perturbation. The difference in the values of $ qh $ and $ 2h\R_{\gamma k}\V_{k}^{+}\R_{\gamma k}^T $ lies in the approaches pursued to evaluate $ 2\V_k^{+} $. 
\end{remark}

\begin{remark}
	For stability of discretizations with wide-stencil second-derivative operators, the terms $ \V_{k}^{+} $ and  $ \V_{v}^{+} $ in \cref{thm:stability} are replaced by $ (\H_k\Lambda_{k}+\Lambda_{k}\H_k)^{-1} $ and $  (\H_v\Lambda_{v}+\Lambda_{v}\H_v)^{-1}  $, respectively. 
\end{remark}

\subsection{Interface SATs}
In this section, we present a few concrete examples of SATs for diffusion problems. The type of SAT used in the discretization affects several numerical properties such as accuracy, stability, conditioning, symmetry, and sparsity. However, we do not analyze many of these properties; rather we limit our focus to aspects of solution and functional convergence. With this in mind, we introduce four SATs, of which two are stable and adjoint consistent while the other two are stable but adjoint inconsistent when implemented with narrow-stencil SBP operators. A more comprehensive analysis of SATs for diffusion problems is presented in \cite{worku2020simultaneous}, and additional types of SAT that are not studied in this work can be found therein.

\subsubsection{BR2 SAT: The modified method of Bassi and Rebay}
A stabilized version of the BR2 method \cite{bassi1997highbr2} for implementation with the narrow-stencil second-derivative SBP operators can be obtained by choosing
\begin{equation} 
\begin{aligned}
	\T_{\gamma k}^{(1)}&=\T_{\gamma v}^{(1)}=\frac{1}{2\alpha_{\gamma k}}\R_{\gamma k}\Lambda_{k}\V_{k}^{+}\Lambda_{k}\R_{\gamma k}^{T}+\frac{1}{2\alpha_{\gamma v}}\R_{\gamma v}\Lambda_{v}\V_{v}^{+}\Lambda_{v}\R_{\gamma v}^T,
	\\
	-\T_{\gamma k}^{(2)}&=-\T_{\gamma v}^{(2)}=\T_{\gamma k}^{(3)}=\T_{\gamma k}^{(3)} = \frac{1}{2},
	\\
	\T_{\gamma k}^{(D)} &=  \frac{2}{\alpha_{\gamma k}} \R_{\gamma k}\Lambda_{k} \V_k^{+}\Lambda_{k} \R_{\gamma k}^T,
	\\
	\T_{\gamma k}^{(4)} &= \T_{\gamma v}^{(4)} = 0.	
\end{aligned}
\end{equation}
The general form of the BR2 SAT was first proposed in \cite{yan2018interior} by discretizing the primal formulation of the DG method using SBP operators. It is straightforward to show that the BR2 SAT coefficients satisfy the adjoint consistency conditions in \cref{eq:adjoint consistency conditions} and the stability requirements in \cref{thm:stability}.

\subsubsection{LDG SAT: The local discontinuous Galerkin method}
We determine the SAT coefficients corresponding to the LDG method \cite{shu2001different} by discretizing the primal LDG formulation of the diffusion problem (see, \eg, \cite{arnold2002unified,peraire2008compact} for the primal LDG formulation). Similar analysis with the wide-stencil second-derivative SBP operators can be found in \cite{carpenter2010revisiting,gong2011interface,berg2012superconvergent,worku2020simultaneous}. Stable LDG SAT coefficients (with no mesh dependent parameter for stabilization) for discretizations with the narrow-stencil second-derivative SBP operators are given by
\begin{align}\label{eq:LDG}
		\T_{\gamma k}^{(1)}&=\T_{\gamma v}^{(1)}= \T_{\gamma k}^{(D)} = \frac{2}{\alpha_{\gamma k}}\R_{\gamma k}\Lambda_{k}\V_{k}^{+}\Lambda_{k}\R_{\gamma k}^{T},
		&&
		-\T_{\gamma k}^{(2)}=\T_{\gamma v}^{(3)}= 1,
		&&
		\T_{\gamma k}^{(3)}=\T_{\gamma v}^{(2)} = \T_{\gamma k}^{(4)} = \T_{\gamma v}^{(4)} = 0.
\end{align}
Clearly, the LDG SAT coefficients in \cref{eq:LDG} satisfy both the adjoint consistency conditions in \cref{eq:adjoint consistency conditions} and the stability demands in \cref{thm:stability}.

\begin{remark}
	The LDG SAT coefficients presented in \cref{eq:LDG} are obtained by using a switch function value of $ 1/2 $ and a global vector pointing to the positive $ x $-axis. We refer the reader to \cite{arnold2002unified,peraire2008compact,worku2020simultaneous} for details regarding the switch function and to \cite{sherwin20062d} for a discussion on the need to use a global vector.
\end{remark}

\begin{remark}
	In one space dimension the LDG and CDG fluxes are identical \cite{peraire2008compact}; therefore, the SAT coefficients presented in \cref{eq:LDG} define the CDG SAT as well.
\end{remark}

\subsubsection{BO SAT: The Baumann-Oden method}
The SAT coefficients corresponding to the BO method \cite{baumann1999discontinuous} do not satisfy the adjoint consistency conditions in \cref{eq:adjoint consistency conditions}; hence, the energy stability requirements in \cref{thm:stability} do not apply. The BO SAT coefficients for implementation with narrow-stencil second-derivative operators are given by
\begin{align}\label{eq:BO}
	\T_{\gamma k}^{(2)}&=\T_{\gamma v}^{(2)} = \T_{\gamma k}^{(3)} = \T_{\gamma v}^{(3)}= \frac{1}{2},
	&&
	\T_{\gamma k}^{(D)} = \frac{2}{\alpha_{\gamma k}}\R_{\gamma k}\Lambda_{k}\V_{k}^{+}\Lambda_{k}\R_{\gamma k}^{T},
	&&
	\T_{\gamma k}^{(1)}=\T_{\gamma v}^{(1)}=\T_{\gamma k}^{(4)} = \T_{\gamma v}^{(4)} = 0.
\end{align}
Except for the conditions on $ \T_{\gamma k}^{(D)} $, the stability conditions for the narrow- and wide-stencil SBP operators are the same when the BO SAT is used. Stability analysis for discretizations with wide-stencil SBP operators and the BO SAT can be found in \cite{carpenter2010revisiting,gong2011interface,worku2020simultaneous}. Note that for the BO SAT coefficients in \cref{eq:BO}, we have $ \T_{\gamma k}^{(1)}=\T_{\gamma v}^{(1)}= \sigma_{k} = \sigma_{v} = 0 $ in the matrix $ \A $ given in \cref{eq:matrices for stability}, and thus $ \A $ is positive semidefinite. The stability analyses for the second and third matrices in \cref{eq:matrices for stability} remain the same as those presented in the proof of \cref{thm:stability}. The BO SAT satisfies the conditions for conservation \cref{eq:conservation conditions}; hence, it leads to a conservative and stable but not adjoint consistent scheme.  

\subsubsection{CNG SAT: The Carpenter-Nordstr{\"o}m-Gottlieb method}
A version of the CNG SAT \cite{carpenter1999stable} that leads to a stable discretization when implemented with narrow-stencil second-derivative SBP operators has the SAT coefficients
\begin{equation}\label{eq:CNG}
	\begin{aligned} 
		\T_{\gamma k}^{(1)}&=\T_{\gamma v}^{(1)}=\frac{1}{8\alpha_{\gamma k}}\R_{\gamma k}\Lambda_{k}\V_{k}^{+}\Lambda_{k}\R_{\gamma k}^{T}+\frac{1}{8\alpha_{\gamma v}}\R_{\gamma v}\Lambda_{v}\V_{v}^{+}\Lambda_{v}\R_{\gamma v}^{T},
		\\
		\T_{\gamma k}^{(2)}&=\T_{\gamma v}^{(2)} =\T_{\gamma k}^{(4)} = \T_{\gamma v}^{(4)} = 0,
		\\
		\T_{\gamma k}^{(3)} &= \T_{\gamma v}^{(3)}= \frac{1}{2},
		\\
		\T_{\gamma k}^{(D)} &= \frac{2}{\alpha_{\gamma k}}\R_{\gamma k}\Lambda_{k}\V_{k}^{+}\Lambda_{k}\R_{\gamma k}^{T}.	
	\end{aligned}
\end{equation}
Clearly, the coefficients in \cref{eq:CNG} satisfy all the conditions in \cref{eq:adjoint consistency conditions} except the second one. Therefore, the CNG SAT leads to conservative but adjoint inconsistent schemes. The stability analyses of the second and third matrices in \cref{eq:matrices for stability} are the same as those presented in the proof of \cref{thm:stability}. The positive semidefiniteness of the matrix $ \A $ in \cref{eq:matrices for stability} requires all the conditions in \cref{eq:positive semidefiniteness} to be satisfied. Substituting the CNG SAT coefficients in $ \A $, we see that $ \A_{22}\succeq 0 $, and it can be shown that $ \left(\I-\A_{22}\A_{22}^{+}\right)\A_{21} = 0 $. Hence, it only remains to find conditions such that $ \A_{11}-\A_{12}\A_{22}^{+}\A_{21} \succeq 0 $, which, after simplification, yields
\begin{equation}\label{eq:CNG proof}
		\left[\begin{array}{cc}
		1 & -1\\
		-1 & 1
		\end{array}\right]\otimes\left[2\T_{\gamma k}^{(1)}-\left(\frac{1}{4\alpha_{\gamma k}}\C_{\gamma k}\V_{k}^{+}\C_{\gamma k}^{T}+\frac{1}{4\alpha_{\gamma k}}\C_{\gamma v}\V_{v}^{+}\C_{\gamma v}^{T}\right)\right] \succeq 0.
\end{equation}
The inequality in \cref{eq:CNG proof} is satisfied by the $ \T_{\gamma k}^{(1)} = \T_{\gamma v}^{(1)} $ coefficient given in \cref{eq:CNG}.

\section{Numerical Results}\label{sec:numerical results}
We consider the one-dimensional Poisson problem with Dirichlet and Neumann boundary conditions, 
\begin{align} \label{eq:Poisson problem}
	-\frac{\partial^2{{\cal U}}}{\partial x^2} & ={\cal F}\quad{\rm in}\;\Omega=\left[0,1\right], &&  {\cal U}\big|_{x=0} ={\cal U}_{D}, && \frac{\partial{\cal U}}{\partial x}\bigg|_{x=1}={\cal U}_{N}.
\end{align}
We use the method of manufactured solution and let $ {\cal U}=\cos(30x) $ as in \cite{eriksson2018dual}; thus, $ {\cal F}=30^{2}\cos(30x) $. We also consider a compatible linear functional given by 
\begin{equation}
	{\cal I}\left({\cal U}\right)=\int_{0}^{1}\cos^{2}\left(30x\right){\rm d}\Omega+\frac{1}{30^{2}}(1-30\sin(30)-\cos(30))\cos(30).
\end{equation}
We are interested in the convergence of the solution and functional errors under mesh refinement. Figure \ref{fig:solution convergence} presents the solution convergence for discretizations with the diagonal-norm narrow-stencil CSBP operators in \cite{mattsson2004summation} and the generalized SBP operators in \cite{del2015SecondDerivative} that have an invertible $ \D_{b,k} $ matrix and satisfy the accuracy conditions, \ie, the $ p=\{2,3\} $ HGTL and $ p=2 $ HGT operators. Note that the degree four HGTL operator in \cite{del2015SecondDerivative} meets the accuracy requirements given in \cref{def:Dk,def:D2} to order $ h^5 $ only, while the degree three and four HGT operators have $ \D_{b,k} $ matrices that cannot be modified as described in \cref{sec:theoretical results} to ensure their invertibility. The interface weight parameters in the SAT coefficients are set as $ \alpha_{\gamma k} = \alpha_{\gamma v} = 1/2 $ in all cases.  

The solution error is computed as 
\[ \sqrt{\sum_{\Omega_{k}\in \fnc{T}_h}(\bm{u}_{h,k}-\bm{u}_k)^T\H_k(\bm{u}_{h,k}-\bm{u}_k)}.\] The convergence rates in \cref{fig:solution convergence} through \cref{fig:functional convergence dense} are calculated by fitting a line through the error values on the mesh resolutions indicated by the short, thin lines, and ``dof" stands for the number of degrees of freedom in the spatial discretization. Figure \ref{fig:solution convergence} shows that a solution convergence rate of $ p+2 $ is attained when order-matched narrow-stencil operators are coupled with the adjoint consistent SATs, except with the degree one SBP operators. The adjoint inconsistent SATs, BO and CNG, exhibit solution convergence rates of p+2 with all the order-matched narrow-stencil SBP operators, except the degree one and three CSBP operators which yield convergence rates of p + 1. This is consistent with the results\footnote{Although not presented here, we observe a solution convergence rate of $ p+2 $ with all the diagonal-norm narrow-stencil CSBP operators presented in \cite{del2015SecondDerivative}.} presented in \cite{del2015SecondDerivative} but somewhat surprising since the well-known even-odd convergence phenomenon (see \eg, \cite{carpenter2010revisiting,shu2001different,kirby2005selecting}) that the BO method displays is not observed with the narrow-stencil SBP operators. Numerical experiments in \cite{worku2020simultaneous} show that the BO and CNG SATs converge at rates of $ p+1 $ and $ p $ when implemented with odd and even degree multidimensional SBP operators, respectively. Indeed, this even-odd convergence phenomenon is also observed when the BO and CNG SATs are implemented with the Legendre-Gauss-Lobatto (LGL) and Legendre-Gauss (LG) wide-stencil SBP operators, as shown in \cref{fig:solution convergence BO}. However, this trend does not hold consistently with the wide-stencil CSBP and HGT operators, as convergence rates of $ p+1 $ are achieved with the $ p=4 $ operators, as depicted in \cref{fig:solution convergence3 p=4}. Therefore, it appears that the even-odd convergence property of the BO and CNG SATs is dependent on the type of SBP operator used, and it is not observed with the diagonal-norm narrow-stencil SBP operators consistently. Implementations of the diagonal-norm wide-stencil SBP operators with the BR2 and LDG SATs lead to solution convergence rates of $ p+1 $, as depicted in \cref{fig:solution convergence BR2}. Finally, \cref{fig:solution convergence dense} shows that the block-norm wide- and narrow-stencil SBP operators presented in \cite{mattsson2013solution}, denoted by CSBP2, achieve a solution convergence rate of $ 2p $ regardless of the type of SAT used. 

The functional error is calculated as $ |I_h(\bm{u}_h) - \fnc{I}(\fnc{U})|$. Figure \ref{fig:functional convergence} shows the functional convergence rates resulting from discretizations with the order-matched narrow-stencil SBP operators. As established in \cref{thm:functional superconvergence}, the figure shows that the functional superconverges at a rate of $ 2p $ when adjoint consistent SATs are used. The adjoint inconsistent SATs yield larger functional error values and lower functional convergence rates when coupled with the degree three and four diagonal-norm narrow-stencil SBP operators, but they attain lower error values and a convergence rate of $ 2p $ for degree one and two operators. As can be seen from \cref{fig:functional convergence BO}, when the adjoint inconsistent SATs are used with the diagonal-norm wide-stencil SBP operators, convergence rates of $ 2p $ are not attained, except for the $ p=1 $ case. It is also evident from \cref{fig:functional convergence,fig:functional convergence BO} that in most cases the diagonal-norm narrow-stencil operators result in lower functional error and larger functional convergence rates than the diagonal-norm wide-stencil operators when used with the adjoint inconsistent SATs. Figure \ref{fig:functional convergence BR2} shows that the adjoint consistent SATs lead to functional convergence rates of $ 2p $ when used with the diagonal-norm wide-stencil SBP operators. Comparing the results depicted in \cref{fig:functional convergence BR2,fig:functional convergence}, we can conclude that diagonal-norm narrow-stencil SBP operators do not offer better functional convergence rates than diagonal-norm wide-stencil SBP operators when coupled with the adjoint consistent SATs, which agrees with the theory. Similarly, the functional convergence rates attained with the block-norm wide- and narrow-stencil SBP operators are comparable, as depicted in \cref{fig:functional convergence dense}. Furthermore, the functional convergence rate with the block-norm SBP operators is $ 2p $ for adjoint consistent as well as adjoint inconsistent schemes, except for the degree three block-norm wide-stencil SBP operator, which exhibits a $ 2p-1 $ convergence rate when implemented with the BO and CNG SATs. In most cases, the BO and CNG SATs yield lower functional error values than the BR2 and LDG SATs when implemented with the block-norm narrow-stencil SBP operators.  
\begin{figure}[!t]
	\centering
	\subfloat[][$p=1$]
		{\includegraphics[scale=0.25]{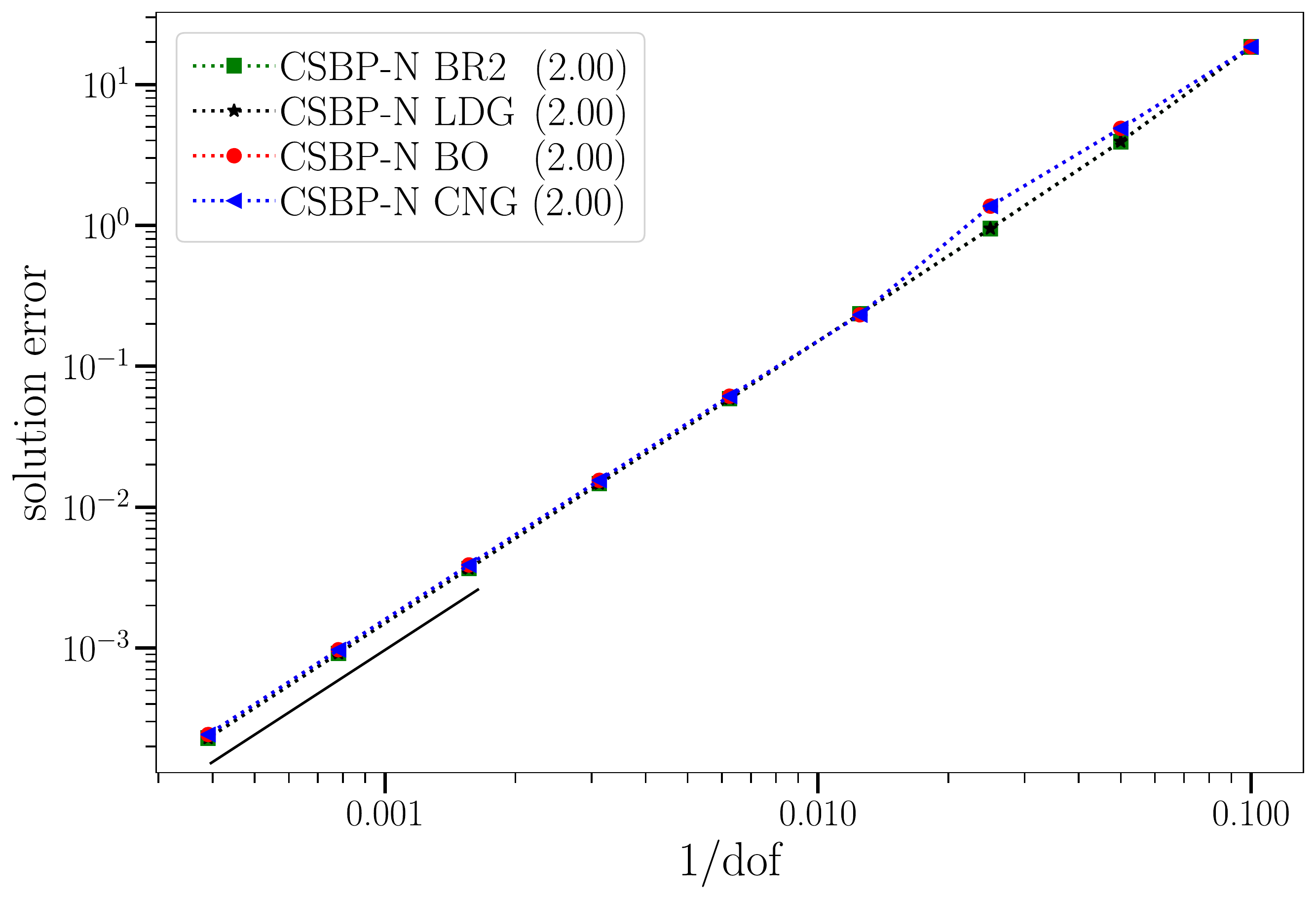} \label{fig:solution convergence p=1}}
	\hfill
	\subfloat[][$p=2$]
		{\includegraphics[scale=0.25]{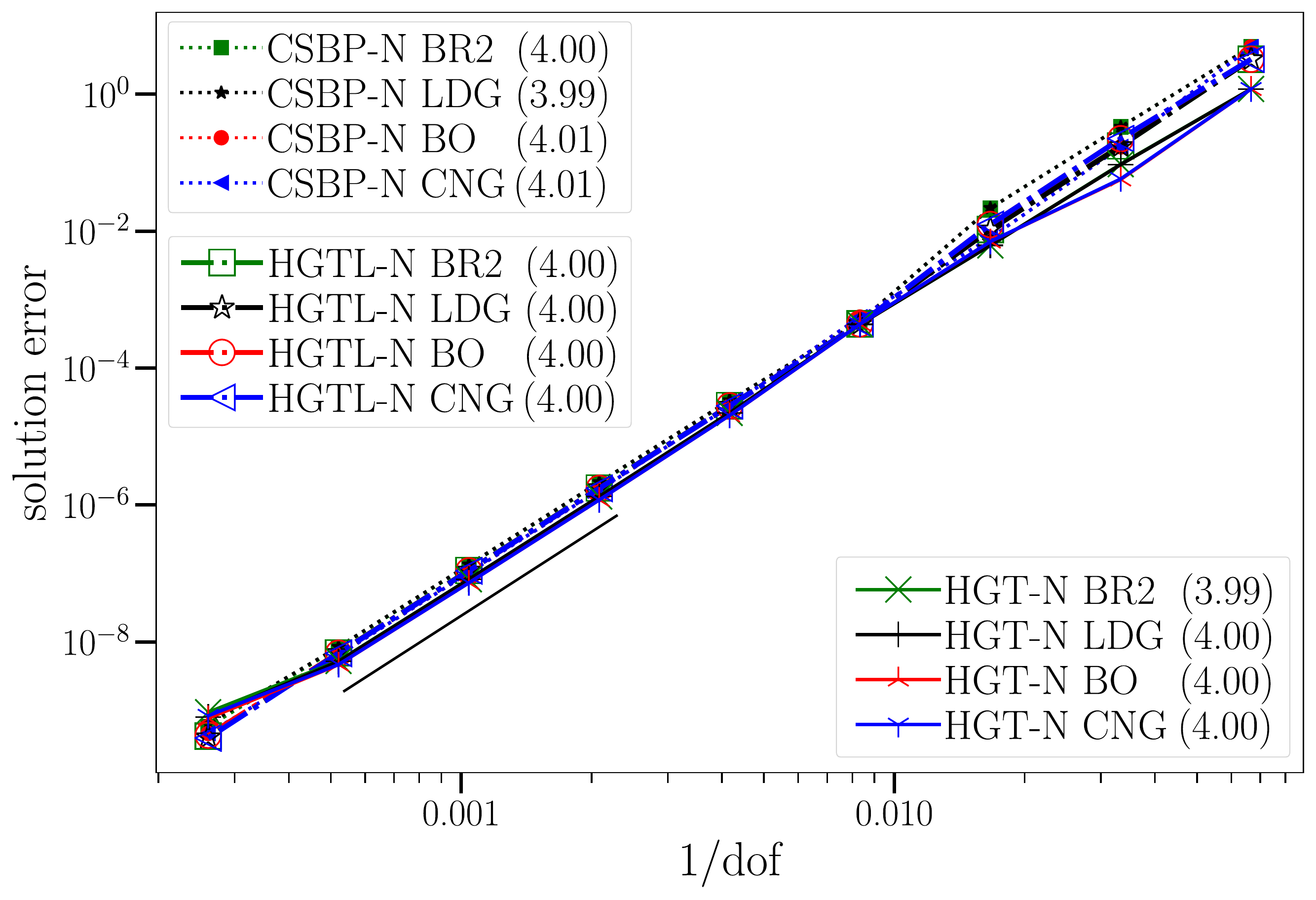} \label{fig:solution convergence p=2}}
	\\
	\subfloat[][$p=3$]
		{\includegraphics[scale=0.25]{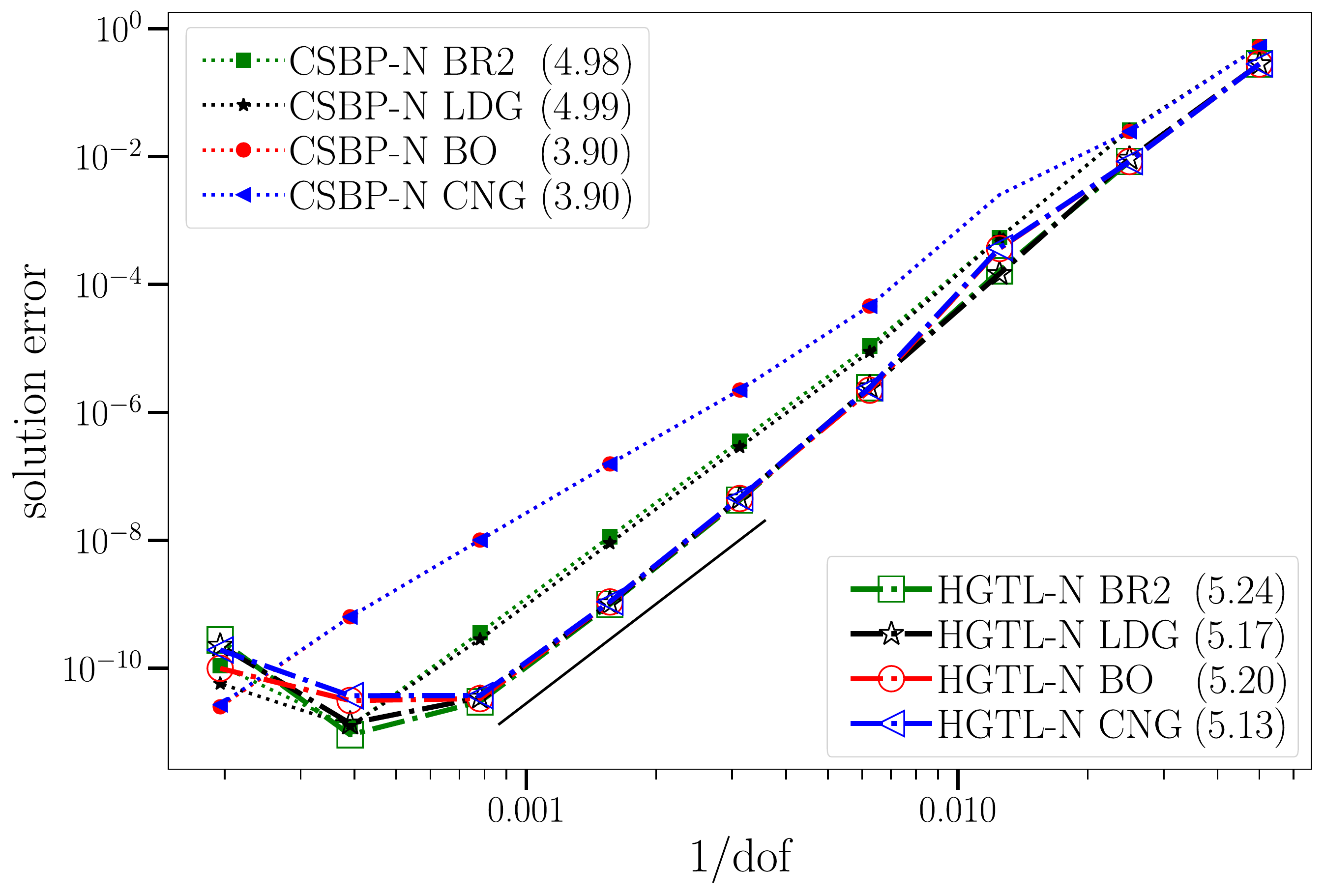} \label{fig:solution convergence p=3}}
	\hfill
	\subfloat[][$p=4$]
		{\includegraphics[scale=0.25]{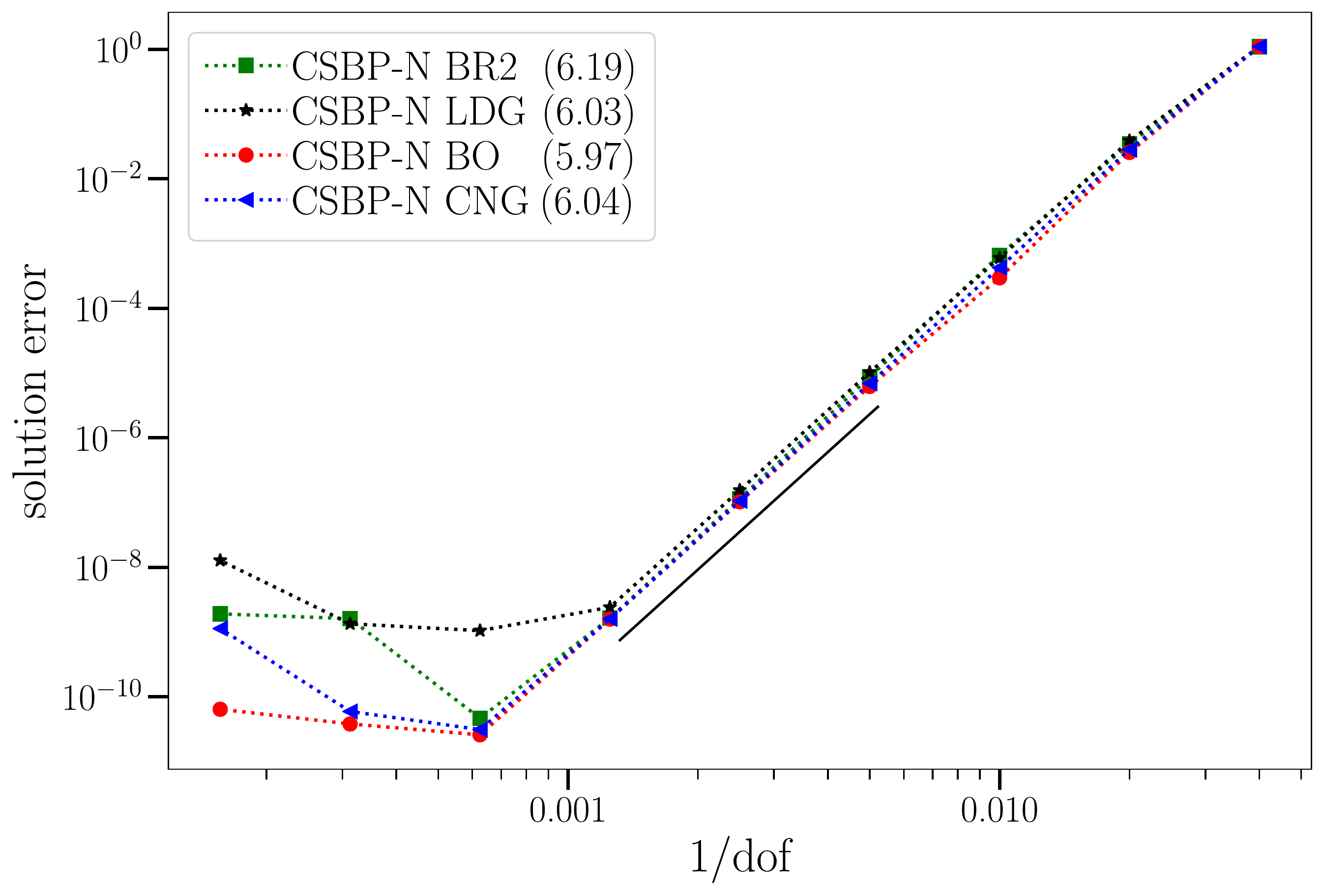} \label{fig:solution convergence p=4}}
	\caption{\label{fig:solution convergence} Solution convergence under mesh refinement. The values in parentheses are the convergence rates, and ``N" stands for narrow-stencil SBP operator.}
\end{figure}

\begin{figure}[!t]
	\centering
	\subfloat[][$p=1$]{
		\includegraphics[scale=0.25]{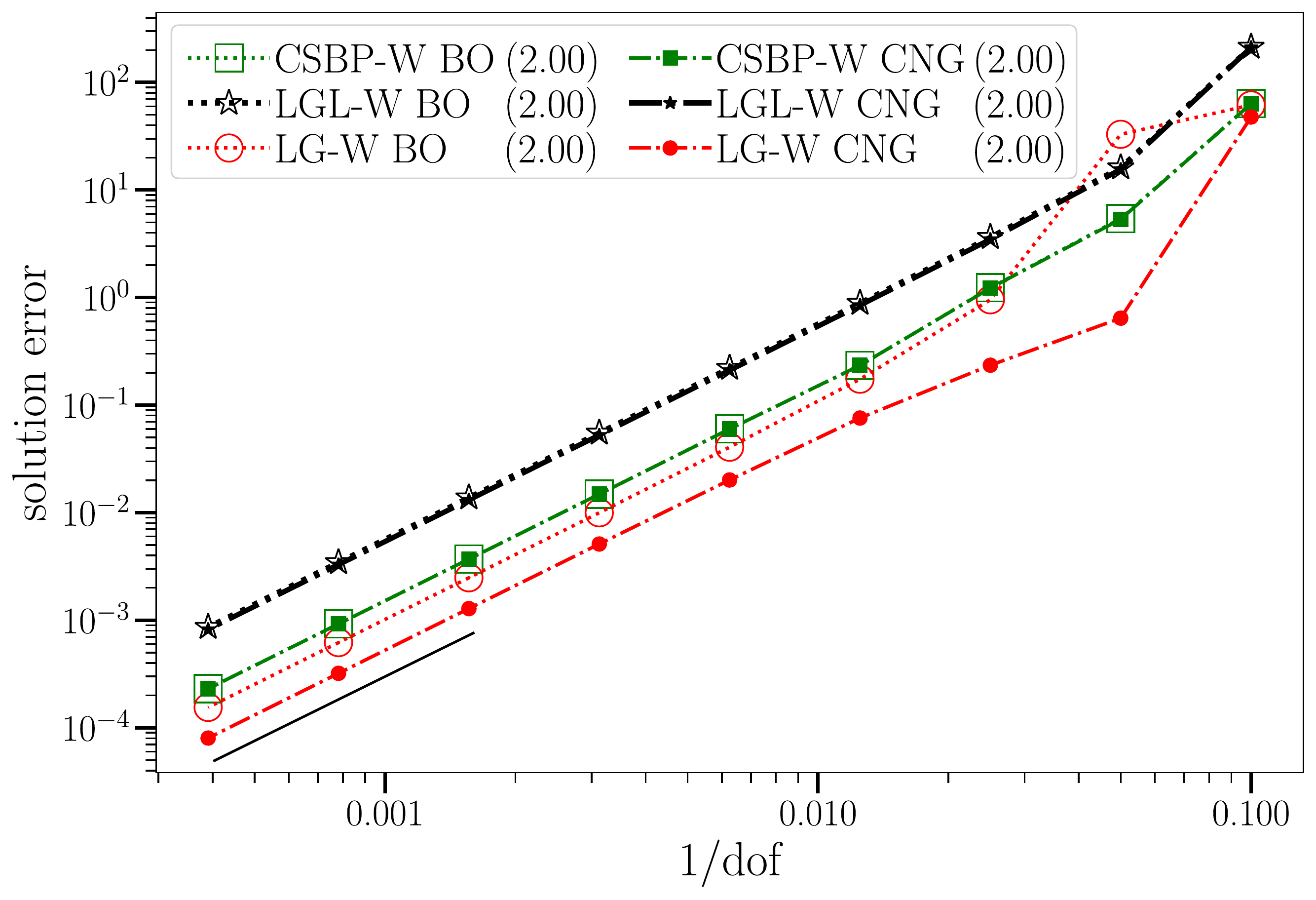}
		\label{fig:solution convergence3 p=1}}
	\hfill
	\subfloat[][$p=2$]{
		\includegraphics[scale=0.25]{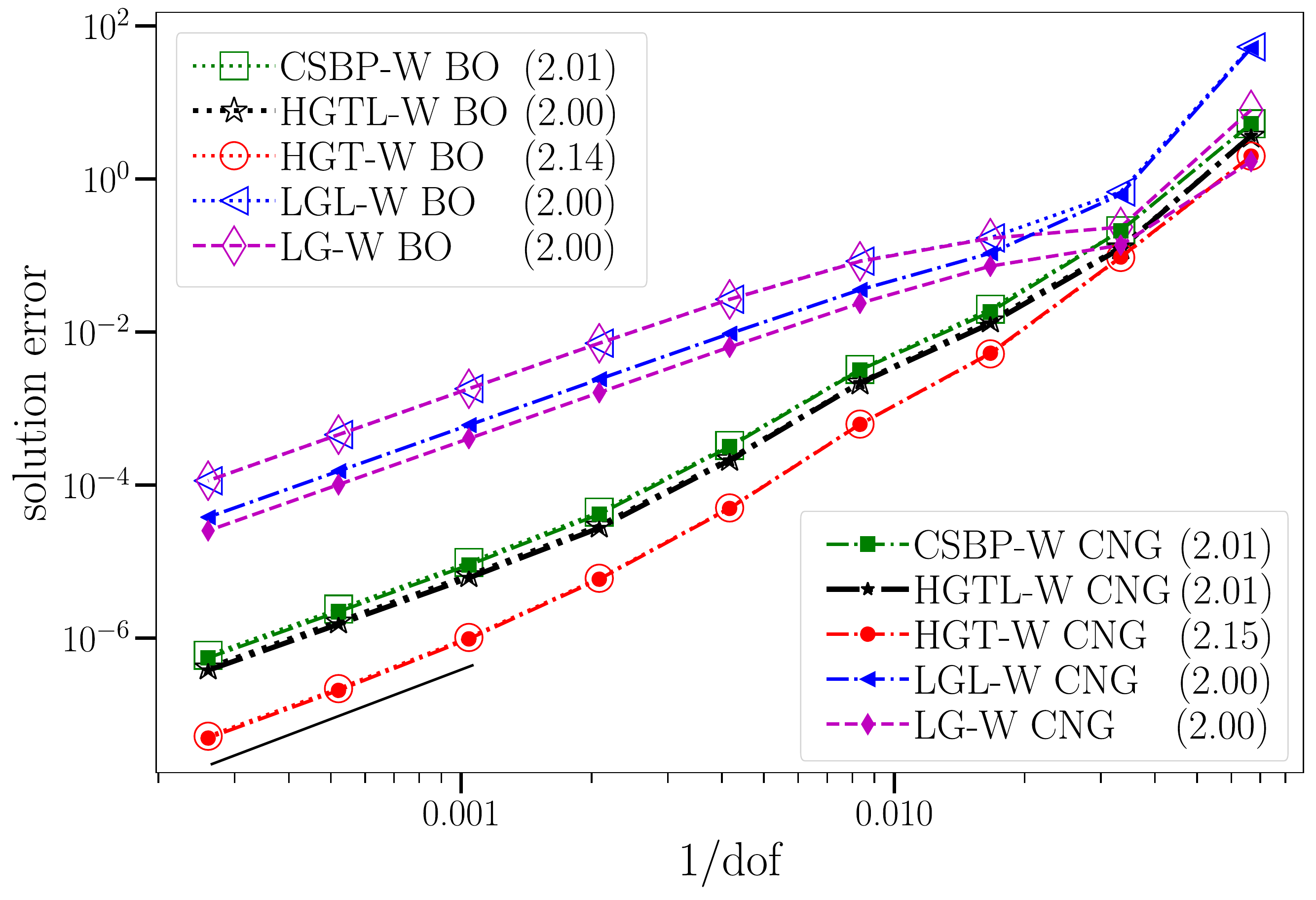}
		\label{fig:solution convergence3 p=2}}
	\\
	\subfloat[][$p=3$]{
		\includegraphics[scale=0.25]{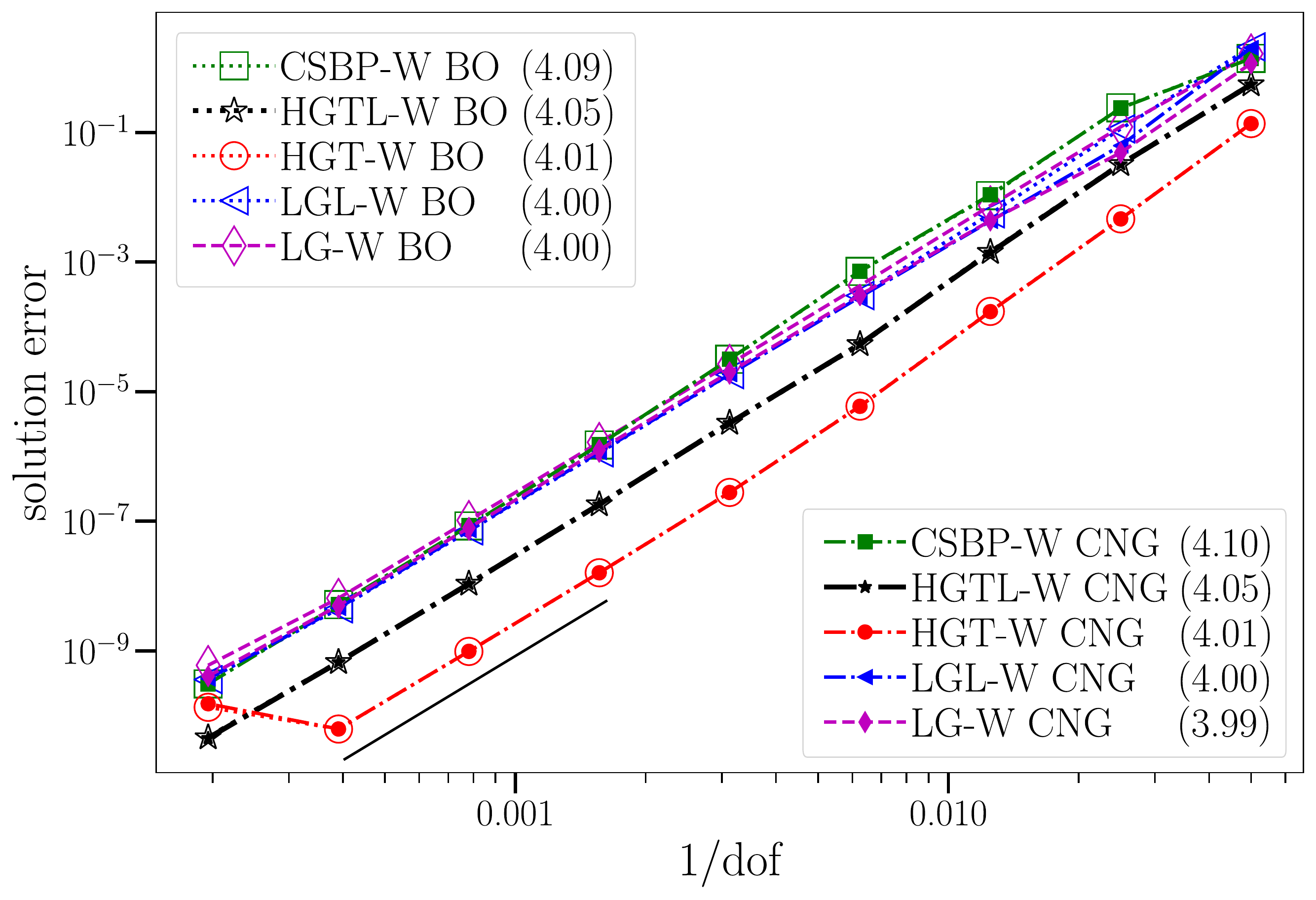}
		\label{fig:solution convergence3 p=3}}
	\hfill
	\subfloat[][$p=4$]{
		\includegraphics[scale=0.25]{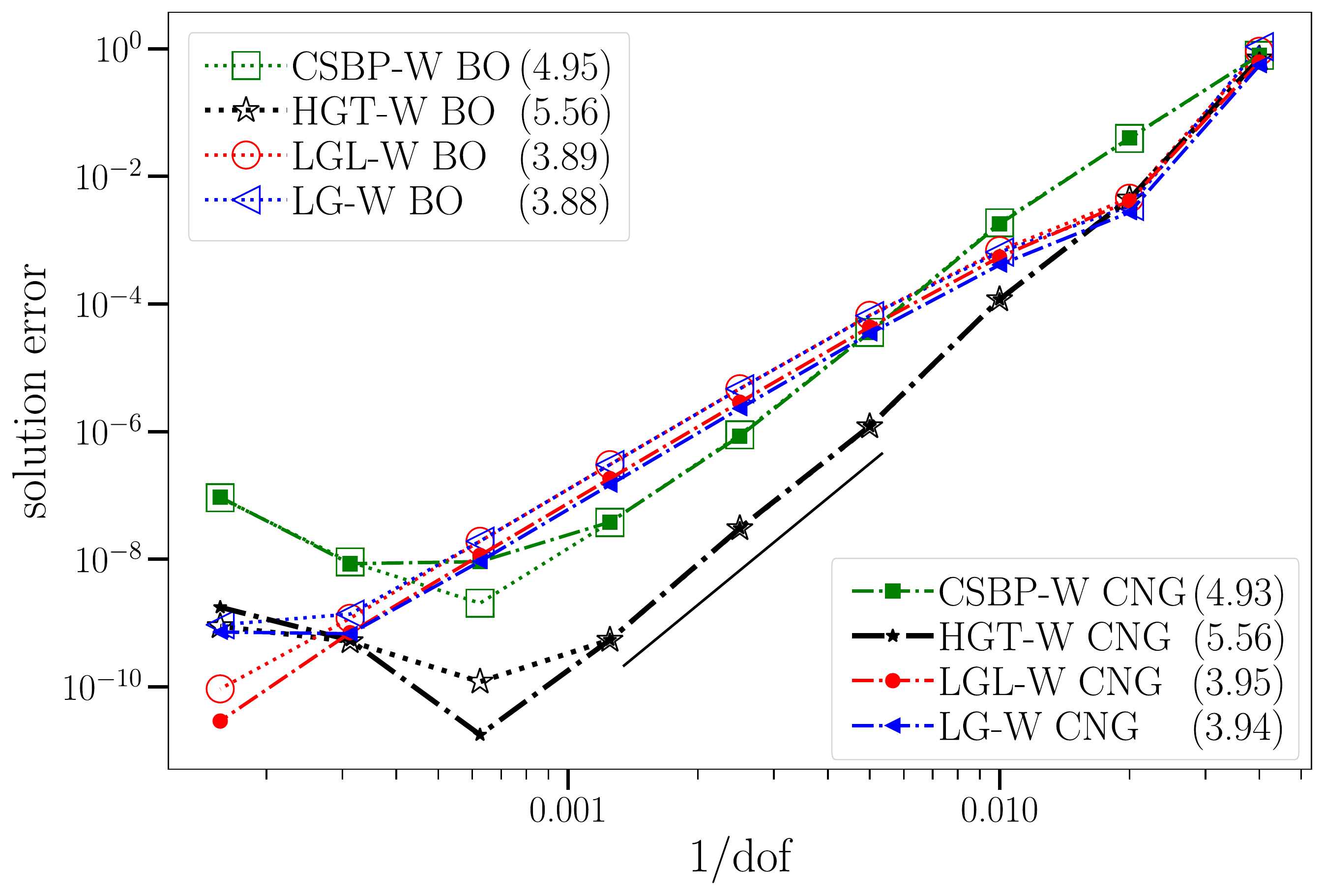}
		\label{fig:solution convergence3 p=4}}
	\caption{\label{fig:solution convergence BO} Solution convergence under mesh refinement with adjoint inconsistent SATs. The values in parentheses are the convergence rates, and ``W" stands for wide-stencil SBP operator.}
\end{figure}

\begin{figure}[!t]
	\centering
	\subfloat[][$p=1$]{
		\includegraphics[scale=0.25]{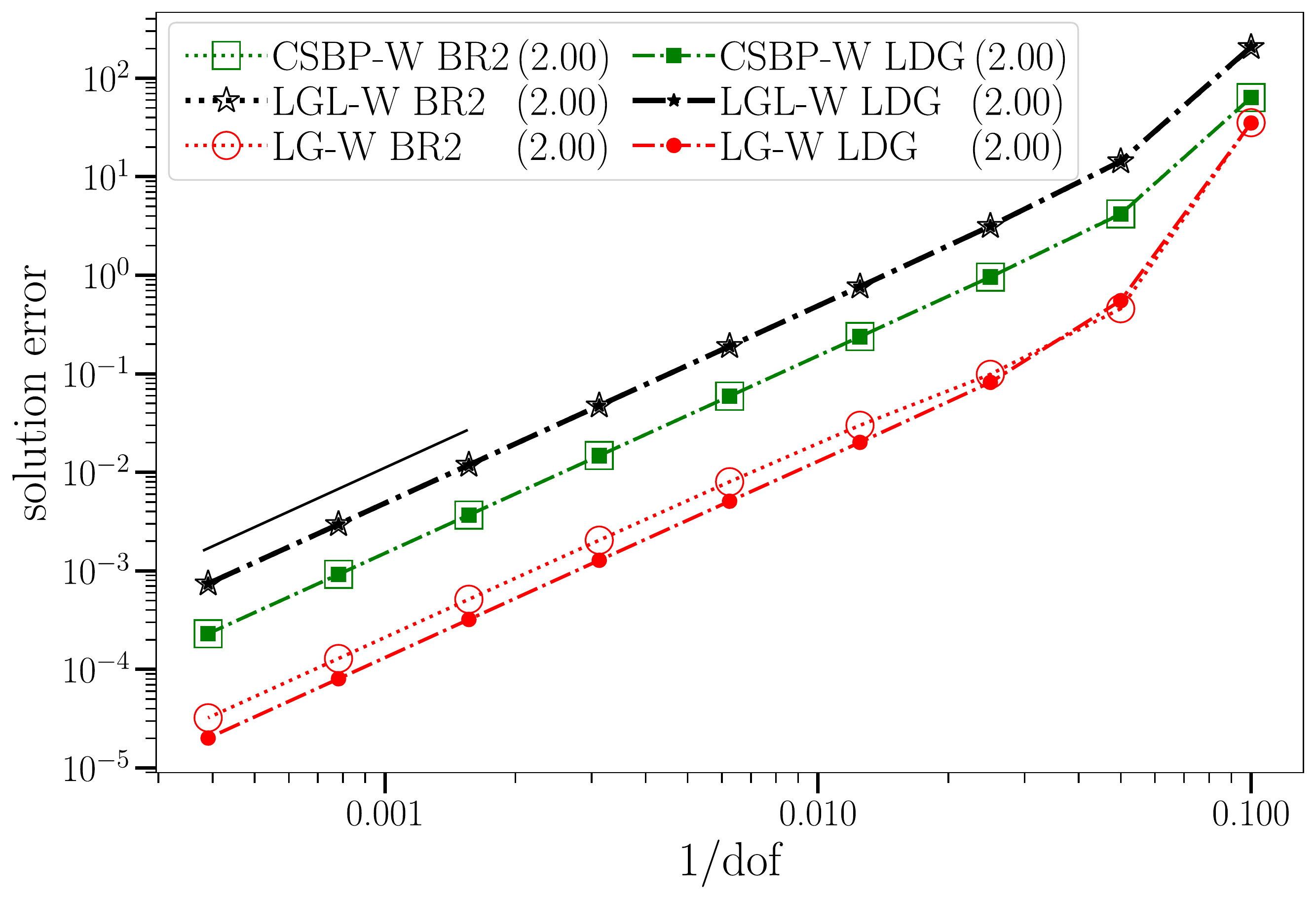}
		\label{fig:solution convergence2 p=1}}
	\hfill
	\subfloat[][$p=2$]{
		\includegraphics[scale=0.25]{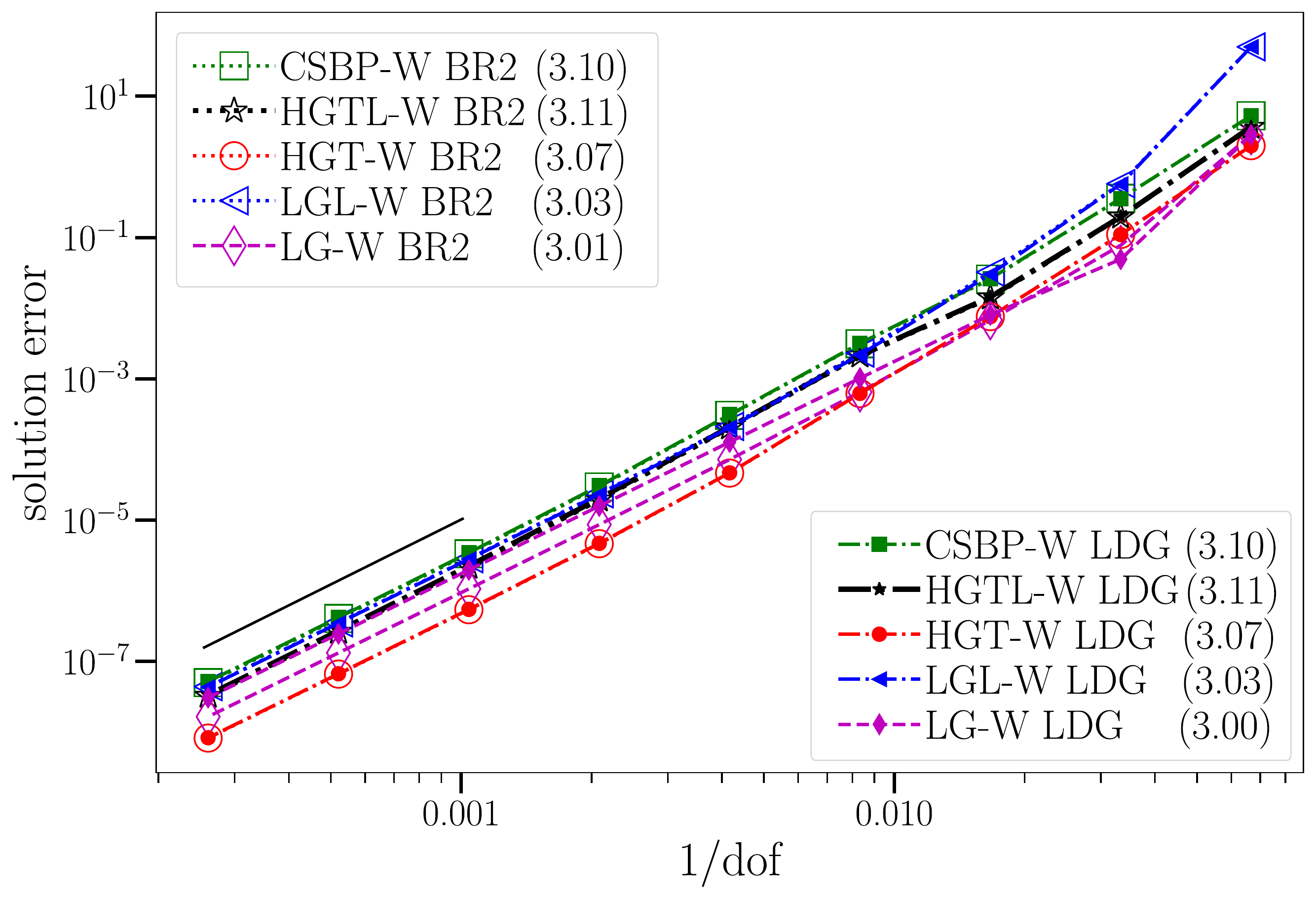}
		\label{fig:solution convergence2 p=2}}
	\\
	\subfloat[][$p=3$]{
		\includegraphics[scale=0.25]{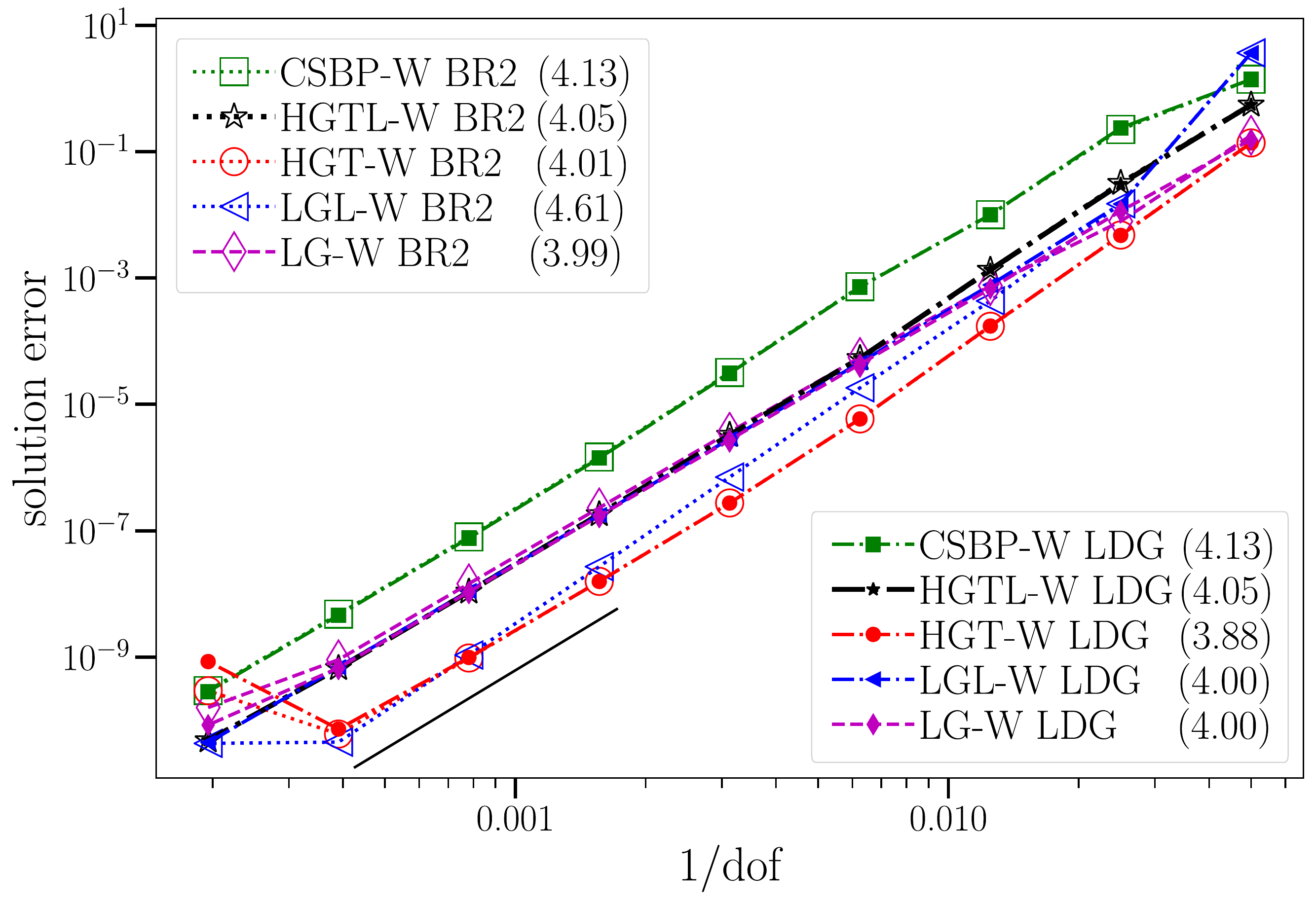}
		\label{fig:solution convergence2 p=3}}
	\hfill
	\subfloat[][$p=4$]{
		\includegraphics[scale=0.25]{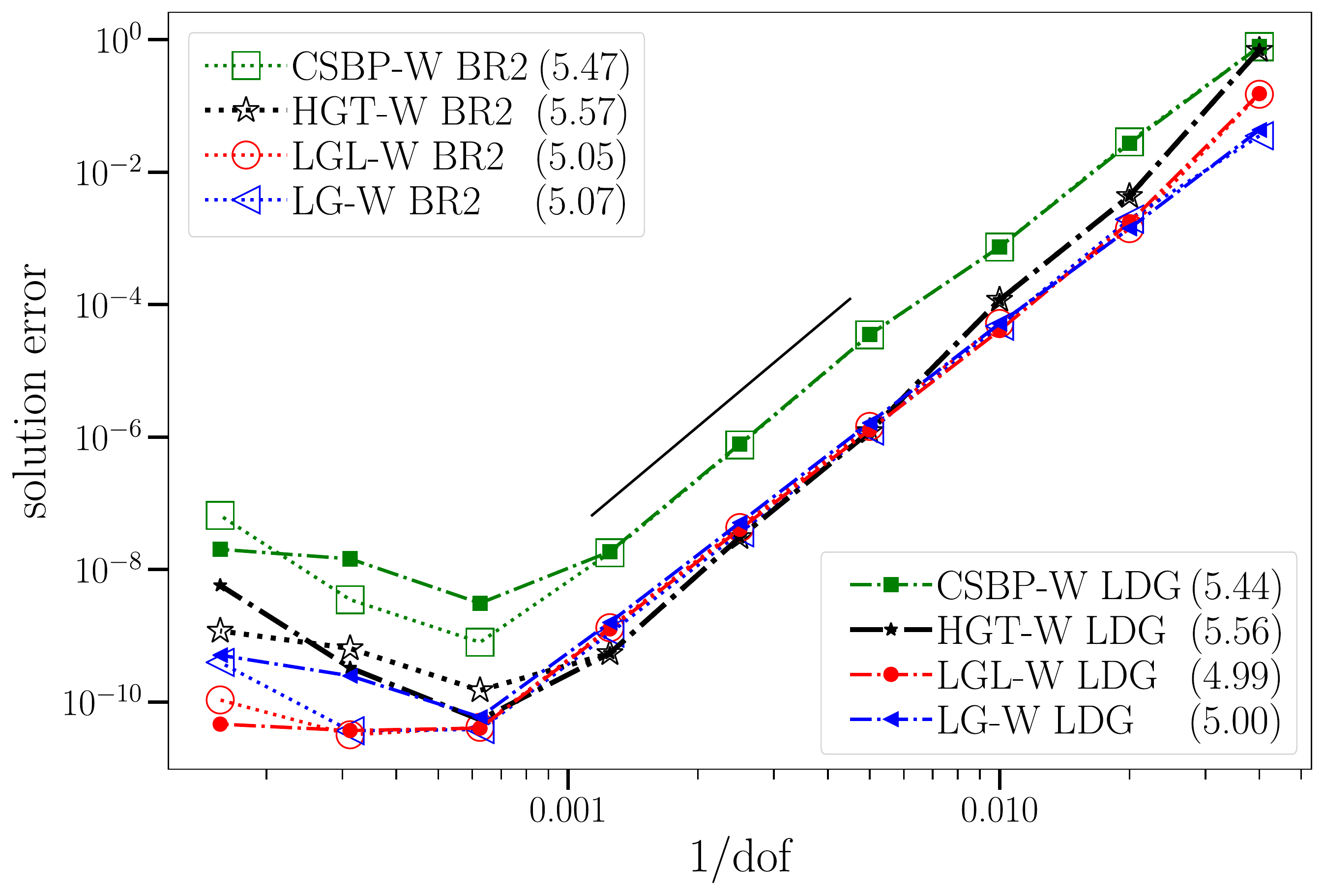}
		\label{fig:solution convergence2 p=4}}
	\caption{\label{fig:solution convergence BR2} Solution convergence under mesh refinement with adjoint consistent SATs. The values in parentheses are the convergence rates, and ``W" stands for wide-stencil SBP operator.}
\end{figure}

\begin{figure}[!t]
	\centering
	\subfloat[][$p=1$]{
		\includegraphics[scale=0.25]{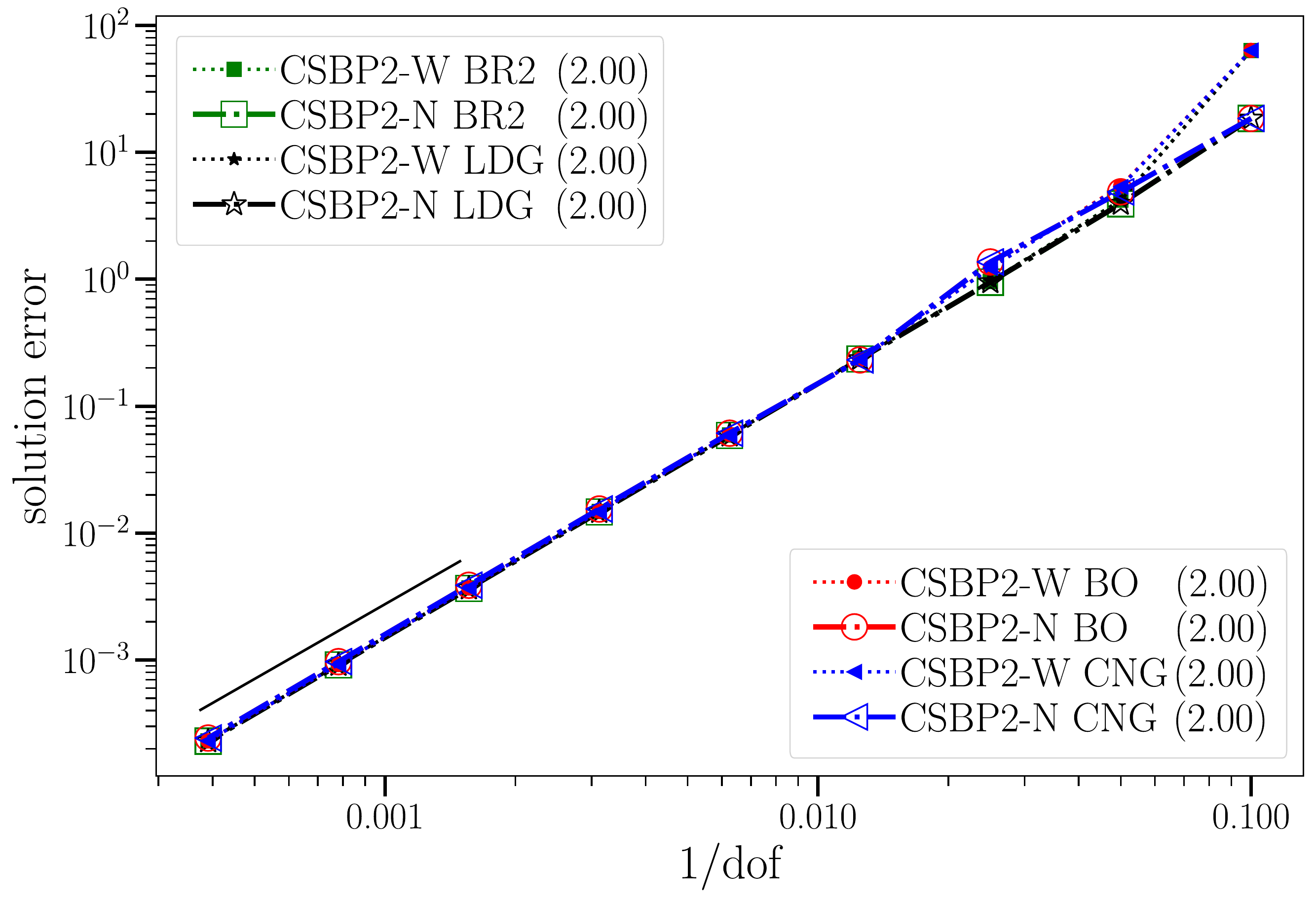}
		\label{fig:solution convergence4 p=1}}
	\hfill
	\subfloat[][$p=2$]{
		\includegraphics[scale=0.25]{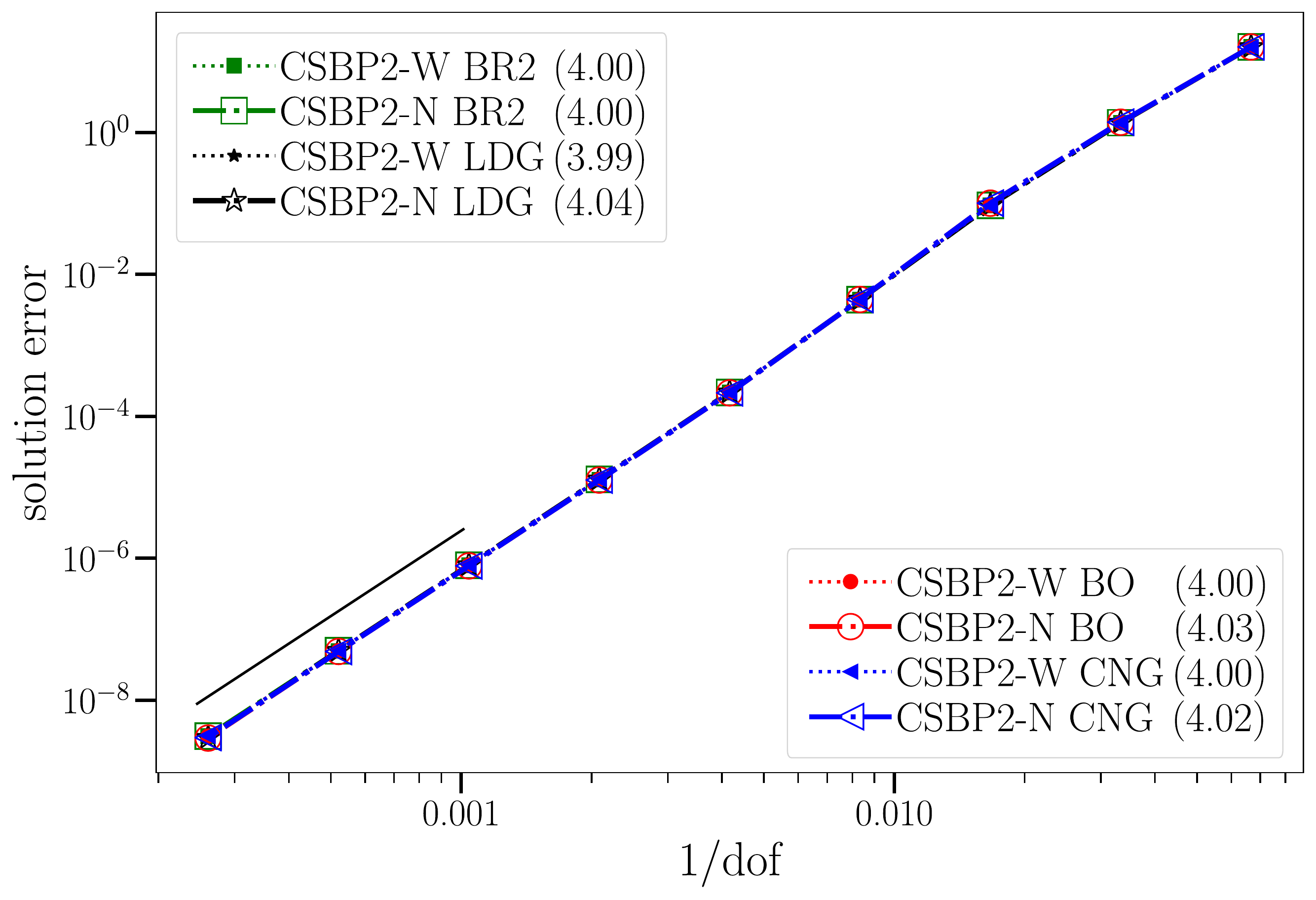}
		\label{fig:solution convergence4 p=2}}
	\\
	\subfloat[][$p=3$]{
		\includegraphics[scale=0.25]{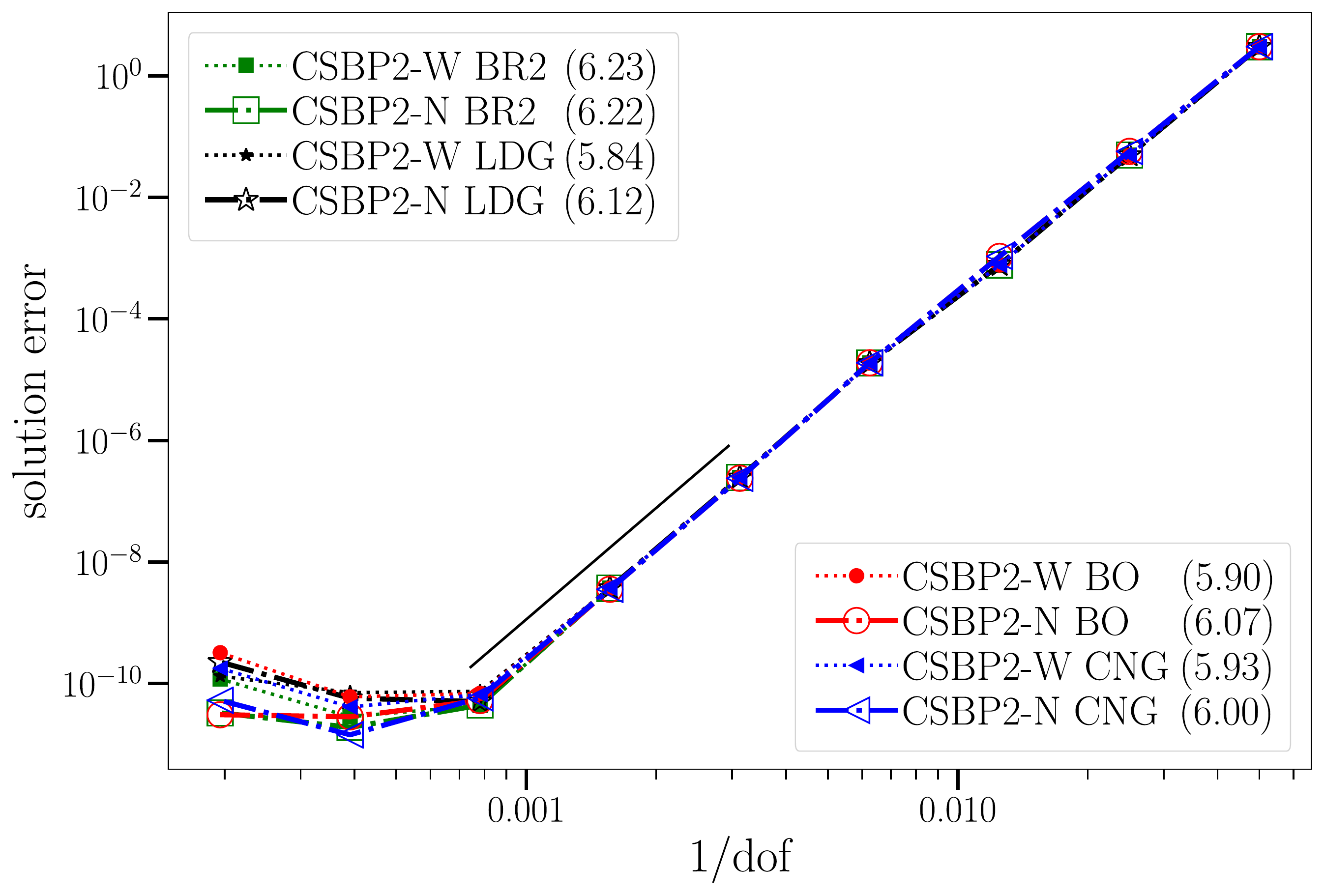}
		\label{fig:solution convergence4 p=3}}
	\hfill
	\subfloat[][$p=4$]{
		\includegraphics[scale=0.25]{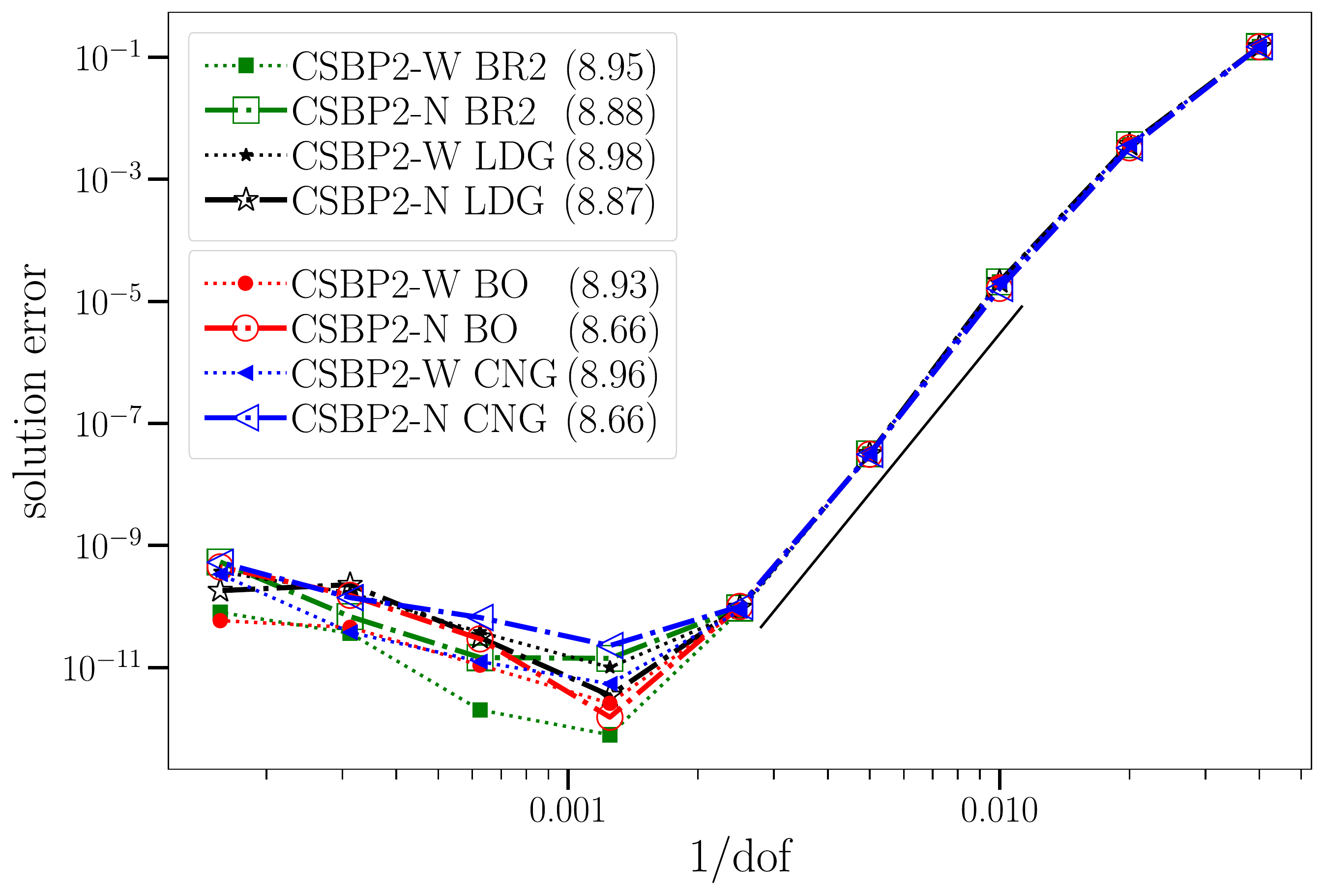}
		\label{fig:solution convergence4 p=4}}
	\caption{\label{fig:solution convergence dense} Solution convergence under mesh refinement with block-norm wide-stencil (``W") and narrow-stencil (``N") CSBP operators. The values in parentheses are the convergence rates.}
\end{figure}

\begin{figure}[!t]
	\centering
	\subfloat[][$p=1$]{
		\includegraphics[scale=0.25]{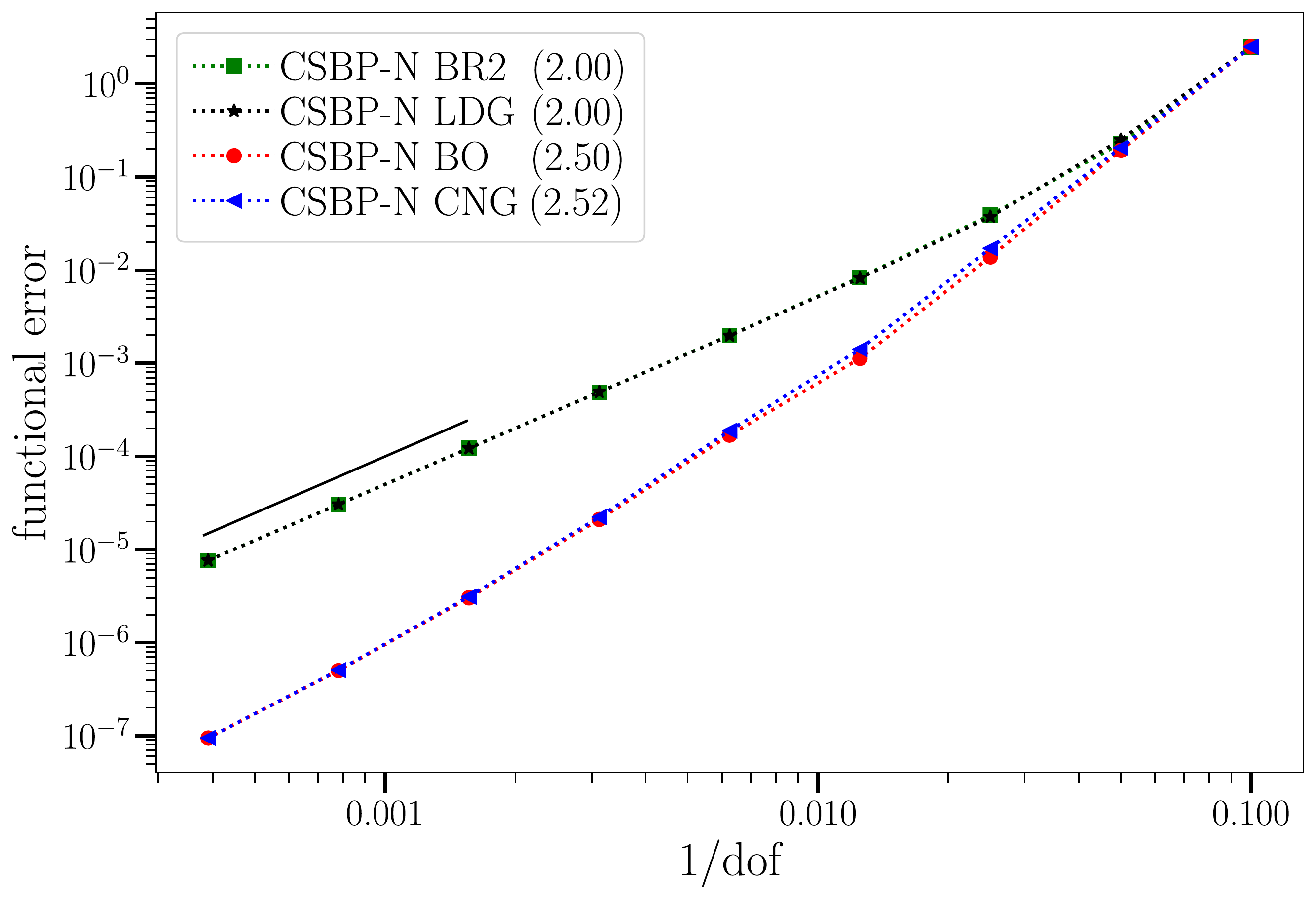}
		\label{fig:functional convergence p=1}}
	\hfill
	\subfloat[][$p=2$]{
		\includegraphics[scale=0.25]{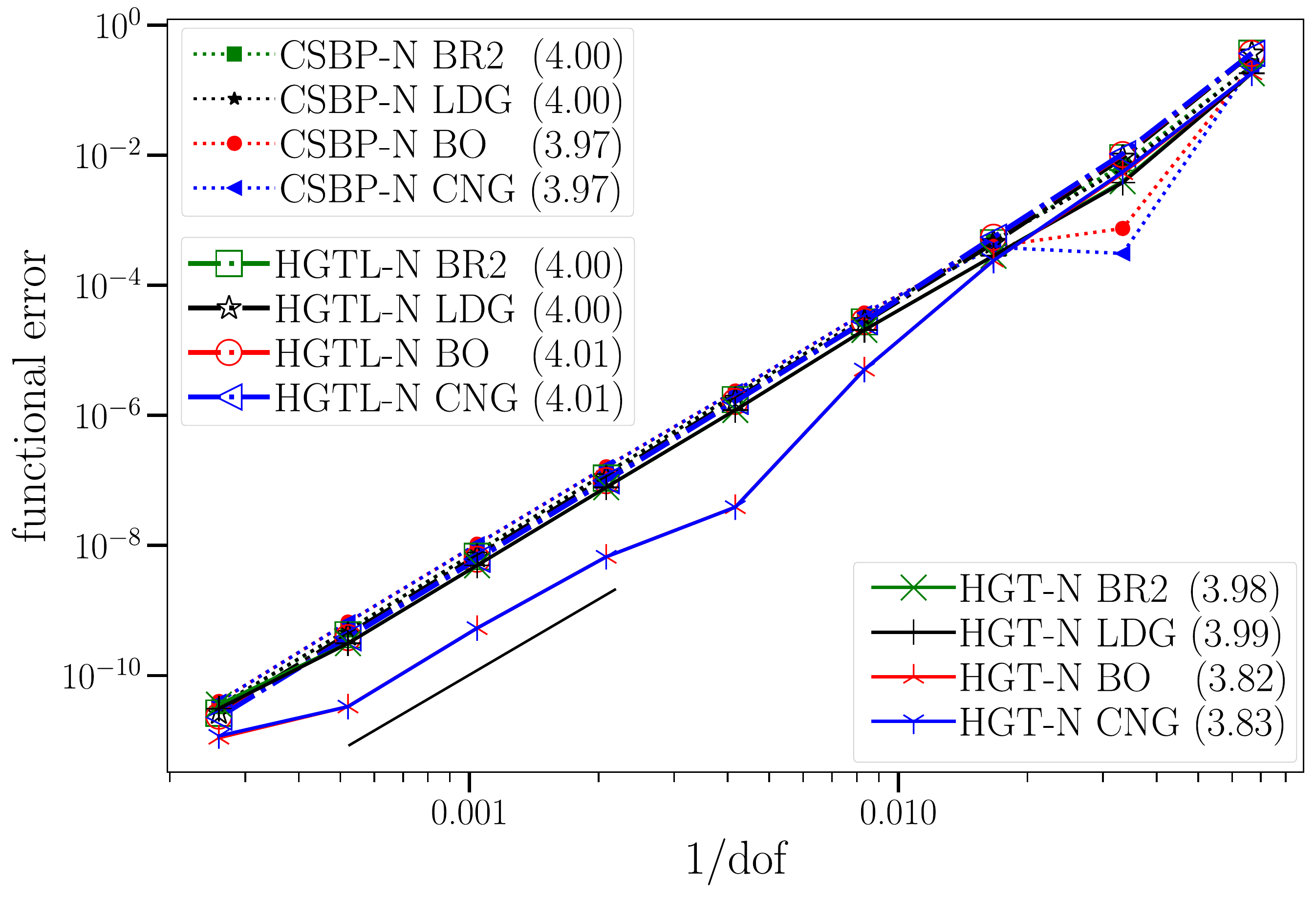}
		\label{fig:functional convergence p=2}}
	\\
	\subfloat[][$p=3$]{
		\includegraphics[scale=0.25]{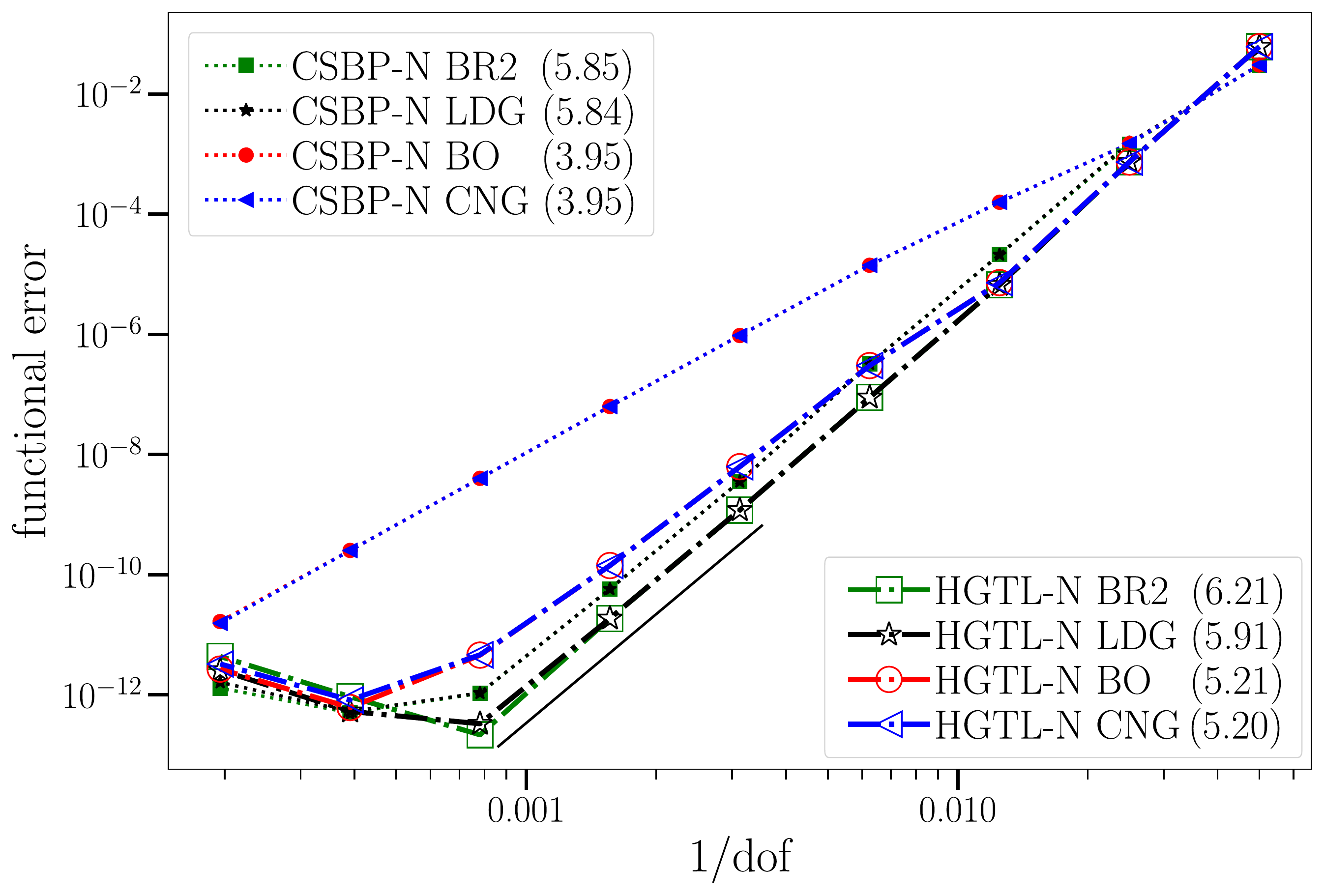}
		\label{fig:functional convergence p=3}}
	\hfill
	\subfloat[][$p=4$]{
		\includegraphics[scale=0.25]{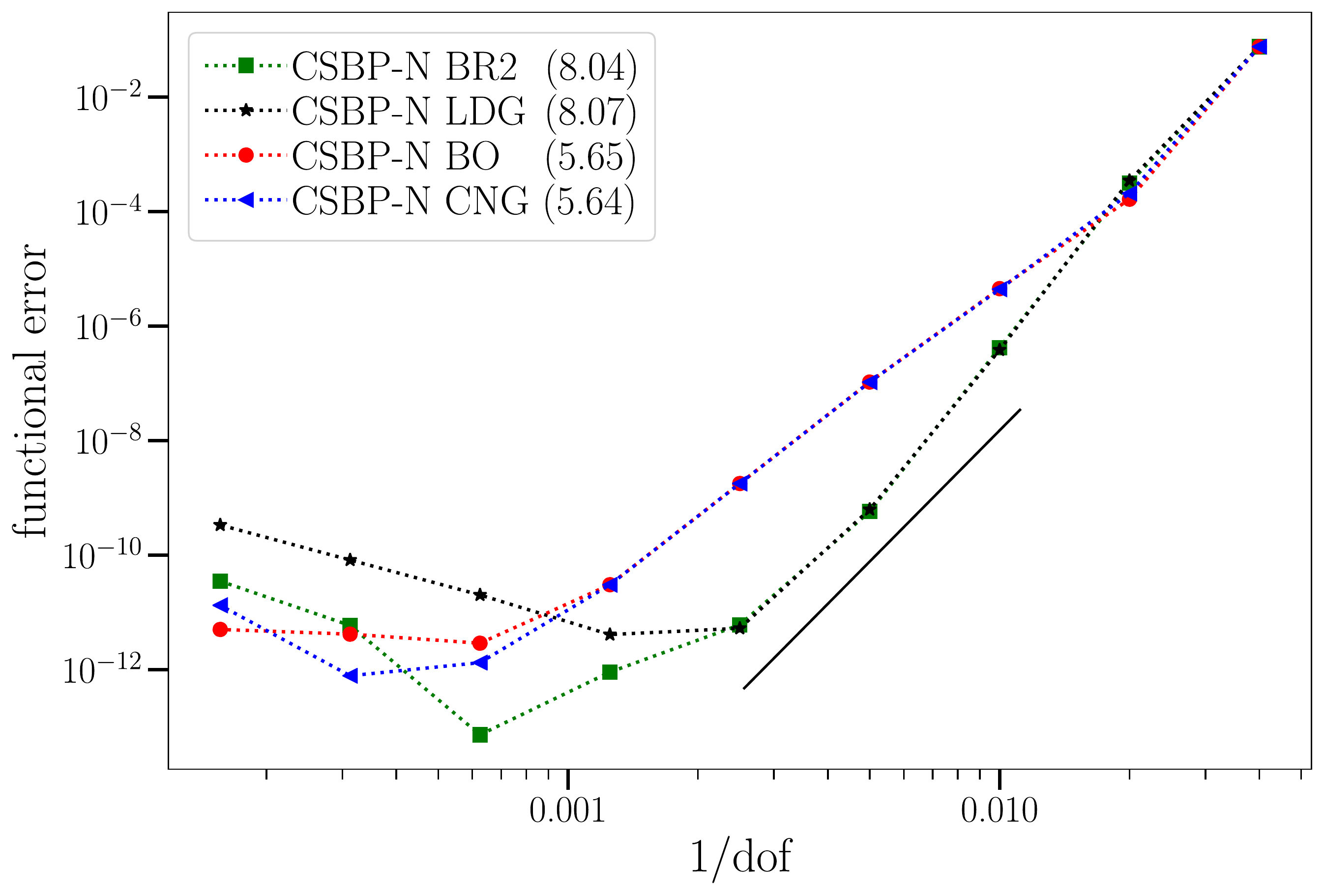}
		\label{fig:functional convergence p=4}}
	\caption{\label{fig:functional convergence} Functional convergence under mesh refinement. The values in parentheses are the convergence rates, and ``N" stands for narrow-stencil SBP operator.}
\end{figure}

\begin{figure}[!t]
	\centering
	\subfloat[][$p=1$]{
		\includegraphics[scale=0.25]{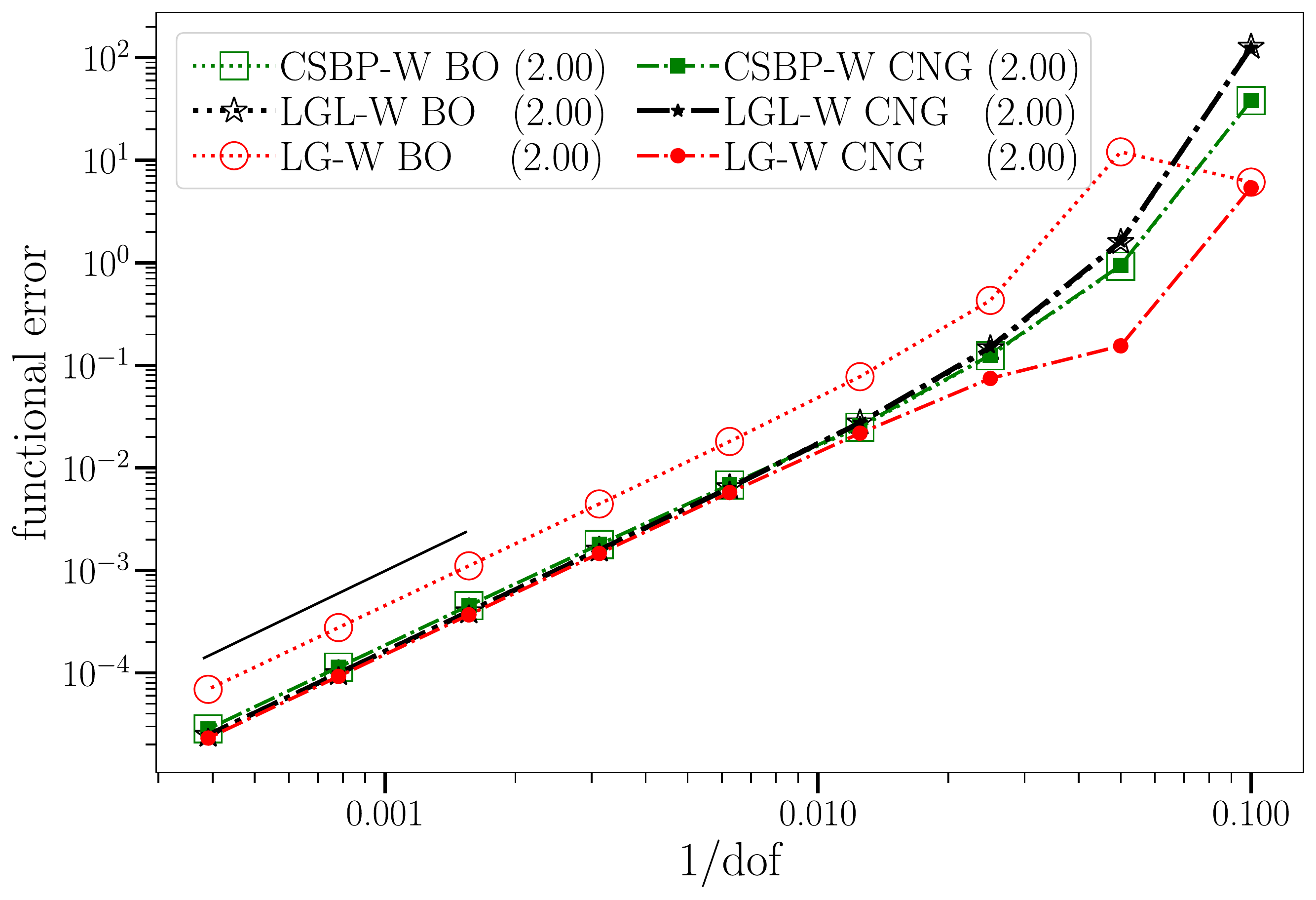}
		\label{fig:functional convergence4 p=1}}
	\hfill
	\subfloat[][$p=2$]{
		\includegraphics[scale=0.25]{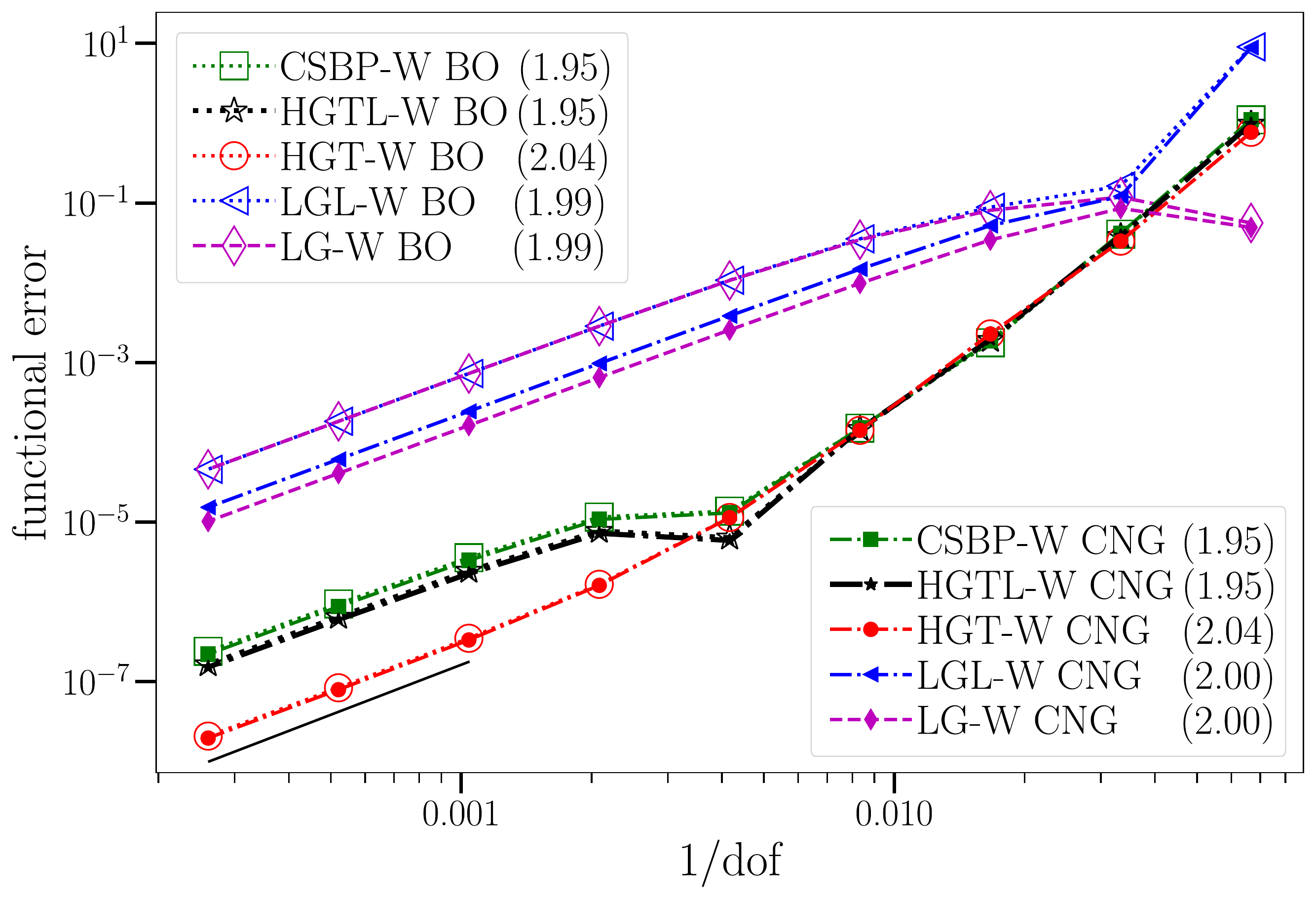}
		\label{fig:functional convergence4 p=2}}
	\\
	\subfloat[][$p=3$]{
		\includegraphics[scale=0.25]{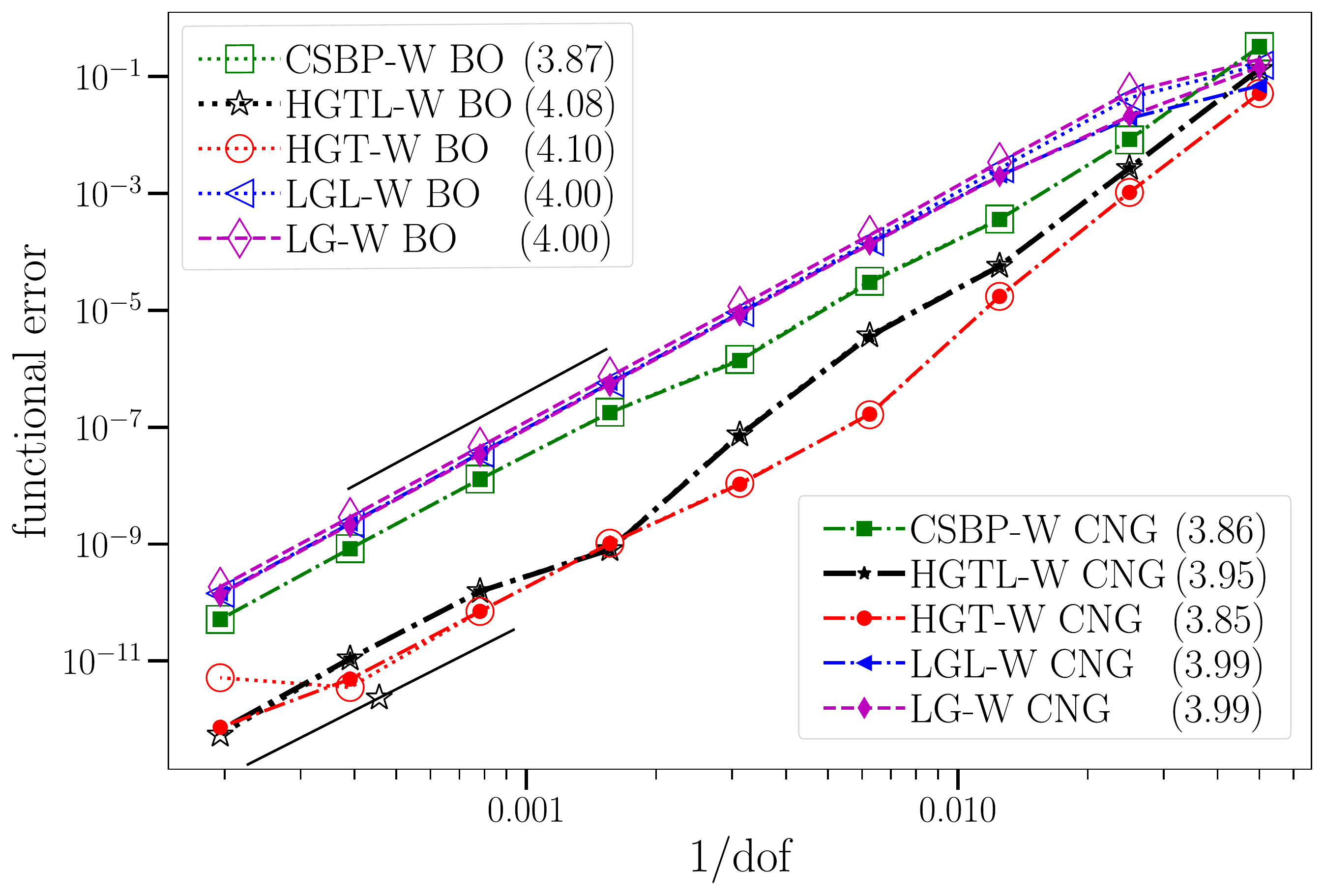}
		\label{fig:functional convergence4 p=3}}
	\hfill
	\subfloat[][$p=4$]{
		\includegraphics[scale=0.25]{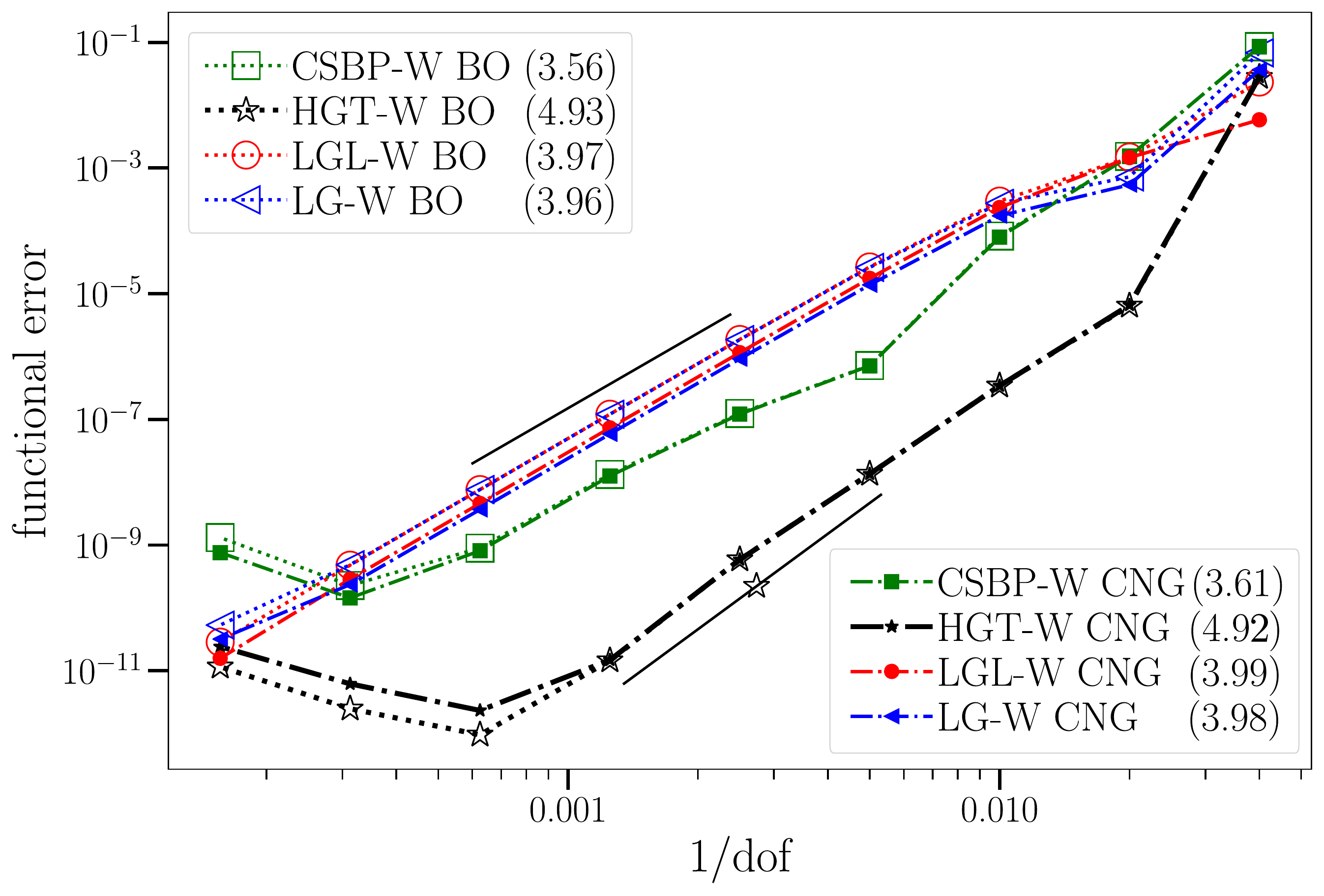}
		\label{fig:functional convergence4 p=4}}
	\caption{\label{fig:functional convergence BO} Functional convergence under mesh refinement with adjoint inconsistent SATs. The values in parentheses are the convergence rates, and ``W" stands for wide-stencil SBP operator. The convergence rates for the $ p=3 $ HGTL and $ p=4 $ HGT operators are calculated by fitting lines through the error values on the mesh resolution indicated by the short, thin lines with a star marker.}
\end{figure}

\begin{figure}[!t]
	\centering
	\subfloat[][$p=1$]{
		\includegraphics[scale=0.25]{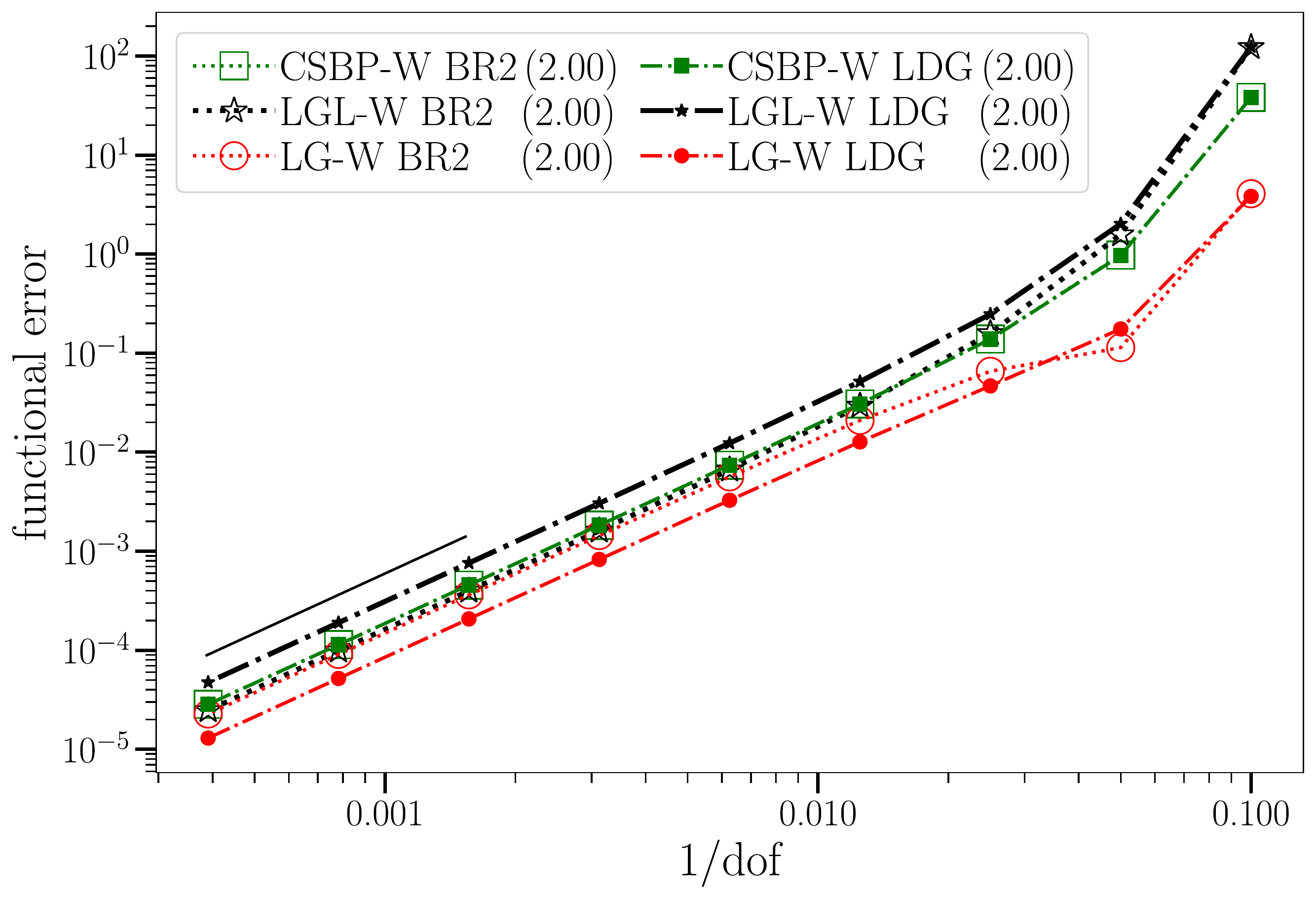}
		\label{fig:functional convergence2 p=1}}
	\hfill
	\subfloat[][$p=2$]{
		\includegraphics[scale=0.25]{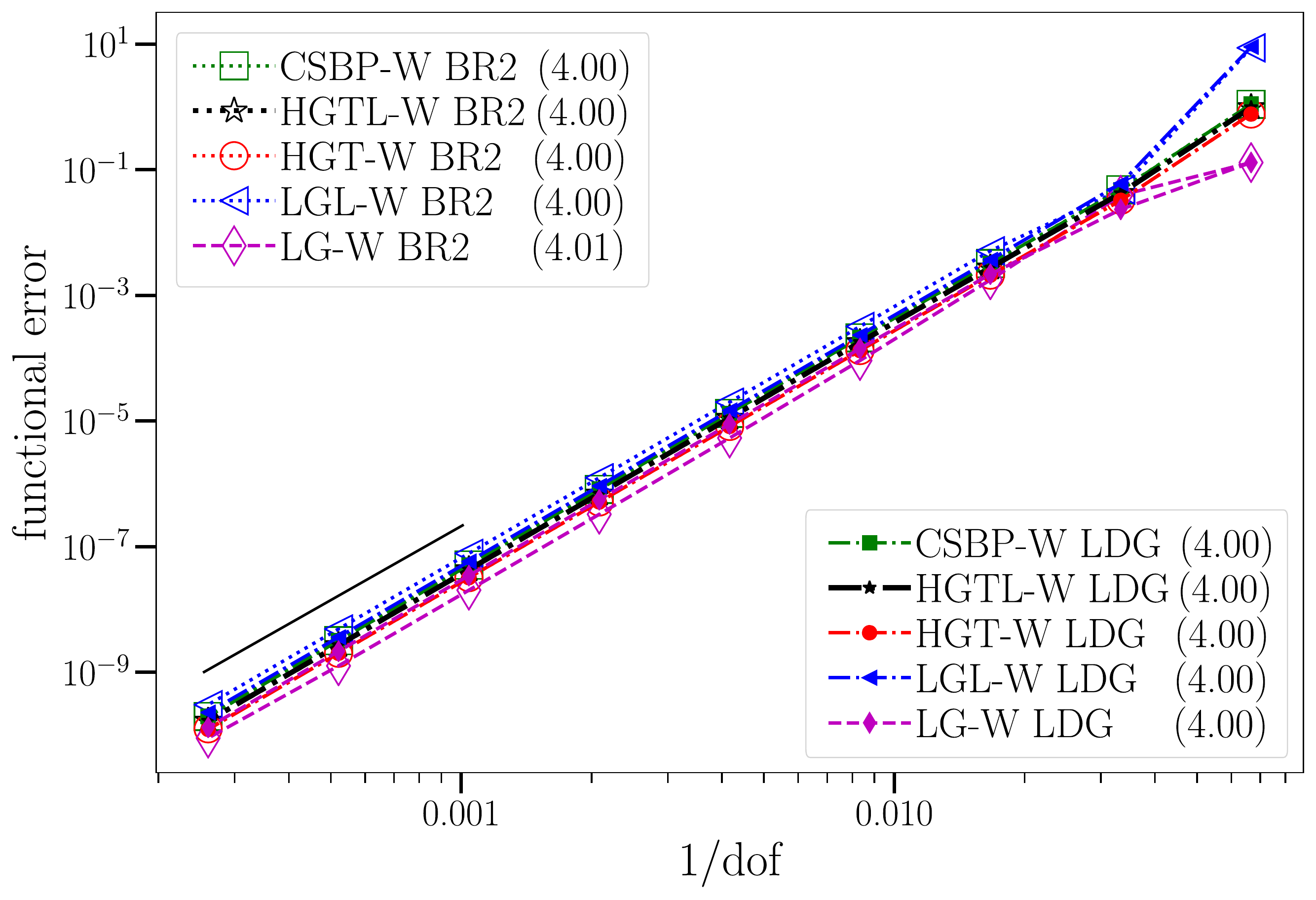}
		\label{fig:functional convergence2 p=2}}
	\\
	\subfloat[][$p=3$]{
		\includegraphics[scale=0.25]{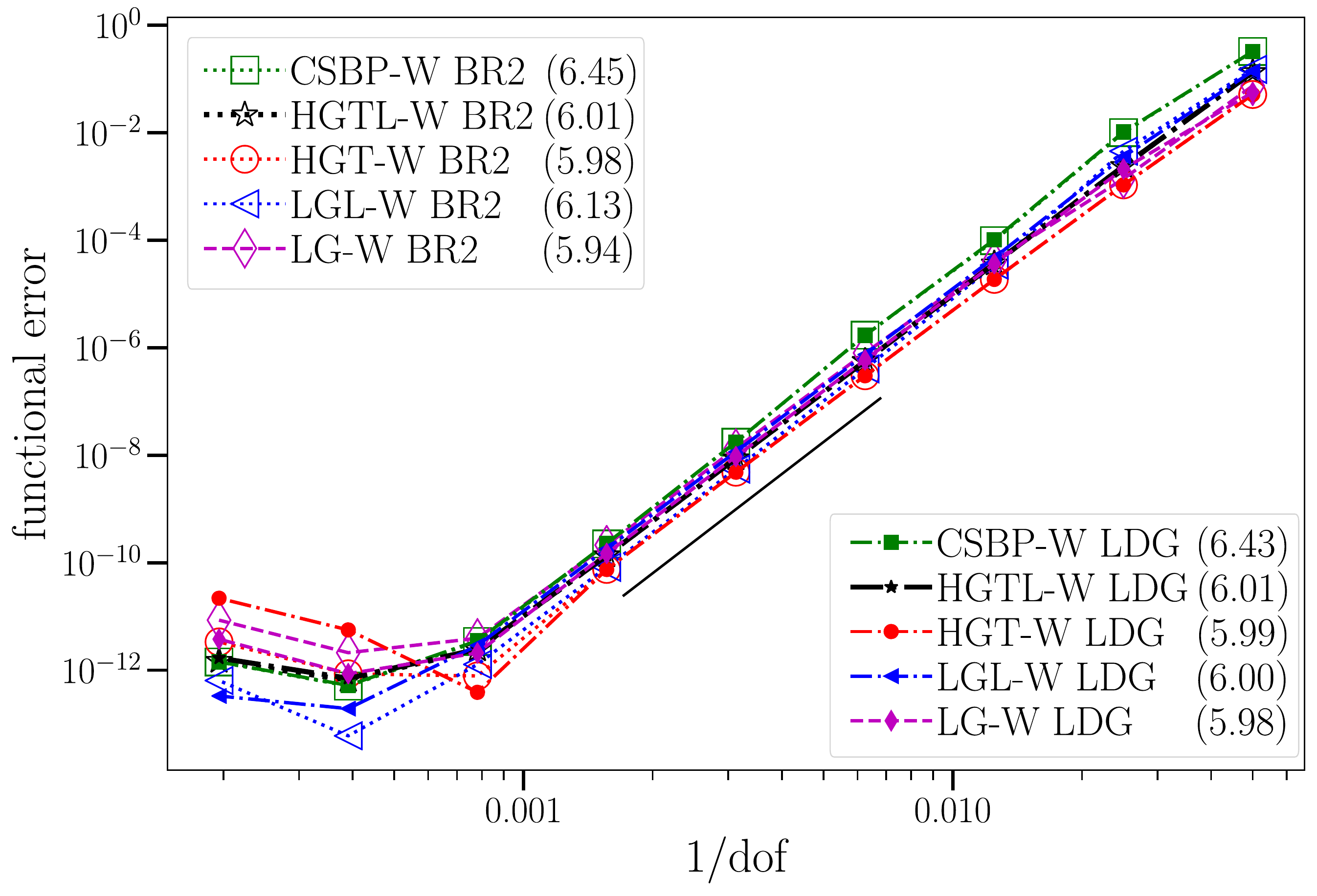}
		\label{fig:functional convergence2 p=3}}
	\hfill
	\subfloat[][$p=4$]{
		\includegraphics[scale=0.25]{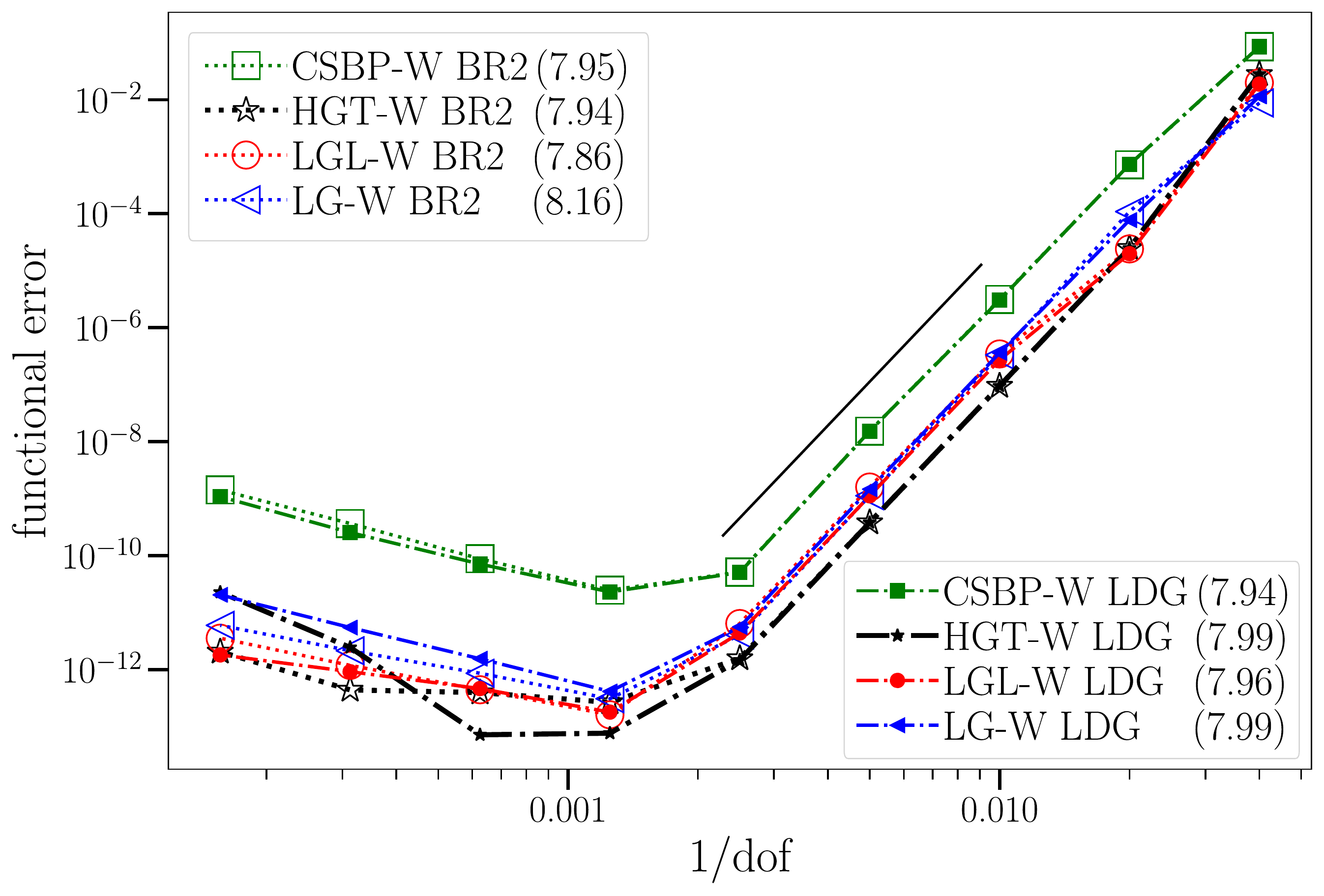}
		\label{fig:functional convergence2 p=4}}
	\caption{\label{fig:functional convergence BR2} Functional convergence under mesh refinement with adjoint consistent SATs.  The values in parentheses are the convergence rates, and ``W" stands for wide-stencil SBP operator.}
\end{figure}

\begin{figure}[!t]
	\centering
	\subfloat[][$p=1$]{
		\includegraphics[scale=0.25]{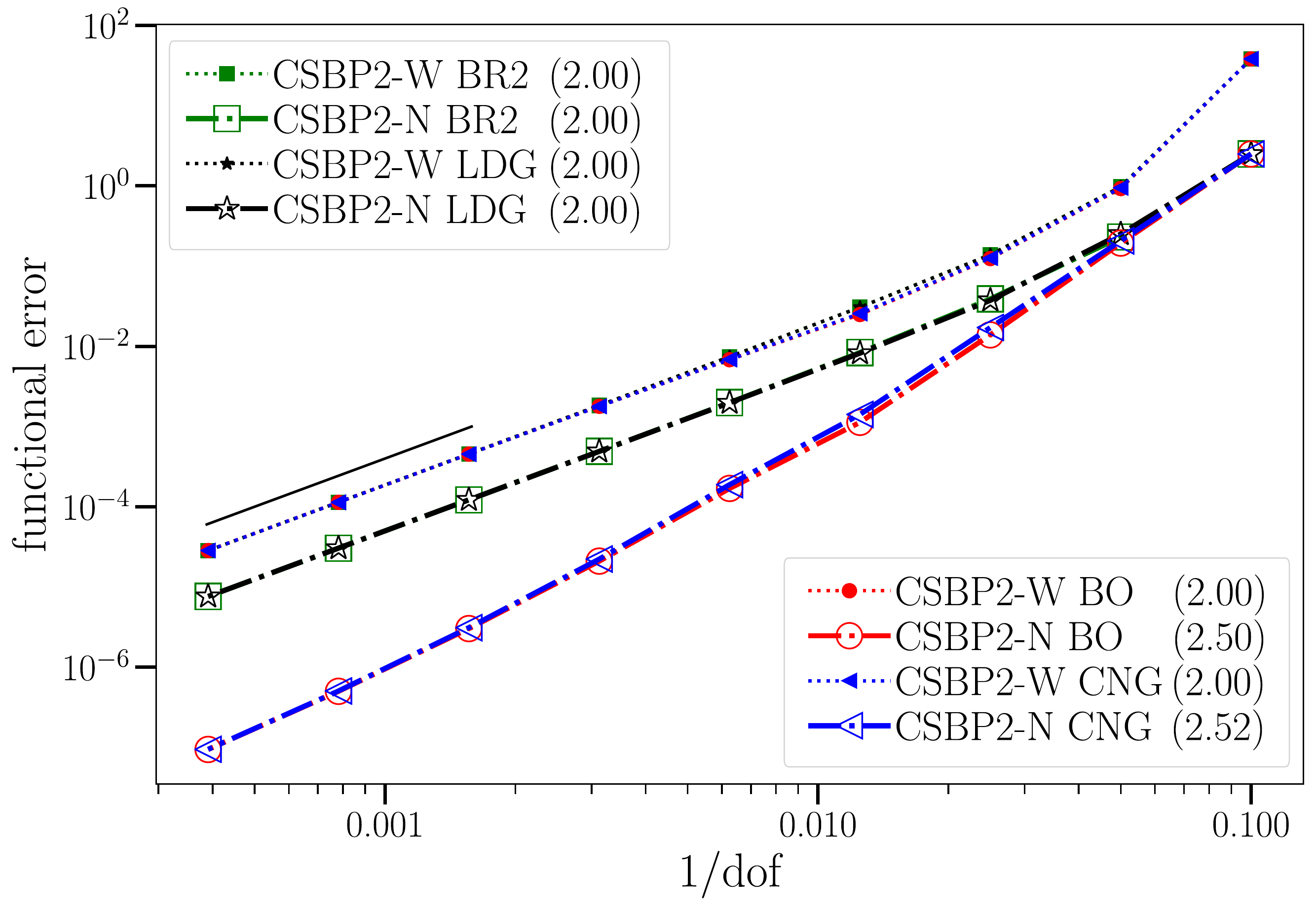}
		\label{fig:functional convergence3 p=1}}
	\hfill
	\subfloat[][$p=2$]{
		\includegraphics[scale=0.25]{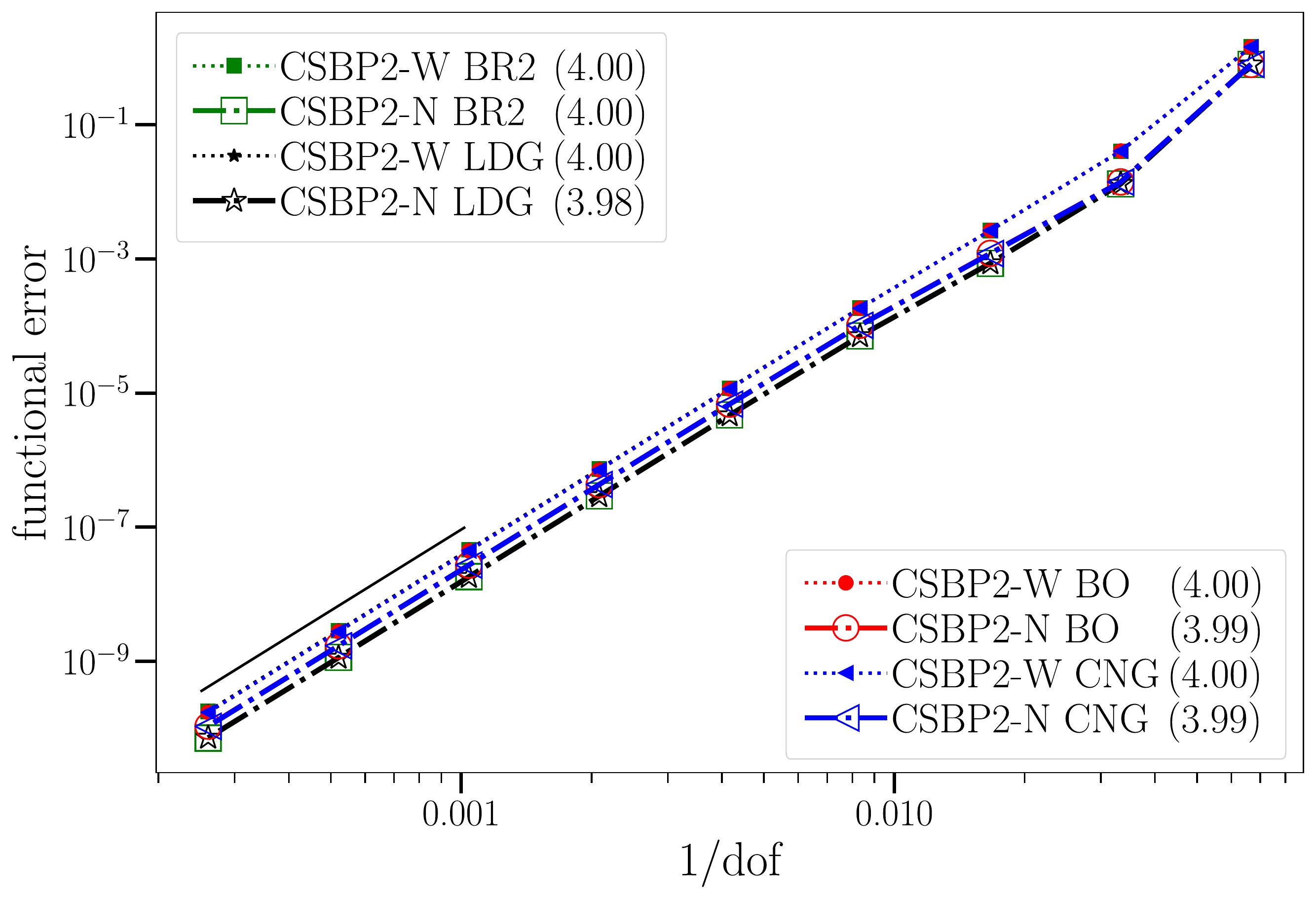}
		\label{fig:functional convergence3 p=2}}
	\\
	\subfloat[][$p=3$]{
		\includegraphics[scale=0.25]{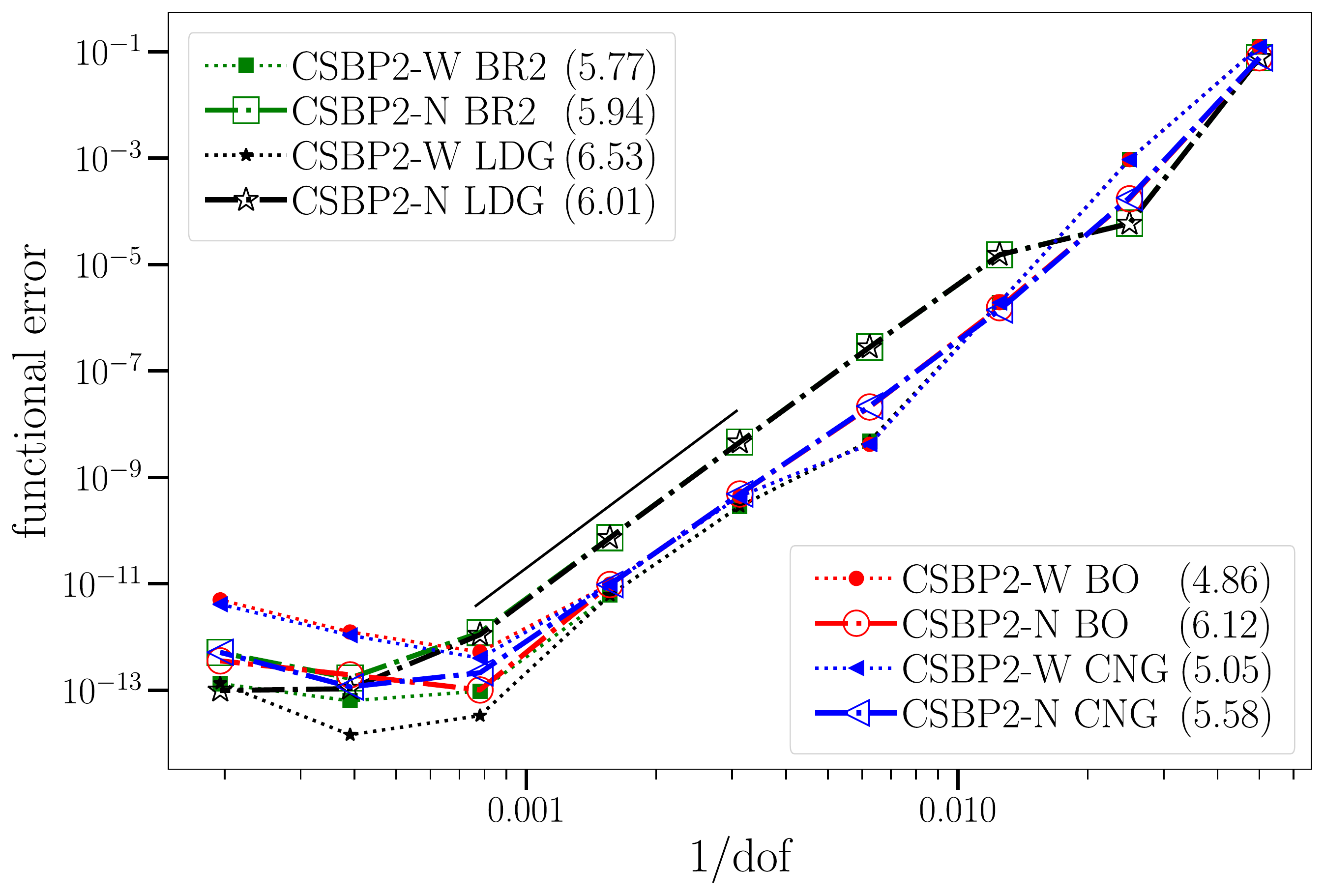}
		\label{fig:functional convergence3 p=3}}
	\hfill
	\subfloat[][$p=4$]{
		\includegraphics[scale=0.25]{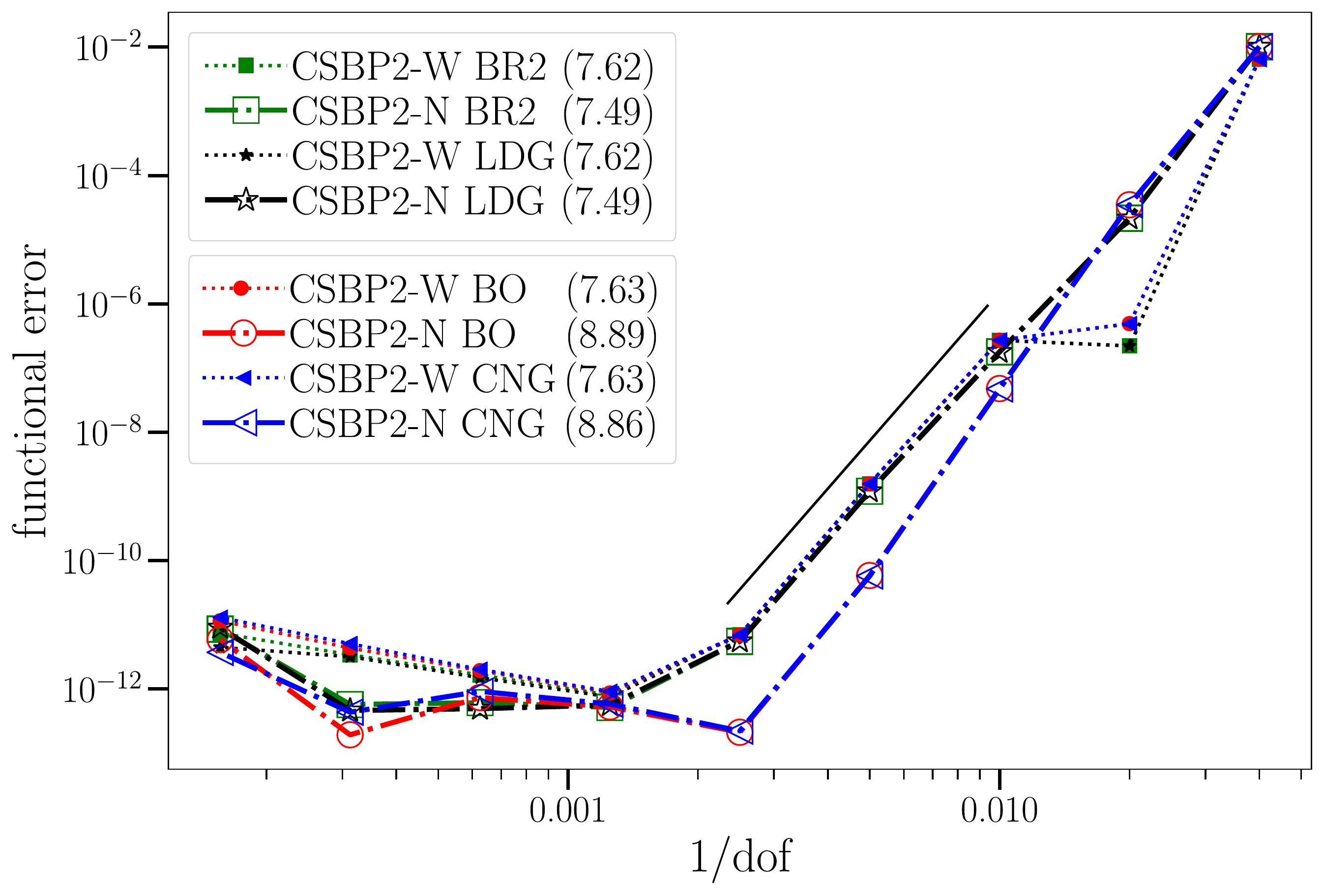}
		\label{fig:functional convergence3 p=4}}
	\caption{\label{fig:functional convergence dense} Functional convergence under mesh refinement with block-norm wide-stencil (``W") and narrow-stencil (``N") CSBP operators. The values in parentheses are the convergence rates.}
\end{figure}

\section{Conclusion}\label{sec:conclusion}

In this paper, we have shown that primal and adjoint consistent SBP-SAT discretizations of diffusion problems with diagonal-norm second-derivative generalized SBP operators lead to functional superconvergence if the primal and adjoint solutions are sufficiently smooth. For block-norm second-derivative operators, however, the analysis and the numerical experiments show that adjoint inconsistency does not degrade the functional convergence rate. We have also derived the conditions required for the stability of adjoint consistent SBP-SAT discretizations with the narrow-stencil second-derivative generalized SBP operators under the assumptions that the operators are consistent and nullspace consistent. The stability analysis also requires that the derivative operator at the element boundaries, $ \D_{b,k} $, be invertible. For most operators, $ \D_{b,k} $ is invertible or can easily be modified to be invertible. For some operators, however, this is not the case, and it might be necessary to enforce the invertibility of this matrix during the construction of the SBP operators to ensure that SBP-SAT discretizations with these operators are stable and adjoint consistent in addition to the other attractive numerical properties that narrow-stencil operators offer. 

Four different types of stable SATs for narrow-stencil SBP operators, among which two are adjoint consistent, are proposed and implemented in the numerical experiments. As predicted by the theory, the numerical experiments show that functionals superconverge at a rate of $ 2p $ when a diagonal-norm degree $ p+1 $ narrow-stencil or degree $ p $ wide-stencil generalized SBP operator is used along with adjoint consistent SATs. It is also observed that the adjoint consistent BR2 and LDG SATs yield solution convergence rates of $ p+1 $ and $ p+2 $ when implemented with diagonal-norm wide- and narrow-stencil SBP operators, respectively. Implementations with the block-norm wide- and narrow-stencil SBP operators show a solution and functional convergence rates of $ 2p $ regardless of the type of SAT used. The even-odd convergence properties of the adjoint inconsistent BO and CNG SATs are not observed when these SATs are implemented with the diagonal- and block-norm narrow-stencil SBP operators.

While the SATs presented in this work ensure the consistency, conservation, adjoint consistency, stability, and functional superconvergence of SBP-SAT discretizations with narrow-stencil generalized SBP operators, optimization of the SAT coefficients to achieve improved numerical properties, \eg, better conditioning and spectral radius, may be pursued in future work. 


\bibliographystyle{spmpsci}      
\bibliography{references}
\addcontentsline{toc}{section}{\refname}

\end{document}